\documentclass[12pt]{amsart}

\usepackage[top=3cm, left=2.8cm, right=2.8cm]{geometry}
\usepackage{amsmath}
\usepackage{amssymb}
\usepackage{amsfonts}
\usepackage{mathrsfs}
\usepackage{mathtools}
\usepackage{nicematrix}
\usepackage{cite}
\allowdisplaybreaks
\usepackage[colorlinks]{hyperref}
\hypersetup{
linkcolor=blue,
citecolor=blue,
urlcolor=blue
}

\theoremstyle{definition}
\newtheorem{definition}{Definition}[section]
\newtheorem{remark}[definition]{Remark}

\newtheorem{lemma}[definition]{Lemma}
\newtheorem{proposition}[definition]{Proposition}
\newtheorem{corollary}[definition]{Corollary}
\newtheorem{theorem}[definition]{Theorem}

\newcommand{\SL}{\operatorname{\mathbf{SL}}}
\newcommand{\GL}{\operatorname{\mathbf{GL}}}
\newcommand{\Gal}{\operatorname{Gal}}
\newcommand{\Vol}{\operatorname{Vol}}
\newcommand{\Lie}{\operatorname{Lie}}
\newcommand{\Ad}{\operatorname{Ad}}

\newcommand{\diag}{\operatorname{diag}}

\newcommand{\diam}{\operatorname{diam}}
\newcommand{\rank}{\operatorname{rank}}

\newcommand{\height}{\operatorname{ht}}

\begin{document}
\title[A shrinking target problem in homogeneous spaces]{A shrinking target problem in homogeneous spaces of semisimple algebraic groups}
\author{Cheng Zheng}
\address{School of Mathematical Sciences, Shanghai Jiao Tong University, China}
\email{zheng.c@sjtu.edu.cn}
\thanks{The author acknowledges the support by the Institute of Modern Analysis-A Frontier Research Centre of Shanghai.}
\maketitle

\begin{abstract}
In this paper, we study a shrinking target problem with target at infinity in a homogeneous space of a semisimple algebraic group from the representation-theoretic point of view. Let $\rho:\mathbf G\to\GL(V)$ be an irreducible $\mathbb Q$-rational representation of a connected semisimple $\mathbb Q$-algebraic group $\mathbf G$ on a complex vector space $V$, $\{a_t\}_{t\in\mathbb R}$ a one-parameter subgroup in a $\mathbb Q$-split torus in $\mathbf G$ and $\psi:\mathbb R_+\to\mathbb R_+$ a positive function on $\mathbb R_+$. We define a subset $S_\rho(\psi)$ of $\psi$-Diophantine elements in $\mathbf G(\mathbb R)$ in terms of the representation $\rho$ and $\{a_t\}_{t\in\mathbb R}$, and prove formulas for the Hausdorff dimension of the complement of $S_{\rho}(\psi)$. We also discuss the connections of our results to Diophantine approximation on flag varieties and rational approximation to linear subspaces in Grassmann varieties.
\end{abstract}

\section{Introduction and main results}\label{intro}
\subsection{Introduction and main problem}
In this paper, we study a shrinking target problem, the prototype of which was first proposed by Hill and Velani \cite{HV95}. Let $f:X\to X$ (or $f_t:X\to X$) be a map (or a flow) on a metric space $X$ with a measure $\mu$. Generally speaking, in the shrinking target problem, one studies the set $S$ of points in $X$ whose orbits under $f$ (or $f_t$) hit a shrinking target infinitely often, and seeks to establish results about the size (the $\mu$-measure or the Hausdorff dimension) of $S$. The question about the measure of $S$, especially when $f_t$ is a flow on a homogeneous space, is closely related to Khintchine-type theorems in the metric theory of Diophantine approximation \cite{K24,K26, G38}, and has been studied extensively in the past few decades (e.g. \cite{AM09,GK17,KO21,KY19,KM99,KZ18,S82,T08,KW19,BDGW24}). In particular, the shrinking target problem of this type for a diagonalizable homogeneous flow $f_t$ has been addressed in full generality by Kleinbock and Margulis \cite{KM99} where a null-conull law on $\mu(S)$ and a logarithm law are established. On the other hand, the study of the Hausdorff dimension of $S$ for a homogeneous flow $f_t$ (when $\mu(S)=0$) is a finer problem, and it is usually related to the Jarn\'ik-Besicovitch theorem in Diophantine approximation \cite{B34,J31}. As the analysis of the Hausdorff dimension of $S$ is more delicate and requires more information about the geometric structure of the space $X$ and the dynamics of the homogeneous flow $f_t$, only a few cases were known (which we will discuss below in this subsection). In this paper, we focus on this Hausdorff dimension version of shrinking target problem for homogeneous flows. (For discussions about other maps or flows, one may refer to e.g. \cite{FMSU15,HV97,LWWX14,HR11,SW13,U02,LLVZ23}.)

In \cite{Z19, FZ22}, we consider a shrinking target problem with target at infinity for the homogeneous flow $a_t:G/\Gamma\to G/\Gamma$ where $G/\Gamma$ is a finite-volume quotient of a rank-one simple Lie group $G$ \cite{Z19} or $G/\Gamma=\SL_3(\mathbb R)/\SL_3(\mathbb Z)$ \cite{FZ22}, and $\{a_t\}_{t\in\mathbb R}$ is a one-parameter diagonalizable subgroup in $G$. In particular, we obtain a formula for the Hausdorff dimension of the set $S$ as defined above, and establish a Jarn\'ik-Besicovitch theorem on Diophantine approximation in Heisenberg groups \cite{Z19}. One may also refer to \cite{HP,HP01,HP02,MP93} for related discussions.

This paper is a continuation of the works \cite{Z19, FZ22}, and we aim to generalize the main results in \cite{Z19, FZ22} for diagonalizable homogeneous flows on finite-volume quotients of semisimple algebraic groups. We will see that our generalization has close connections to Diophantine approximation on flag varieties~\cite{L65, d21} and rational approximation to linear subspaces in Grassmann varieties~\cite{Sch67,d25}, and it can imply certain Jarn\'ik-Besicovitch type theorems about well approximable subsets. Currently there are references \cite{BD86,D92,DFSU,S21} which may be related to the topic of this paper with different emphases. In \cite{BD86, D92}, the main results are Jarn\'ik-Besicovitch theorems about Diophantine matrices, but with the help of a generalized Dani correspondence developed in \cite{D85,KM99}, these results can be reformulated as a shrinking target problem for the homogeneous flow $h_t:X_{m+n}\to X_{m+n}$ where $X_{m+n}=\SL_{m+n}(\mathbb R)/\SL_{m+n}(\mathbb Z)$ and $$h_t=\diag(\underbrace{e^{t/m},\dots,e^{t/m}}_{m\textup{ times}}, \underbrace{e^{-t/n},\dots,e^{-t/n}}_{n\textup{ times}})\quad(t\in\mathbb R).$$ In \cite{DFSU}, a variational principle is established to analyze $h_t$-orbits with various behaviors in $X_{m+n}$, which makes it possible to study the Hausdorff dimensions of a variety of Diophantine subsets in the space of $m\times n$ matrices. This variational principle is generalized in \cite{S21} for any diagonalizable flow on $X_{m+n}$ with respect to a nonstandard Hausdorff dimension. We remark that the results in \cite{BD86,D92,DFSU,S21} mainly deal with diagonalizable flows on $X_{m+n}$.

Now we propose the main problem of this paper. Let $\mathbf G$ be a semisimple algebraic group defined over $\mathbb Q$ and $\rho:\mathbf G\to\GL(V)$ a finite-dimensional irreducible representation of $\mathbf G$ over $\mathbb Q$ on a complex vector space $V$ with a $\mathbb Q$-structure. The vector space $V$ may be identified with $\mathbb C^d$ ($d=\dim_{\mathbb C} V$) equipped with a norm $\|\cdot\|$ so that $\mathbb Z^d\subset\mathbb C^d$ is compatible with the $\mathbb Q$-structure in $V$. For any discrete subgroup $\Lambda$ in $V$, define the first minimum of $\Lambda$ by $$\delta(\Lambda)=\inf_{v\in\Lambda\setminus\{0\}}\|v\|.$$ Let $\{a_t\}_{t\in\mathbb R}$ be a one-parameter Ad-diagonalizable subgroup in $\mathbf G(\mathbb R)$, $\mathbb R_+$ the set of positive real numbers and $\psi:\mathbb R_+\to\mathbb R_+$ a positive function on $\mathbb R_+$. Then we aim to estimate the Hausdorff dimension of the complement of the set $$\{g\in\mathbf G(\mathbb R): \delta(\rho(a_t\cdot g)\cdot\mathbb Z^d)\geq C\cdot\psi(t)\;(\forall t>0)\textup{ for some }C>0\}\subset\mathbf G(\mathbb R)$$ with respect to a standard Riemannian metric on $\mathbf G(\mathbb R)$.

By Mahler's compactness criterion, the main problem above is clearly  a shrinking target problem with target at infinity for the homogeneous flow $a_t:\mathbf G(\mathbb R)/\Gamma\to\mathbf G(\mathbb R)/\Gamma$, where $\Gamma$ is an arithmetic lattice in $\mathbf G(\mathbb R)$ preserving $\mathbb Z^d$. Note that the results in \cite{BD86, D92} can be translated in the setting of this problem if we let $\mathbf G=\SL_{m+n}$, $\rho$ the standard representation of $\SL_{m+n}$ on $\mathbb C^{m+n}$ and $\{a_t\}_{t\in\mathbb R}=\{h_t\}_{t\in\mathbb R}$. We will see later in our theorems that the main problem also includes the cases in \cite{Z19, FZ22}.

\subsection{Main results}\label{results}
In this paper, we address the main problem above under the following assumption. We assume that $\mathbf G$ is a connected semisimple algebraic group over $\mathbb Q$ (with respect to the Zariski topology) and $\mathbf T$ is a maximal $\mathbb Q$-split torus in $\mathbf G$. Let $\{a_t\}_{t\in\mathbb R}$ be a one-parameter subgroup in $\mathbf T(\mathbb R)$, $\psi:\mathbb R_+\to\mathbb R_+$ a positive function on $\mathbb R_+$ and $\rho$ a finite-dimensional irreducible representation of $\mathbf G$ defined over $\mathbb Q$ with $\dim\ker\rho=0$.

\begin{definition}\label{def11}
An element $g\in\mathbf G(\mathbb R)$ is called $\psi$-Diophantine if there exists a constant $C>0$ such that $$\delta(\rho(a_t\cdot g)\mathbb Z^d)\geq C\cdot\psi(t)\textup{ for any }t>0$$ where $\delta$ is the first minimum function. We denote by $S_\rho(\psi)$ the set of all $\psi$-Diophantine elements, and $S_\rho(\psi)^c$ its complement in $\mathbf G(\mathbb R)$.
\end{definition}
\begin{remark}\label{r12}
In the rest of the paper, we assume that $\mathbf G$ is $\mathbb Q$-isotropic since otherwise $\mathbf T=\{e\}$. Note that for any $g\in\mathbf G(\mathbb R)$, $\rho(a_t\cdot g)\cdot\mathbb Z^d$ is a unimodular lattice in the vector space of real points in $V$, and $\delta(\rho(a_t\cdot g)\cdot\mathbb Z^d)$ is bounded above by a constant depending only on $\rho$. Due to this fact, we may also assume that the function $\psi$ is bounded since otherwise $S_\rho(\psi)=\emptyset$ by Definition~\ref{def11}. We denote by $\dim_H(S)$ the Hausdorff dimension of a subset $S$ in a smooth manifold $\mathcal M$ with respect to a Riemannian metric on $\mathcal M$.
\end{remark}

To state the first two theorems about $\dim_HS_\rho(\psi)^c$, we need to introduce some notation. We choose a minimal parabolic $\mathbb Q$-subgroup $\mathbf P_0$ in $\mathbf G$ containing $\mathbf T$ with the Levi subgroup $Z(\mathbf T)$ (the centralizer of $\mathbf T$ in $\mathbf G$). Then $\mathbf P_0$ and $\mathbf T$ define a root system ($\Phi,\;\Phi^+,\;\Delta$) where $\Phi$ is the set of $\mathbb Q$-roots relative to $\mathbf T$, $\Phi^+$ is the set of positive $\mathbb Q$-roots determined by $\mathbf P_0$ and $\Delta$ is the set of simple $\mathbb Q$-roots in $\Phi^+$. Let $\overline{\mathbf P}_0$ be the opposite minimal parabolic $\mathbb Q$-subgroup of $\mathbf P_0$ defined by $\Phi\setminus\Phi^+$ with the same Levi subgroup $Z(\mathbf T)$. Without loss of generality, we may assume that the stable horospherical subgroup of $\{a_t\}_{t\in\mathbb R}$ is contained in the unipotent radical $R_u(\mathbf P_0)$ of $\mathbf P_0$ and the unstable horospherical subgroup of $\{a_t\}_{t\in\mathbb R}$ is contained in the unipotent radical $R_u(\overline{\mathbf P}_0)$ of $\overline{\mathbf P}_0$. One can write the space $V$ in the representation $\rho$ as a direct sum of weight spaces with respect to the action of $\mathbf T$ $$V=\bigoplus_{\beta}V_\beta.$$ By the structure theory of irreducible representations of complex semisimple groups and semisimple Lie algebras, there is a highest weight $\beta_0$ among the weights $\beta$'s (where the order is determined by $\Phi^+$), and we denote by $V_{\beta_0}$ its corresponding weight space. (See \S\ref{pre} for more details). The stabilizer of the weight space $V_{\beta_0}$ in $\mathbf G$ is a parabolic $\mathbb Q$-subgroup $\mathbf P_{\beta_0}$ containing $\mathbf P_0$, and its unipotent radical is denoted by $R_u(\mathbf P_{\beta_0})$. It is known that there exists a $\mathbb Q$-torus $\mathbf T_{\beta_0}$ in $\mathbf T$ such that $Z(\mathbf T_{\beta_0})$ (the centralizer of $\mathbf T_{\beta_0}$ in $\mathbf G$) is a Levi subgroup of $\mathbf P_{\beta_0}$. We denote by $\overline{\mathbf P}_{\beta_0}$ the opposite parabolic $\mathbb Q$-subgroup of $\mathbf P_{\beta_0}$ containing $\overline{\mathbf P}_0$ with the same Levi subgroup $Z(\mathbf T_{\beta_0})$. The unipotent radical of $\overline{\mathbf P}_{\beta_0}$ is denoted by $R_u(\overline{\mathbf P}_{\beta_0})$.

In the following, if an algebraic $\mathbb Q$-subgroup $\mathbf F\subset\mathbf G$ is normalized by $\mathbf T$, then we write $\Phi(\mathbf F)$ for the set of $\mathbb Q$-roots in $\mathbf F$ relative to $\mathbf T$, while the symbol $\sum_{\alpha\in\Phi(\mathbf F)}$ (or $\prod_{\alpha\in\Phi(\mathbf F)}$) stands for the sum (or product) over all the $\mathbb Q$-roots $\alpha\in\Phi(\mathbf F)$ counted with multiplicities (i.e., the dimensions of the corresponding $\mathbb Q$-root spaces associated to $\alpha\in\Phi(\mathbf F)$ in the Lie algebra of $\mathbf F$). For a $\mathbb Q$-root $\lambda$ in $\mathbf G$ or a $\mathbb Q$-weight $\lambda$ in $\rho$ relative to $\mathbf T$ (which is a $\mathbb Q$-character of $\mathbf T$), we will often consider it as a linear functional on the Lie algebra $\mathfrak a$ of $\mathbf T(\mathbb R)$ and use the same symbol. In particular, we will write $\lambda(a_t)$ ($t\in\mathbb R$) for the values of $\lambda$ (as a linear functional) on the Lie algebra of $\{a_t\}_{t\in\mathbb R}$ (so that $\lambda(a_t)$ is linear on the parameter $t\in\mathbb R$). We denote by $\nu_{0}$ any $\mathbb Q$-root relative to $\mathbf T$ in $R_u(\overline{\mathbf P}_{0})$ satisfying the condition that $$\nu_{0}(a_1)=\max\{\alpha(a_1):\alpha\in\Phi(R_u(\overline{\mathbf P}_{0}))\}.$$ Since the unstable horospherical subgroup of $\{a_t\}_{t\in\mathbb R}$ is contained in $R_u(\overline{\mathbf P}_0)$, the value $\nu_0(a_1)>0$.

Now we can state our first two theorems. Here for any $\psi:\mathbb R_+\to\mathbb R_+$, we define $$\tau(\psi):=\liminf_{t\to\infty}\left(-\frac{\ln(\psi(t))}t\right)$$ which will be an important quantity in the formula of the Hausdorff dimension of $S_\rho(\psi)^c$. We prove that the effective range of $\tau(\psi)$ is $0\leq\tau(\psi)\leq\beta_0(a_{-1})$.
\begin{theorem}\label{range}
We have $\tau(\psi)\geq0$ and $\beta_0(a_{-1})>0$. If $\tau(\psi)>\beta_0(a_{-1})$, then $S_\rho(\psi)^c=\emptyset$.
\end{theorem}

\begin{theorem}\label{mthm11}
Let $\mathbf G$ be a connected semisimple algebraic group defined over $\mathbb Q$, $\mathbf T$ a maximal $\mathbb Q$-split torus in $\mathbf G$ and $\{a_t\}_{t\in\mathbb R}$ a one-parameter subgroup in $\mathbf T(\mathbb R)$. Let $\rho$ be a finite-dimensional irreducible representation of $\mathbf G$ defined over $\mathbb Q$ on a complex vector space $V$ with $\dim\ker\rho=0$, and $\psi:\mathbb R_+\to\mathbb R_+$ a function on $\mathbb R_+$. Then 
\begin{align*}
\dim_H S_\rho(\psi)^c\geq&\dim\mathbf G-\frac{\tau(\psi)}{\beta_0(a_{-1})\nu_{0}(a_1)}\cdot\sum_{\alpha\in\Phi(R_u(\overline{\mathbf P}_{\beta_0}))}\alpha(a_1)
\end{align*}
for any $0\leq\tau(\psi)<\beta_0(a_{-1})$. In the case $\tau(\psi)=\beta_0(a_{-1})$ we have
\begin{enumerate}
\item If $\psi(t)\cdot e^{\beta_0(a_{-1})t}$ is bounded, then $S_\rho(\psi)^c=\emptyset$.
\item If $\psi(t)\cdot e^{\beta_0(a_{-1})t}$ is unbounded, then $S_\rho(\psi)^c\neq\emptyset$ and $$\dim_HS_\rho(\psi)^c\geq\dim\mathbf G-\frac{\tau(\psi)}{\beta_0(a_{-1})\nu_{0}(a_1)}\cdot\sum_{\alpha\in\Phi(R_u(\overline{\mathbf P}_{\beta_0}))}\alpha(a_1).$$
\end{enumerate}
\end{theorem}

To state the next theorem, we need to introduce another notation. Let $N(\mathbf T)$ and $Z(\mathbf T)$ be the normalizer and centralizer of $\mathbf T$ in $\mathbf G$ respectively. Then the Weyl group relative to $\mathbb Q$ is defined by $$_{\mathbb Q}{W}=N(\mathbf T)/Z(\mathbf T).$$ Let $\mathbf P_0$ be the minimal parabolic $\mathbb Q$-subgroup of $\mathbf G$ defined as in Theorem~\ref{mthm11} such that the stable horospherical subgroup of $\{a_t\}_{t\in\mathbb R}$ is contained in the unipotent radical $R_u(\mathbf P_0)$ of $\mathbf P_0$ and the unstable horospherical subgroup of $\{a_t\}_{t\in\mathbb R}$ is contained in the unipotent radical $R_u(\overline{\mathbf P}_0)$ of $\overline{\mathbf P}_0$. It is known that the Bruhat decomposition of $\mathbf G$ is the following \cite[Theorem 21.15]{B91} $$\mathbf G(\mathbb Q)=\mathbf P_0(\mathbb Q)\cdot {_{\mathbb Q}{W}}\cdot\mathbf P_0(\mathbb Q).$$  Let ${_{\mathbb Q}\mathcal W}$ be a set of representatives of $_{\mathbb Q}{W}$ in $N(\mathbf T)(\mathbb Q)$. We may assume that $e\in{_{\mathbb Q}\mathcal W}$. Then one can deduce from the Bruhat decomposition that (see \S\ref{pre}) $$\mathbf G(\mathbb Q)=\mathbf P_0(\mathbb Q)\cdot{_{\mathbb Q}{\mathcal W}}\cdot\overline{\mathbf P}_0(\mathbb Q),\quad\mathbf G(\mathbb Q)=\mathbf P_0(\mathbb Q)\cdot {_{\mathbb Q}{\mathcal W}}\cdot R_u(\overline{\mathbf P}_0)(\mathbb Q).$$ For each $w\in{_{\mathbb Q}\mathcal W}$, define $$\mathbf F_{w}=R_u(\overline{\mathbf P}_0)\cap w^{-1}R_u(\overline{\mathbf P}_{\beta_0})w,\quad\mathbf H_{w}=R_u(\overline{\mathbf P}_0)\cap w^{-1}\mathbf P_{\beta_0}w$$ and by \cite[Proposition 21.9]{B91}, we have $$R_u(\overline{\mathbf P}_0)=\mathbf H_{w}\cdot\mathbf F_{w}.$$ It follows that 
\begin{align*}
\mathbf G(\mathbb Q)=&\mathbf P_{\beta_0}(\mathbb Q)\cdot {_{\mathbb Q}{\mathcal W}}\cdot R_u(\overline{\mathbf P}_0)(\mathbb Q)\\
=&\bigcup_{w\in{_{\mathbb Q}\mathcal W}} \mathbf P_{\beta_0}(\mathbb Q)\cdot(w\mathbf H_{w}w^{-1})(\mathbb Q)\cdot w\cdot\mathbf F_{w}(\mathbb Q)\\
=&\bigcup_{w\in{_{\mathbb Q}\mathcal W}}\mathbf P_{\beta_0}(\mathbb Q)\cdot w\cdot\mathbf F_{w}(\mathbb Q).
\end{align*}
Note that the subsets $\mathbf P_{\beta_0}(\mathbb Q)\cdot w\cdot\mathbf F_{w}(\mathbb Q)$ $(w\in{_{\mathbb Q}\mathcal W})$ above may overlap, and for our purpose, we choose a subset ${_{\mathbb Q}\overline{\mathcal W}}$ of ${_{\mathbb Q}\mathcal W}$ (as small as possible) such that $$\mathbf G(\mathbb Q)=\bigcup_{w\in{_{\mathbb Q}\overline{\mathcal W}}}\mathbf P_{\beta_0}(\mathbb Q)\cdot w\cdot\mathbf F_{w}(\mathbb Q).$$ In the statement of Theorem~\ref{mthm12} below, we fix any such subset ${_{\mathbb Q}\overline{\mathcal W}}$ of ${_{\mathbb Q}\mathcal W}$.

Let $\bigwedge^{\dim V_{\beta_0}}V$ be the $\dim V_{\beta_0}$-exterior product space of $V$ over $\mathbb C$, and denote by $\rho_{\beta_0}$ the natural extension of $\rho$ on $\bigwedge^{\dim V_{\beta_0}}V$. Let $$\{e_1,e_2,\dots, e_{\dim V_{\beta_0}}\}\subset\mathbb Z^d$$ be an integral basis in $V_{\beta_0}$ which spans $V_{\beta_0}\cap\mathbb Z^d$. We write $$e_{V_{\beta_0}}:=e_1\wedge e_2\wedge\cdots\wedge e_{\dim V_{\beta_0}}\in\bigwedge^{\dim V_{\beta_0}}V.$$ For each $w\in{_{\mathbb Q}\overline{\mathcal W}}$, define the following morphism $$\Psi_{w}:\mathbf F_{w}(\mathbb R)\to\bigwedge^{\dim V_{\beta_0}} V,\quad\Psi_{w}(x)=\rho_{\beta_0}(wxw^{-1})\cdot e_{V_{\beta_0}}.$$ Note that $w\mathbf F_{w}w^{-1}\subset R_u(\overline{\mathbf P}_{\beta_0})$ and the stabilizer of the weight space $V_{\beta_0}$ is $\mathbf P_{\beta_0}$. So $\Psi_{w}$ is an isomorphism onto its image. For any $R>0$, let $$B_R=\left\{v\in\bigwedge^{\dim V_{\beta_0}} V: \|v\|\leq R\right\}$$ where $\|\cdot\|$ is a Euclidean norm on $\bigwedge^{\dim V_{\beta_0}}V$. We denote by $a_{w}$ the growth rate of the asymptotic volume estimate of the real variety $\textup{Im}(\Psi_{w})$ $$a_w=\lim_{R\to\infty}\frac{\log\mu_{\mathbf F_w}(\Psi_w^{-1}(B_R))}{\log R},$$ where $\mu_{\mathbf F_w}$ is the Haar measure on $\mathbf F_w(\mathbb R)$, and define by $$A_w=A_w(\Omega_{\beta_0}):=\limsup_{H\to\infty}\frac{\log|\{g\in\Omega_{\beta_0}\cap w\mathbf F_{w}w^{-1}(\mathbb Q): \height(g)\leq H\}|}{\log H}$$  the growth rate of the number of rational elements in $\Omega_{\beta_0}\cap w\mathbf F_{w}w^{-1}(\mathbb Q)$ where $\Omega_{\beta_0}$ is a fixed open bounded subset in $R_u(\overline{\mathbf P}_{\beta_0})(\mathbb R)$ containing a fundamental domain of $R_u(\overline{\mathbf P}_{\beta_0})(\mathbb R)/(R_u(\overline{\mathbf P}_{\beta_0})(\mathbb R)\cap\Gamma)$
(See \S\ref{ubd}). The definition of the height function $\height(\cdot)$ is given in \S\ref{pre} (Definition~\ref{d25}).  As before, let $\nu_0$ be any $\mathbb Q$-root relative to $\mathbf T$ in $R_u(\overline{\mathbf P}_0)$ such that $$\nu_0(a_1)=\max\{\alpha(a_1):\alpha\in\Phi(R_u(\overline{\mathbf P}_{0}))\}.$$

\begin{theorem}\label{mthm12}
Let $\mathbf G$ be a connected semisimple algebraic group defined over $\mathbb Q$, $\mathbf T$ a maximal $\mathbb Q$-split torus in $\mathbf G$ and $\{a_t\}_{t\in\mathbb R}$ a one-parameter subgroup in $\mathbf T(\mathbb R)$. Let $\rho$ be a finite-dimensional irreducible representation of $\mathbf G$ defined over $\mathbb Q$ on a complex vector space $V$ with $\dim\ker\rho=0$, and $\psi:\mathbb R_+\to\mathbb R_+$ a function on $\mathbb R_+$. Then 
\begin{align*}
&\dim_HS_\rho(\psi)^c\\
\leq&\max_{\substack{w\in{_{\mathbb Q}\overline{\mathcal W}}\\ \beta_0(wa_{-1}w^{-1})\\\geq\tau(\psi)}}\left\{\dim\mathbf G-\sum_{\alpha\in\Phi(\mathbf F_{w})}\frac{\alpha(a_1)}{\nu_{0}(a_1)}+\frac{(\beta_0(wa_{-1}w^{-1})-\tau(\psi))}{\nu_{0}(a_1)}\cdot\max\{a_{w},A_{w}\}\cdot\dim V_{\beta_0}\right\}
\end{align*}
for any $0\leq\tau(\psi)\leq\beta_0(a_{-1})$.
\end{theorem}

Combining Theorem~\ref{mthm11}, Theorem~\ref{mthm12} and \cite[Theorem 4]{MG14}, we obtain the following theorem.

\begin{theorem}\label{thm10}
Let $\mathbf G$ be a connected $\mathbb Q$-simple group, $\mathbf T$ a maximal $\mathbb Q$-split torus in $\mathbf G$ and $\{a_t\}_{t\in\mathbb R}$ a one-parameter subgroup in $\mathbf T(\mathbb R)$. Let $\rho$ be a finite-dimensional irreducible representation of $\mathbf G$ defined over $\mathbb Q$ on a complex vector space $V$ with $\dim\ker\rho=0$, and $\psi:\mathbb R_+\to\mathbb R_+$ a function on $\mathbb R_+$. Let $\beta_0$ and $\nu_0$ be defined as in Theorems~\ref{mthm11} and~\ref{mthm12}. Suppose that the highest weight $\beta_0$ is a multiple of $\sum_{\alpha\in\Phi(R_u(\overline{\mathbf P}_{\beta_0}))}{\alpha}$. Then 
\begin{align*}
\dim_H S_\rho(\psi)^c=&\dim\mathbf G-\frac{\tau(\psi)}{\beta_0(a_{-1})\nu_{0}(a_1)}\cdot\sum_{\alpha\in\Phi(R_u(\overline{\mathbf P}_{\beta_0}))}\alpha(a_1)
\end{align*}
for any $0\leq\tau(\psi)<\beta_0(a_{-1})$. In the case $\tau(\psi)=\beta_0(a_{-1})$, we have
\begin{enumerate}
\item If $\psi(t)\cdot e^{\beta_0(a_{-1})t}$ is bounded, then $S_\rho(\psi)^c=\emptyset$.
\item If $\psi(t)\cdot e^{\beta_0(a_{-1})t}$ is unbounded, then $S_\rho(\psi)^c\neq\emptyset$ and $$\dim_HS_\rho(\psi)^c=\dim\mathbf G-\frac{\tau(\psi)}{\beta_0(a_{-1})\nu_{0}(a_1)}\cdot\sum_{\alpha\in\Phi(R_u(\overline{\mathbf P}_{\beta_0}))}\alpha(a_1).$$
\end{enumerate}
\end{theorem}

Any irreducible $\mathbb Q$-rational representation of a $\mathbb Q$-rank one $\mathbb Q$-simple group satisfies the conditions of Theorem~\ref{thm10}. In this case, Theorem~\ref{thm10} recovers and largely extends \cite[Theorem 1.1]{Z19} in the arithmetic setting. We will discuss the relation between Theorem~\ref{thm10} and \cite[Theorem 1.1]{Z19} in \S\ref{cor}. We remark that in \cite{Z19}, we also deal with the non-arithmetic case due to the fact that there are non-arithmetic lattices in $\mathbb R$-rank one simple groups.

Besides $\mathbb Q$-rank one $\mathbb Q$-simple groups, in the following we also list some other examples of Theorem~\ref{thm10}. The first example concerns the standard representation of $\SL_n$ and any one-parameter diagonal subgroup $\{a_t\}_{t\in\mathbb R}$ in $\SL_n(\mathbb R)$. It can be considered as a generalization of \cite{BD86, D92} from the viewpoint of shrinking target problem. 

\begin{theorem}\label{thm11}
Let $\rho:\SL_n\to\GL(V)$ be the standard representation of $\SL_n$ on $V=\mathbb C^n$ defined via matrix multiplication $$\rho(g)\cdot v=g\cdot v\quad(g\in\SL_n, v\in V).$$ Let $\{a_t\}_{t\in\mathbb R}$ be a one-parameter diagonal subgroup in $\SL_n(\mathbb R)$, $\beta_0$ the highest weight of $\rho$ with respect to $\{a_t\}_{t\in\mathbb R}$ and $\nu_0$ the $\mathbb Q$-root in $\SL_n$ defined as in Theorem~\ref{thm10}. Let $\psi:\mathbb R_+\to\mathbb R_+$ be a function on $\mathbb R_+$. Then we have $$\dim_HS_\rho(\psi)^c=\dim\SL_n-\frac{n\cdot\tau(\psi)}{\nu_0(a_1)}$$
for any $0\leq\tau(\psi)<\beta_0(a_{-1})$. In the case $\tau(\psi)=\beta_0(a_{-1})$, we have
\begin{enumerate}
\item If $\psi(t)\cdot e^{\beta_0(a_{-1})t}$ is bounded, then $S_\rho(\psi)^c=\emptyset$.
\item If $\psi(t)\cdot e^{\beta_0(a_{-1})t}$ is unbounded, then $S_\rho(\psi)^c\neq\emptyset$ and $$\dim_HS_\rho(\psi)^c=\dim\SL_n-\frac{n\cdot\tau(\psi)}{\nu_{0}(a_1)}.$$
\end{enumerate}
\end{theorem}

The second example deals with the adjoint representation of $\SL_n$. We will prove in \S\ref{cor} that it generalizes the main result of \cite{FZ22}.
\begin{theorem}\label{thm12}
Let $\rho:\SL_n\to\GL(V)$ be the adjoint representation of $\SL_n$ where $V=\mathfrak{sl}_n$ is the Lie algebra of $\SL_n$. Let $\{a_t\}_{t\in\mathbb R}$ be a one-parameter diagonal subgroup in $\SL_n(\mathbb R)$, $\nu_0$ the $\mathbb Q$-root in $\SL_n$ defined as in Theorem~\ref{thm10} and $\psi:\mathbb R_+\to\mathbb R_+$ a function on $\mathbb R_+$. Then we have $$\dim_HS_\rho(\psi)^c=\dim\SL_n-\frac{(n-1)\cdot\tau(\psi)}{\nu_0(a_1)}$$ for any $0\leq\tau(\psi)<\nu_0(a_1)$. In the case $\tau(\psi)=\nu_0(a_{1})$, we have
\begin{enumerate}
\item If $\psi(t)\cdot e^{\nu_0(a_{1})t}$ is bounded, then $S_\rho(\psi)^c=\emptyset$.
\item If $\psi(t)\cdot e^{\nu_0(a_{1})t}$ is unbounded, then $S_\rho(\psi)^c\neq\emptyset$ and $$\dim_HS_\rho(\psi)^c=\dim\SL_n-\frac{(n-1)\cdot\tau(\psi)}{\nu_{0}(a_1)}.$$
\end{enumerate}
\end{theorem}

\begin{remark}\label{r19}
One can see from the proofs that all the theorems stated above also hold for $\dim_H(S_\rho(\psi)^c\cap U)$ where $U$ is any open bounded subset in $\mathbf G(\mathbb R)$. (See Remark~\ref{r56}.)
\end{remark}

\subsection{Connections to Diophantine approximation}
In this subsection, we discuss two connections of Theorem~\ref{thm10} to the metric theory of Diophantine approximation on certain algebraic varieties, and present some related Jarn\'ik-Besicovitch type theorems.

\subsubsection{Connection to Diophantine approximation on flag varieties}\label{flag}
Diophantine approximation on group varieties is an open problem raised by Lang \cite{L65}. In recent years, many important results have been established about Diophantine approximation on various group varieties including spheres, quadrics and algebraic groups \cite{GGN14,GGN15,FKMS22,GGN22,AG22, d22, Kd18, D05, KM15, GGN18, GGN13}. Here we discuss a connection between Theorem~\ref{thm10} and Diophantine approximation on generalized flag varieties, and derive an analogue of Jarn\'ik-Besicovitch theorem. We mainly follow the exposition in \cite{d21} by de Saxc\'e. For more details about Diophantine approximation on flag varieties, one may refer to \cite{d21}.

Let $\mathbf X=\mathbf P\backslash\mathbf G$ be a flag variety where $\mathbf G$ is a connected $\mathbb Q$-simple group and $\mathbf P$ a parabolic $\mathbb Q$-subgroup of $\mathbf G$. Let $V$ be a finite dimensional vector space defined over $\mathbb Q$, $\|\cdot\|$ a Euclidean norm on $V$ and $(e_i)_{1\leq i\leq d}$ a rational orthonormal basis of $V$. For any rational point $v$ in $\mathbf P(V)(\mathbb Q)$, we define the height $H(v)$ of $v$ by $$H(v)=\|\mathbf v\|$$ where $\mathbf v$ is the representative primitive integral vector (up to sign) of $v$ in $\oplus_{1\leq i\leq d}\mathbb Z e_i$. Note that the height function $H(\cdot)$ is defined up to a multiplicative constant.
  
If $\rho_\chi:\mathbf G\to\GL(V)$ is an irreducible $\mathbb Q$-rational representation of $\mathbf G$ generated by a one-dimensional weight space $V_{\chi}\subset V$ of highest weight $\chi$ such that the stabilizer of $V_{\chi}$ is $\mathbf P$, then we may identify $\mathbf X=\mathbf P\backslash\mathbf G$ with orbit $\mathbf G\cdot V_{\chi}$ in $\mathbf P(V)$ by $$\mathbf X\to\mathbf G\cdot V_\chi,\quad\mathbf P\cdot g\mapsto\rho_\chi(g^{-1})\cdot V_\chi\in\mathbf G\cdot V_\chi$$ and the height of any rational point in $\mathbf X=\mathbf P\backslash\mathbf G$ can be defined by restriction of the height function $H(\cdot)$ on $\mathbf G\cdot V_\chi\subset\mathbf P(V)$. 
 
As in Theorem~\ref{thm10}, we may choose a maximal $\mathbb Q$-split torus $\mathbf T$ and a minimal parabolic $\mathbb Q$-subgroup $\mathbf P_0$ in $\mathbf G$ such that $\mathbf T\subset\mathbf P_0\subset\mathbf P$. Then $\mathbf P_0$ and $\mathbf T$ define a root system ($\Phi,\;\Phi^+,\;\Delta$) where $\Phi$ is the set of $\mathbb Q$-roots relative to $\mathbf T$, $\Phi^+$ is the set of positive $\mathbb Q$-roots determined by $\mathbf P_0$ and $\Delta$ is the set of simple $\mathbb Q$-roots in $\Phi^+$. We denote by $R_u(\mathbf P)$ (resp. $R_u(\mathbf P_0)$) the unipotent radical of $\mathbf P$ (resp. $\mathbf P_0$), and write $\Phi(R_u(\mathbf P))$ for the set of $\mathbb Q$-roots in $R_u(\mathbf P)$ relative to $\mathbf T$. The symbol $\sum_{\alpha\in\Phi(R_u(\mathbf P))}$ stands for the sum over all $\mathbb Q$-roots $\alpha\in\Phi(R_u(\mathbf P))$ counted with multiplicities (i.e., the dimensions of the $\mathbb Q$-root spaces associated to $\alpha\in\Phi(R_u(\mathbf P))$ in the Lie algebra of $R_u(\mathbf P)$).

In order to apply Theorem~\ref{thm10}, we choose a one-parameter subgroup $\{a_t\}_{t\in\mathbb R}$ in $\mathbf T(\mathbb R)$ as follows: $a_t=\exp(t\cdot Y)$ where $Y$ is an element in the Lie algebra of $\mathbf T(\mathbb R)$ satisfying $$\begin{cases} \alpha(Y)=\alpha(a_1)=0,\textup{ if }\alpha\in\Delta_{\mathbf P}\\ \alpha(Y)=\alpha(a_1)=-1,\textup{ if }\alpha\in\Delta\setminus\Delta_{\mathbf P}\end{cases}.$$ Here $\Delta_{\mathbf P}\subset\Delta$ is the set of simple $\mathbb Q$-roots associated to $\mathbf P$ . We denote by $\nu_{0}$ any $\mathbb Q$-root relative to $\mathbf T$ in $\Phi^+$ satisfying $$\nu_{0}(a_{-1})=\max\{\alpha(a_{-1}):\alpha\in\Phi^+\}.$$
 
On the variety $\mathbf X(\mathbb R)$, there is a Carnot-Carath\'eodory distance $d_{\textup{CC}}$, and its explicit definition can be found in \cite[Chapter 2, \S2]{d21}. With this Carnot-Carath\'eodory distance, for any $x\in\mathbf X(\mathbb R)$, one may define the Diophantine exponent $\beta_\chi(x)$ of $x$ by $$\beta_\chi(x)=\inf\{\beta>0:\exists\textup{$C>0$ such that }d_{\textup{CC}}(x,v)\geq C\cdot H(v)^{-\beta}\;(\forall v\in\mathbf X(\mathbb Q))\}.$$ We have the following theorem about the Diophantine exponent $\beta_\chi(x)$.
 
\begin{theorem}[{\cite[Theorem 2.4.5]{d21}}]\label{thm110}
There exits a positive constant $\beta_\chi>0$ such that for almost every $x\in\mathbf X(\mathbb R)$ (with respect to a Riemannian volume form on $\mathbf X(\mathbb R)$) $$\beta_\chi(x)=\beta_\chi.$$ Moreover, we have $\beta_\chi=-1/\chi(Y)=1/\chi(a_{-1})$.
\end{theorem}

\begin{definition}\label{def111}
For any decreasing function $\psi:\mathbb R_+\to\mathbb R_+$, define the set of $\psi$-Diophantine points in $\mathbf X(\mathbb R)$ by $$E_{\mathbf X}(\psi)=\{x\in\mathbf X(\mathbb R): \exists C>0 \textup{ such that } d_{\textup{CC}}(x,v)\geq C\cdot(H(v))^{-\beta_\chi}\psi(H(v))\;(\forall v\in\mathbf X(\mathbb Q))\}$$ where the constant $\beta_\chi$ is defined as in Theorem~\ref{thm110}. We denote by $E_{\mathbf X}(\psi)^c$ the complement of $E_{\mathbf X}(\psi)$ in $\mathbf X(\mathbb R)$. 
\end{definition}
For any function $\psi:\mathbb R_+\to\mathbb R_+$, define the lower order at infinity of $\psi$ by $$\gamma(\psi)=\liminf_{t\to\infty}\left(-\frac{\ln\psi(t)}{\ln t}\right).$$ From a correspondence theorem \cite[Proposition 3.2.4]{d21} and Theorem~\ref{thm10}, we may deduce the following result. 
 \begin{theorem}\label{thm13}
Let $\psi:\mathbb R_+\to\mathbb R_+$ be a decreasing function. Then $0\leq\gamma(\psi)\leq\infty$. If the highest weight $\chi$ of $\rho_\chi$ is a multiple of $\sum_{\alpha\in\Phi(R_u(\mathbf P))}{\alpha}$, then the Hausdorff dimension of $E_{\mathbf X}(\psi)^c$ (with respect to a standard Riemannian metric on $\mathbf X(\mathbb R)$) equals 
\begin{align*}
\dim_H E_{\mathbf X}(\psi)^c=&\dim\mathbf X-\left(\frac1{\beta_\chi}-\frac1{\gamma(\psi)+\beta_\chi}\right)\frac{1}{\chi(a_{-1})\nu_0(a_{-1})}\cdot\sum_{\alpha\in\Phi(R_u(\mathbf P))}\alpha(a_{-1})
\end{align*}
for any $0\leq\gamma(\psi)\leq\infty$.
 \end{theorem}

\subsubsection{Connection to rational approximation to linear subspaces}
Diophantine approximation on the Grassmann variety of $l$-dimensional subspaces in $\mathbb R^n$ by $k$-dimensional $\mathbb Q$-subspaces $(1\leq k\leq l<n)$ is a problem suggested by Schmidt in \cite{Sch67}. It has attracted increasing attention and one may refer to e.g. \cite{M20,d24,C24,J22,Jo22,G23}. Recently, de Saxc\'e has made a breakthrough in \cite{d25} which answers several Schmidt's problems in \cite{Sch67}. Here we discuss a connection between Theorem~\ref{thm10} and \cite{d25}. We mainly follow the setup in \cite{d25}.

For any $1\leq l\leq n-1$ ($l\in\mathbb N$), let $X_l(\mathbb R)$ be the Grassmann variety of $l$-dimensional linear $\mathbb R$-subspaces in $\mathbb R^n$. Let $X_l(\mathbb Q)$ be the set of $\mathbb Q$-rational points in $X_l(\mathbb R)$, i.e. the set of $l$-dimensional $\mathbb Q$-subspaces in $\mathbb R^n$. We define a distance $d(\cdot,\cdot)$ between subspaces in $\mathbb R^n$ as follows. Let $P_1$ and $P_2$ be two linear $\mathbb R$-subspaces. If $P_1=\mathbb R\cdot v_1$ and $P_2=\mathbb R\cdot v_2$ are two lines $(v_1,v_2\in\mathbb R^n\setminus\{0\})$, then the distance between $P_1$ and $P_2$ is the usual distance on the projective space $\mathbf P\mathbb R^{n-1}$, i.e. $$d(P_1,P_2)=\sin(\theta(P_1,P_2))=\frac{\|v_1\wedge v_2\|}{\|v_1\|\|v_2\|}$$ where $\theta(P_1,P_2)$ is the angle between $P_1$ and $P_2$, and the norms are the Euclidean norms on $\mathbb R^n$ and $\bigwedge^2\mathbb R^n$ respectively. If $P_1$ is a line, then the distance between $P_1$ and $P_2$ is defined by $$d(P_1,P_2)=\min\{d(P_1,Q): Q\textup{ is a one-dimensional linear subspace in } P_2\}.$$ If $P_1$ and $P_2$ are two general subspaces in $\mathbb R^n$, then we define the distance between $P_1$ and $P_2$ by $$d(P_1,P_2)=\begin{cases}\max\{d(Q,P_2):Q\textup{ is a one-dimensional subspace in } P_1\},\textup{ if }\dim P_1\leq\dim P_2\\\max\{d(Q,P_1):Q\textup{ is a one-dimensional subspace in } P_2\},\textup{ if } \dim P_2\leq\dim P_1\end{cases}$$

For $1\leq k\leq n-1$, we define a height function $H(\cdot)$ on $X_k(\mathbb Q)$ as follows. For any $\mathbb Q$-subspace $P\in X_k(\mathbb Q)$, we choose an integral basis $\{v_i\}_{1\leq i\leq k}\subset\mathbb Z^n$ for $P$ and define the height of $P$ by $$H(P)=\|v_1\wedge\cdots\wedge v_k\|$$ where the norm is the Euclidean norm on $\bigwedge^k\mathbb R^n$.

Now let $1\leq k\leq l<n$. With the distance $d(\cdot,\cdot)$ and the height function $H(\cdot)$ as above, for any $x\in X_l(\mathbb R)$, we may define the Diophantine exponent of $x$ for approximation by $k$-dimensional rational subspaces by $$\beta_k(x)=\inf\{\beta>0:d(v,x)\geq C\cdot H(v)^{-\beta}\;(\forall v\in X_k(\mathbb Q))\textup{ for some }C>0\}.$$

\begin{theorem}[{\cite[Theorem 1]{d25}}]
Let $1\leq k\leq l<n$. For any $x\in X_l(\mathbb R)$, we have $$\beta_k(x)\geq\frac{n}{k(n-l)}$$ and the equality holds for almost every point $x\in X_l(\mathbb R)$ with respect to a Riemannian volume form  on $X_l(\mathbb R)$.
\end{theorem}

\begin{definition}\label{def114}
For any decreasing function $\psi:\mathbb R_+\to\mathbb R_+$, define the set of $\psi$-Diophantine subspaces in $X_l(\mathbb R)$ for approximation by $k$-dimensional $\mathbb Q$-subspaces by $$E_{l,k}(\psi)=\{x\in X_l(\mathbb R):d(v,x)\geq C\cdot H(v)^{-\frac{n}{k(n-l)}}\cdot\psi(H(v))\;(\forall v\in X_k(\mathbb Q))\textup{ for some }C>0\}.$$ We denote by $E_{l,k}(\psi)^c$ the complement of $E_{l,k}(\psi)$ in $X_l(\mathbb R)$.
\end{definition}

As in \S\ref{flag}, for any function $\psi:\mathbb R_+\to\mathbb R_+$, define the lower order at infinity of $\psi$ by $$\gamma(\psi)=\liminf_{t\to\infty}\left(-\frac{\ln\psi(t)}{\ln t}\right).$$ From a correspondence theorem \cite[Proposition 1]{d25} and Theorem~\ref{thm10} with the representation $\rho_k:\SL_n\to\GL(V)$ $(V=\bigwedge^k\mathbb C^n)$ which is the $k$-th exterior product of the standard representation of $\SL_n$, we may deduce the following result. 
\begin{theorem}\label{thm15}
Let $\psi:\mathbb R_+\to\mathbb R_+$ be a decreasing function, and $1\leq k\leq l<n$. Then $0\leq\gamma(\psi)\leq\infty$, and the Hausdorff dimension of $E_{l,k}(\psi)^c$ (with respect to a standard Riemannian metric on $X_l(\mathbb R)$) equals $$\dim_HE_{l,k}(\psi)^c=(l-k)(n-l)+\frac{n}{n/(k(n-l))+\gamma(\psi)}$$ for any $0\leq\gamma(\psi)\leq\infty$.
\end{theorem}

 \subsection{Strategy of the proofs and organization of the paper}
The investigation of the shrinking target problem in $\SL_3(\mathbb R)/\SL_3(\mathbb Z)$ in \cite{FZ22} provides a strategy to address the main problem above. However, in this paper, since we are dealing with a shrinking target problem in a homogeneous space of a semisimple algebraic group $\mathbf G$ from the perspective of representations, the studies on the structure of the group $\mathbf G$ and the structure of the representation $\rho$ would be fundamental in our analysis. Due to the complexities of the structures of $\mathbf G$ and $\rho$, we have to upgrade the method in \cite{FZ22} considerably to tackle the main problem. Here we may reduce the problem to computing $\dim_HS_\rho(\psi)^c\cap R_u(\overline{\mathbf P}_0)(\mathbb R)$.

We first define a notion of rational elements in $\mathbf G(\mathbb R)$ according to the representation $\rho:\mathbf G\to\GL(V)$, and mainly focus on the rational elements in $R_u(\overline{\mathbf P}_0)(\mathbb R)$. The set of rational elements in $R_u(\overline{\mathbf P}_0)(\mathbb R)$ is usually not countable, but a collection of leaves in a foliation in $R_u(\overline{\mathbf P}_0)(\mathbb R)$. To get a lower bound of the Hausdorff dimension of $S_\rho(\psi)^c\cap R_u(\overline{\mathbf P}_0)(\mathbb R)$, we would like to choose neighborhoods around these (leaves of) rational elements and construct a Cantor-type subset. In order to apply Hausdorff dimension formulas, at each level of the Cantor-type subset, the neighborhoods should be disjoint. In our case, it may happen that the neighborhoods at each level are not disjoint if we proceed in the usual way. It implies that there are surplus rational elements in the selected leafs, and we have to sieve them out so that the neighborhoods around the remaining rational elements are disjoint. It is also required that the proportion of the remaining rational elements is not small compared to the set of all rational elements in order to avoid any loss of Hausdorff dimension. For this purpose, we will need the mixing property of the flow $\{a_t\}_{t\in\mathbb R}$. The mixing property of $\{a_t\}_{t\in\mathbb R}$, together with the structure of the Siegel sets in $\mathbf G(\mathbb R)$, not only helps us sieve out surplus rational elements, but also gives an asymptotic estimate about the measure of the subset of the remaining rational elements. Then we may apply the Hausdorff dimension formula of a Cantor-type set and get a lower bound for $\dim_HS_\rho(\psi)^c\cap R_u(\overline{\mathbf P}_0)(\mathbb R)$.

To get an upper bound of $\dim_HS_\rho(\psi)^c\cap R_u(\overline{\mathbf P}_0)(\mathbb R)$, we will need the Bruhat decomposition of $\mathbf G$, or precisely the decomposition $$\mathbf G(\mathbb Q)=\bigcup_{w\in{_{\mathbb Q}\overline{\mathcal W}}}\mathbf P_{\beta_0}(\mathbb Q)\cdot w\cdot\mathbf F_{w}(\mathbb Q).$$ For any $p\in S_\rho(\psi)^c\cap R_u(\overline{\mathbf P}_0)(\mathbb R)$, from the structures of the representation $\rho$ and the Siegel sets in $\mathbf G(\mathbb R)$, one may deduce that there exist a Weyl element $w\in{_\mathbb Q\overline{\mathcal W}}$, a sequence of rational elements $\{q_k\}_{k\in\mathbb N}$ and a sequence of subsets $\{E_k\}_{k\in\mathbb N}$ in $\mathbf F_{w}(\mathbb R)$ (defined from the preimages of certain open subsets under the morphism $\Psi_w$ contracted by the action of $\{a_t\}_{t\in\mathbb R}$) such that $$p\in E_k\cdot w^{-1}q_k\;(k\in\mathbb N).$$ So we can divide the set $S_\rho(\psi)^c\cap R_u(\overline{\mathbf P}_0)(\mathbb R)$ into finitely many subsets $\mathcal E_{w,\epsilon}(\psi)$ (indexed by the elements $w\in{_\mathbb Q\overline{\mathcal W}}$),  each of which corresponds to a type of Diophantine approximation in the sense that any $p\in\mathcal E_{w,\epsilon}(\psi)$ is covered by infinitely many subsets of the form $E\cdot w^{-1}q$ described as above. Now fix $w\in{_\mathbb Q\overline{\mathcal W}}$. As mentioned before, the set of rational elements $q$ is a collection of leaves in some foliation in $\mathbf G(\mathbb R)$, and by the structure of the representation $\rho$ and the morphism $\Psi_w$, the shape of the subset $E$ defined above may be arbitrary and there may be rational elements $q$ far away from a point $p\in\mathcal E_{w,\epsilon}(\psi)$ but one still has $p\in E\cdot w^{-1}q$. So if we construct an open cover of $\mathcal E_{w,\epsilon}(\psi)$ directly from cutting all the subsets of the form $E\cdot w^{-1}q$ into cubes, there would be many unnecessary open cubes counted in the open cover. In order to obtain an upper bound of the Hausdorff dimension of $\mathcal E_{w,\epsilon}(\psi)$, we project the subset $\mathcal E_{w,\epsilon}(\psi)$ into the compact quotient space $R_u(\overline{\mathbf P}_0)(\mathbb R)/(R_u(\overline{\mathbf P}_0)(\mathbb R)\cap\Gamma)$, and construct an open cover for the projection of $\mathcal E_{w,\epsilon}(\psi)$ instead. The projections of the subsets of the form $E\cdot w^{-1}q$ overlap in $R_u(\overline{\mathbf P}_0)(\mathbb R)/(R_u(\overline{\mathbf P}_0)(\mathbb R)\cap\Gamma)$, and by a careful analysis, one can show that the union of these projections can be described in a quite economical way, from which an efficient open cover can be constructed for the projection of $\mathcal E_{w,\epsilon}(\psi)$. Then by the countable stability of Hausdorff dimension, we may obtain an upper bound of $\dim_H\mathcal E_{w,\epsilon}(\psi)$ for each $w\in{_\mathbb Q\overline{\mathcal W}}$, and consequently an upper bound for $\dim_HS_\rho(\psi)^c\cap R_u(\overline{\mathbf P}_0)(\mathbb R)$. Note that in this way, we are able to avoid many computations in \cite{FZ22}, and the only important information we need is about the growth rates of the volume estimates of certain real algebraic varieties and the growth rates of the numbers of rational points in these algebraic varieties, for which we may appeal to algebraic geometry (Cf. Manin's conjecture \cite{BM90}).
 
The paper is organized as follows:
\begin{enumerate}
\item[$\bullet$] In \S\ref{pre}, we first give introductions to the reduction theory of arithmetic lattices and the theory of irreducible representations of complex semisimple groups and complex Lie algebras. Then we prove Theorem~\ref{range}. After that, we define a notion of rational elements and discuss some properties of the rational elements which will be important in the subsequent sections.
\item[$\bullet$] In \S\ref{counting}, we assume that the action of $\{a_t\}_{t\in\mathbb R}$ on any arithmetic quotient $\mathbf G(\mathbb R)^0/\Gamma$ is mixing, where $\mathbf G(\mathbb R)^0$ is the identity component of $\mathbf G(\mathbb R)$ (as a Lie group) and $\Gamma$ is a subgroup of finite index in $\mathbf G(\mathbb Z)$. We select certain rational elements in $R_u(\overline{\mathbf P}_0)(\mathbb R)$ according to some algebraic conditions, and apply the mixing property to obtain an asymptotic estimate about the measure of the subset of the selected rational elements. In \S\ref{lbd1}, we construct a Cantor-type subset $\mathbf A_\infty$ in $R_u(\overline{\mathbf P}_0)(\mathbb R)$ from neighborhoods of these selected rational elements. With the asymptotic estimate result established in \S\ref{counting} and the condition that $0\leq\tau(\psi)<\beta_0(a_{-1})$, we compute the Hausdorff dimension of $\mathbf A_\infty$ and prove Theorem~\ref{mthm11}.
\item[$\bullet$] In \S\ref{lbd2}, we discuss the case where $0\leq\tau(\psi)<\beta_0(a_{-1})$ but the action of $\{a_t\}_{t\in\mathbb R}$ on any arithmetic quotient of $\mathbf G(\mathbb R)^0$ is not mixing. Then there is a proper normal $\mathbb Q$-subgroup $\mathbf H$ of $\mathbf G$ such that $\mathbf H$ contains $\{a_t\}_{t\in\mathbb R}$ and the action of $\{a_t\}_{t\in\mathbb R}$ on any arithmetic quotient of $\mathbf H(\mathbb R)^0$(= the identity component of $\mathbf H(\mathbb R)$) is mixing.  We show that the arguments in \S\ref{counting} and \S\ref{lbd1} work equally well in the group $\mathbf H(\mathbb R)$ and Theorem~\ref{mthm11} still holds in this case. At the end of this section, we discuss the case $\tau(\psi)=\beta_0(a_{-1})$. This completes the proof of Theorem~\ref{mthm11}.
\item[$\bullet$] In \S\ref{ubd}, we divide the subset $S_\rho(\psi)^c\cap R_u(\overline{\mathbf P}_0)(\mathbb R)$ into finitely many subsets $\mathcal E_{w,\epsilon}(\psi)$ indexed by elements $w\in{_\mathbb Q\overline{\mathcal W}}$, each of which corresponds to a type of Diophantine approximation. We project the subsets $\mathcal E_{w,\epsilon}(\psi)$ into the quotient space $R_u(\overline{\mathbf P}_0)(\mathbb R)/R_u(\overline{\mathbf P}_0)(\mathbb R)\cap\Gamma$ and construct open covers for the projections of $\mathcal E_{w,\epsilon}(\psi)$. This gives an upper bound of the Hausdorff dimension of $S_\rho(\psi)^c\cap R_u(\overline{\mathbf P}_0)(\mathbb R)$ and proves Theorem~\ref{mthm12}.
\item[$\bullet$] In \S\ref{cor}, we derive Theorems~\ref{thm10},~\ref{thm11} and~\ref{thm12} from Theorems~\ref{mthm11} and ~\ref{mthm12}. The proofs of Theorem~\ref{thm13} and Theorem~\ref{thm15} are given in \S\ref{Diophantine}.
\end{enumerate}
 
\section{Preliminaries}\label{pre}
In this section, we list some preliminaries needed in this paper and prove Theorem~\ref{range}. Then we give a definition of rational elements in $\mathbf G(\mathbb R)$ and discuss some properties of rational elements. Let $\mathbf G$, $\mathbf T$, $\{a_t\}_{t\in\mathbb R}$ and the representation $\rho:\mathbf G\to\GL(V)$ satisfy the assumption in \S\ref{results}. 

We first need the reduction theory of arithmetic subgroups of $\mathbf G(\mathbb R)$ \cite{BH62,B69}. Let $K$ be a maximal compact subgroup in $\mathbf G(\mathbb R)$ and $\Gamma$ an arithmetic subgroup in $\mathbf G(\mathbb Z)$. We fix a minimal parabolic $\mathbb Q$-subgroup $\mathbf P_0$ in $\mathbf G$ containing $\mathbf T$ with the Levi subgroup $Z(\mathbf T)$ (the centralizer of $\mathbf T$ in $\mathbf G$). Denote by $\Phi$ the set of $\mathbb Q$-roots in $\mathbf G$ with respect to $\mathbf T$, $\Phi^+$ the set of positive $\mathbb Q$-roots corresponding to the minimal parabolic $\mathbb Q$-subgroup $\mathbf P_0$ and $\Delta$ the set of simple $\mathbb Q$-roots in $\Phi^+$. For any parabolic $\mathbb Q$-subgroup $\mathbf P$, denote by $R_u(\mathbf P)$ its unipotent radical. We write $\mathbf P_0=\mathbf M_0\cdot R_u(\mathbf P_0)$ where $\mathbf M_0=Z(\mathbf T)$ is the Levi subgroup, and write $\mathbf M_{a}$ for the maximal $\mathbb Q$-anisotropic subgroup in $\mathbf M_0$ so that $\mathbf M_0=\mathbf T\cdot\mathbf M_{a}$. Let $M=\mathbf M_a(\mathbb R)^0$ be the identity component of the Lie group $\mathbf M_a(\mathbb R)$. For $\eta>0$, denote by $$T_\eta=\{a\in\mathbf T(\mathbb R):\lambda(a)\leq\eta,\;\lambda\textup{ a simple root in }\Delta\}.$$ A Siegel set in $\mathbf G(\mathbb R)$ is a subset of the form $S_{\eta,\Omega}=K\cdot T_\eta\cdot\Omega$ for some $\eta>0$ and some relatively compact open subset $\Omega$ containing identity in $M\cdot R_u(\mathbf P_0)(\mathbb R)$, and the group $\mathbf G(\mathbb R)$ can be written as $$\mathbf G(\mathbb R)=S_{\eta,\Omega}\cdot\mathcal K\cdot\Gamma$$ for some Siegel set $S_{\eta,\Omega}$ and a finite subset $\mathcal K\subset\mathbf G(\mathbb Q)$. Moreover, the finite set $\mathcal K$ satisfies the property that $$\mathbf G(\mathbb Q)=\mathbf P_0(\mathbb Q)\cdot \mathcal K\cdot\Gamma$$ where $\mathbf P_0$ is the minimal parabolic $\mathbb Q$-subgroup in $\mathbf G$. Denote by $$\mathcal K=\{x_1,x_2,\dots,x_k\}=\{x_j\}_{j\in J}\subset\mathbf G(\mathbb Q)$$ and we may assume that $e\in\mathcal K$. In what follows, we choose $\Gamma$ to be an arithmetic subgroup in $\mathbf G(\mathbb Z)\cap\mathbf G(\mathbb R)^0$, where $\mathbf G(\mathbb R)^0$ denotes the identity component of the Lie group $\mathbf G(\mathbb R)$, so that $\rho(\Gamma)$ preserves the lattice $\mathbb Z^d$ in $V$. Without loss of generality, we may assume that the stable horospherical subgroup of $\{a_t\}_{t\in\mathbb R}$ is contained in $R_u(\mathbf P_0)$ and the unstable horospherical subgroup of $\{a_t\}_{t\in\mathbb R}$ is contained in $R_u(\overline{\mathbf P}_0)$, where $\overline{\mathbf P}_0$ is the opposite minimal parabolic $\mathbb Q$-subgroup of $\mathbf P_0$ determined by $\Phi\setminus\Phi^+$ with the same Levi subgroup $\mathbf M_0=Z(\mathbf T)$.

Now we choose a maximal $\mathbb Q$-torus $\mathbf S$ in $\mathbf P_0$ containing $\mathbf T$. The Lie algebra $\mathfrak g$ of $\mathbf G$ can be written as a direct sum of root spaces relative to $\mathbf S$ via the adjoint representation of $\mathbf G$ $$\mathfrak g=\mathfrak g_0\oplus\bigoplus_{\alpha\in\Psi}\mathfrak g_{\alpha},$$ where $\Psi$ is the set of roots in $\mathfrak g$ relative to $\mathbf S$. We may determine a set of positive roots in $\Psi$, which we denote by $\Psi^+$, such that $$\Lie(R_u(\mathbf P_0))\subset\sum_{\alpha\in\Psi^+}\mathfrak g_\alpha\textup{ and }\Lie(R_u(\overline{\mathbf P}_0))\subset\sum_{\alpha\in\Psi\setminus\Psi^+}\mathfrak g_\alpha.$$ Here $\Lie(R_u(\mathbf P_0))$ and $\Lie(R_u(\overline{\mathbf P}_0))$ denote the Lie algebras of $R_u(\mathbf P_0)$ and $R_u(\overline{\mathbf P}_0)$ respectively. The set of simple roots in $\Psi^+$ is denoted by $\Pi$.

It is known that there are complete classifications of irreducible representations of complex semisimple groups and Lie algebras \cite{H72, K86}. Let $V$ be the vector space in the representation $\rho$, and $$V=\bigoplus_{\lambda} V_{\lambda}$$ the decomposition of $V$ into the direct sum of weight spaces $V_\lambda$ relative to $\mathbf S$. According to the theorem of highest weight, there is a unique highest weight $\lambda_0$ among the weights $\lambda$'s (the order is determined by $(\Psi,\Psi^+,\Pi)$) such that
\begin{enumerate}
\item $\dim_{\mathbb C} V_{\lambda_0}=1$.
\item for any $\alpha\in\Psi^+$, any $E_\alpha\in\mathfrak g_\alpha$ annihilates $V_{\lambda_0}$ via the differential $d\rho$ of $\rho$, any $n_\alpha\in\exp(\mathfrak g_\alpha)$ fixes elements in $V_{\lambda_0}$ via $\rho$ (where $\exp$ is the exponential map), and elements of $V_{\lambda_0}$ are the only vectors with this property.
\item every weight $\lambda$ in $\rho$ is of the form $\lambda_0-\sum_{i=1}^ln_i\alpha_i$ where $n_i\in\mathbb Z_{\geq0}$ and $\alpha_i\in\Pi$.
\end{enumerate}
Note that by our choices of $\Phi$ and $\Psi$, we have $$\{\alpha|_{\mathbf T}:\alpha\in\Psi,\alpha|_{\mathbf T}\neq 0\}=\{\alpha:\alpha\in\Phi\}.$$ On the other hand, one can also write $V$ as a direct sum of weight spaces relative to $\mathbf T$ as follows $$V=\bigoplus_{\beta}V_\beta.$$ Let $\beta_0$ be the weight in the decomposition above such that its weight space $V_{\beta_0}$ contains $V_{\lambda_0}$. Note that $\beta_0$ is defined over $\mathbb Q$ and $\beta_0=\lambda_0|_{\mathbf T}$. In particular, $\lambda_0(\mathbf T(\mathbb R))\subset\mathbb R$ and $\beta_0$ is the highest weight among $\beta$'s relative to $\mathbf T$ where the order is determined by the root system $(\Phi,\Phi^+,\Delta)$.

Since the stable subgroup of $\{a_t\}_{t\in\mathbb R}$ is contained in $R_u(\mathbf P_0)$, we have $\alpha(a_t)\leq0$ $(t>0)$ for any $\alpha\in\Phi^{+}$. From property (3) of the representation $\rho$ above and the fact that $\textup{Im}\rho\subset\SL(V)$, one can deduce that $\beta_0(a_t)=\lambda_0(a_t)\leq0$ $(t>0)$. If $\beta_0(a_t)=\lambda_0(a_t)=0$ $(t>0)$, then by the fact that $\textup{Im}\rho\subset\SL(V)$ and $\lambda_0$ is the highest weight relative to $\mathbf S$, for any other weight $\lambda$ in $\rho$ relative to $\mathbf S$, we also have $\lambda(a_t)=0$ $(t>0)$. This implies that $\{a_t\}_{t\in\mathbb R}\subset\ker\rho$, which contradicts the assumption that $\dim\ker\rho=0$. Therefore we have $\beta_0(a_t)=\lambda_0(a_t)<0$ $(t>0)$. Note that $$\rho(a_t)\cdot v=e^{\beta_0(a_t)}\cdot v\to 0\textup{ as }t\to\infty$$ for any $v\in V_{\beta_0}$ (here $\beta_0(a_{-1})>0$ is the fastest contracting rate under the action of $\{a_t\}_{t\in\mathbb R}$).

Let $\mathbf P_{\beta_0}$ be the stabilizer of the weight space $V_{\beta_0}$ in $\mathbf G$. Then $\mathbf P_{\beta_0}$ contains the minimal parabolic $\mathbb Q$-subgroup $\mathbf P_0$ since $\beta_0$ is the highest weight among the weights $\beta$'s relative to $\mathbf T$ in $\rho$. Therefore, $\mathbf P_{\beta_0}$ is a parabolic $\mathbb Q$-subgroup in $\mathbf G$. It is known that there exists a $\mathbb Q$-torus $\mathbf T_{\beta_0}$ in $\mathbf T$ such that $Z(\mathbf T_{\beta_0})$ (the centralizer of $\mathbf T_{\beta_0}$ in $\mathbf G$) is a Levi subgroup of $\mathbf P_{\beta_0}$. We denote by $\overline{\mathbf P}_{\beta_0}$ the opposite parabolic $\mathbb Q$-subgroup of $\mathbf P_{\beta_0}$ containing $\overline{\mathbf P}_0$ with the same Levi subgroup $Z(\mathbf T_{\beta_0})$, and denote by $R_u(\overline{\mathbf P}_{\beta_0})$ the unipotent radical of $\overline{\mathbf P}_{\beta_0}$.

Let $N(\mathbf T)$ and $Z(\mathbf T)$ be the normalizer and centralizer of $\mathbf T$ in $\mathbf G$ respectively. Then the Weyl group relative to $\mathbb Q$ is defined by $$_{\mathbb Q}{W}=N(\mathbf T)/Z(\mathbf T).$$ Let $\mathbf P_0$ be the minimal parabolic $\mathbb Q$-subgroup of $\mathbf G$ as defined above, and ${_{\mathbb Q}\mathcal W}$ a set of representatives of $_{\mathbb Q}{W}$ in $N(\mathbf T)(\mathbb Q)$. Then the Bruhat decomposition of $\mathbf G$ is the following \cite[Theorem 21.15]{B91} $$\mathbf G(\mathbb Q)=\mathbf P_0(\mathbb Q)\cdot {_{\mathbb Q}{W}}\cdot\mathbf P_0(\mathbb Q).$$ Note that the Weyl group $_{\mathbb Q}{W}$ acts transitively on the set of minimal parabolic $\mathbb Q$-subgroups containing $Z(\mathbf T)$ \cite[Corollary 21.3]{B91}. So there exists an element in $_{\mathbb Q}{W}$ which sends $\mathbf P_0$ to $\overline{\mathbf P}_0$, and we let $\bar w$ be the representative of this element in ${_{\mathbb Q}\mathcal W}$. Then we have $$\mathbf G(\mathbb Q)=\mathbf G(\mathbb Q)\cdot\bar w^{-1}=\mathbf P_0(\mathbb Q)\cdot{_{\mathbb Q}\mathcal W}\cdot(\mathbf P_0(\mathbb Q)\cdot\bar w^{-1})=\mathbf P_0(\mathbb Q)\cdot{_{\mathbb Q}\mathcal W}\cdot\overline{\mathbf P}_0(\mathbb Q)$$ which implies that $$\quad\mathbf G(\mathbb Q)=\mathbf P_0(\mathbb Q)\cdot{_{\mathbb Q}\mathcal W}\cdot R_u(\overline{\mathbf P}_0)(\mathbb Q).$$ For each $w\in{_{\mathbb Q}\mathcal W}$, define $$\mathbf F_{w}=R_u(\overline{\mathbf P}_0)\cap w^{-1}R_u(\overline{\mathbf P}_{\beta_0})w,\quad\mathbf H_{w}=R_u(\overline{\mathbf P}_0)\cap w^{-1}\mathbf P_{\beta_0}w$$ and by \cite[Proposition 21.9]{B91}, we have $$R_u(\overline{\mathbf P}_0)=\mathbf H_{w}\cdot\mathbf F_{w}.$$ We may assume that $e\in{_{\mathbb Q}\mathcal W}$. It follows that 
\begin{align*}
\mathbf G(\mathbb Q)=&\bigcup_{w\in{_{\mathbb Q}\mathcal W}} \mathbf P_0(\mathbb Q)\cdot(w\mathbf H_{w}w^{-1})(\mathbb Q)\cdot w\cdot\mathbf F_{w}(\mathbb Q)\\
=&\bigcup_{w\in{_{\mathbb Q}\mathcal W}}\mathbf P_{\beta_0}(\mathbb Q)\cdot w\cdot\mathbf F_{w}(\mathbb Q).
\end{align*}
Here, as discussed in \S\ref{results}, the subsets $\mathbf P_{\beta_0}(\mathbb Q)\cdot w\cdot\mathbf F_{w}(\mathbb Q)$ $(w\in{_{\mathbb Q}\mathcal W})$ may overlap, and we choose any subset ${_{\mathbb Q}\overline{\mathcal W}}$ in ${_{\mathbb Q}\mathcal W}$ (as small as possible) such that $$\mathbf G(\mathbb Q)=\bigcup_{w\in{_{\mathbb Q}\overline{\mathcal W}}}\mathbf P_{\beta_0}(\mathbb Q)\cdot w\cdot\mathbf F_{w}(\mathbb Q).$$ Note that, by \cite[Proposition 21.9]{B91}, the group $w\mathbf H_{w}w^{-1}\subset\mathbf P_{\beta_0}$ is generated by unipotent subgroups whose Lie algebras are sums of root spaces of positive $\mathbb Q$-roots in $\mathfrak g$ relative to $\mathbf T$, and $\beta_0$ is the highest weight in $\rho$. Hence $w\mathbf H_{w}w^{-1}$ fixes every element in $V_{\beta_0}$ by the structure of the representation $\rho$.

Now we prove Theorem~\ref{range}.
\begin{proof}[Proof of Theorem~\ref{range}]
Since $\psi:\mathbb R_+\to\mathbb R_+$ is bounded, by definition $$\tau(\psi)=
\liminf_{t\to\infty}\left(-\frac{\ln(\psi(t))}t\right)\geq0.$$ Note that we have already proved $\beta_0(a_{-1})>0$. Now suppose that $\tau=\tau(\psi)>\beta_0(a_{-1})$. Then for any sufficiently small $\epsilon>0$, there exists $C_\epsilon>0$ such that $$\psi(t)\leq C_\epsilon\cdot e^{-(\tau-\epsilon)t}\quad(\forall t>0).$$ Let $g\in\mathbf G(\mathbb R)$. Since $\beta_0(a_{-1})>0$ is the fastest contracting rate of the $\{a_t\}_{t\in\mathbb R}$-action on $V$, we have $$\|\rho(a_t)v\|\geq e^{-\beta_0(a_{-1})t}\|v\|\quad(\forall v\in\rho(g)\cdot\mathbb Z^d\setminus\{0\})$$ and $$\delta(\rho(a_t\cdot g)\cdot\mathbb Z^d)\geq e^{-\beta_0(a_{-1})t}\cdot\delta(\rho(g)\cdot\mathbb Z^d).$$ Since $\tau>\beta_0(a_{-1})$, by choosing $\epsilon$ sufficiently small, one may find a constant $C>0$ such that $$\delta(\rho(a_t\cdot g)\cdot\mathbb Z^d)\geq C\cdot\psi(t)\quad(t>0).$$ which shows that $g\in S_\rho(\psi)$ and $S_\rho(\psi)=\mathbf G(\mathbb R)$. Note that this argument also works for $\tau=\infty$. This completes the proof of Theorem~\ref{range}.
\end{proof}

In the rest of the paper, we assume that $0\leq\tau(\psi)\leq\beta_0(a_{-1})$. In the following, we define the notion of rational elements and discuss some related properties.

\begin{definition}\label{def21}
An element $g\in\mathbf G(\mathbb R)$ is called rational if $\rho(g)\mathbb Z^d\cap V_{\beta_0}$ is Zariski dense in $V_{\beta_0}$, or equivalently, if $\rho(g)\mathbb Z^d\cap V_{\beta_0}$ is a lattice in the vector space of real points in $V_{\beta_0}$. 
\end{definition}

\begin{lemma}\label{l22}
An element $g\in R_u(\overline{\mathbf P}_{\beta_0})(\mathbb R)$ is rational if and only if $g\in R_u(\overline{\mathbf P}_{\beta_0})(\mathbb Q)$.
\end{lemma}
\begin{proof}
Let $g\in R_u(\overline{\mathbf P}_{\beta_0})(\mathbb R)$ be a rational element in $\mathbf G(\mathbb R)$. Then by definition, there exists a discrete subgroup $\Lambda_g\subset\mathbb Z^d$ such that $\rho(g)\cdot\Lambda_g$ is a lattice in $V_{\beta_0}(\mathbb R)$. Choose $\sigma\in\Gal(\mathbb R/\mathbb Q)$. Since $V_{\beta_0}$ is defined over $\mathbb Q$, we have $$\sigma(\rho(g)\cdot\Lambda_g)=\rho(\sigma(g))\cdot\Lambda_g\subset\sigma(V_{\beta_0}(\mathbb R))=V_{\beta_0}(\mathbb R)$$ and $\rho(\sigma(g))\cdot\Lambda_g$ is Zariski dense in $V_{\beta_0}$. Then we obtain that $$\rho(\sigma(g)g^{-1})V_{\beta_0}=\rho(\sigma(g)g^{-1})\overline{\rho(g)\cdot\Lambda_g}=\overline{\rho(\sigma(g))\cdot\Lambda_g}=V_{\beta_0}$$ and $\sigma(g)g^{-1}$ is in the stabilizer $\mathbf P_{\beta_0}$ of $V_{\beta_0}$. Here $\overline{\rho(g)\cdot\Lambda_g}$ and $\overline{\rho(\sigma(g))\cdot\Lambda_g}$ denote the Zariski closures of $\rho(g)\cdot\Lambda_g$ and $\rho(\sigma(g))\cdot\Lambda_g$ respectively. On the other hand, $\sigma(g)g^{-1}\in R_u(\overline{\mathbf P}_{\beta_0})$ and $R_u(\overline{\mathbf P}_{\beta_0})\cap\mathbf P_{\beta_0}=\{e\}$. Therefore, $\sigma(g)=g$ for any $\sigma\in\Gal(\mathbb R/\mathbb Q)$, and $g\in R_u(\overline{\mathbf P}_{\beta_0})(\mathbb Q)$. The other direction is clear.
\end{proof}

\begin{corollary}\label{c23}
Let $w\in{_{\mathbb Q}\mathcal W}$ and $g\in R_u(\overline{\mathbf P}_0)(\mathbb R)$. Then $w\cdot g$ is rational if and only if $g\in\mathbf H_{w}(\mathbb R)\cdot\mathbf F_{w}(\mathbb Q)$.
\end{corollary}
\begin{proof}
Let $h\in\mathbf H_{w}(\mathbb R)$ and $f\in\mathbf F_{w}(\mathbb R)$ such that $g=h\cdot f$. Suppose that $w\cdot g$ is rational. Since $w\in\mathbf G(\mathbb Q)$, by definition, $w\cdot g\cdot w^{-1}=w\cdot h\cdot f\cdot w^{-1}$ is also rational in $\mathbf G(\mathbb R)$. Since $w\cdot h\cdot w^{-1}$ preserves $V_{\beta_0}$, we get that $w\cdot f\cdot w^{-1}$ is rational. By the fact that $w\mathbf F_{w} w^{-1}\subset R_u(\overline{\mathbf P}_{\beta_0})$ and Lemma~\ref{l22}, we conclude that $w\cdot f\cdot w^{-1}\in R_u(\overline{\mathbf P}_{\beta_0})(\mathbb Q)$ and $f\in\mathbf F_{w}(\mathbb Q)$.

Conversely, suppose that $g=h\cdot f$ where $h\in\mathbf H_{w}(\mathbb R)$ and $f\in\mathbf F_{w}(\mathbb Q)$. Then $$w\cdot g=(whw^{-1})\cdot (wfw^{-1})\cdot w.$$ Note that $whw^{-1}\in\mathbf P_{\beta_0}$ and $w\in\mathbf G(\mathbb Q)$. By definition, we know that $w\cdot g$ is rational. This completes the proof of the corollary.
\end{proof}

By properties of the subset $\mathcal K\subset\mathbf G(\mathbb Q)$ in the reduction theory, we can give another characterization of rational elements in $R_u(\overline{\mathbf P}_0)(\mathbb R)$.

\begin{lemma}\label{l24}
Let $w\in{_{\mathbb Q}\mathcal W}$ and $g\in R_u(\overline{\mathbf P}_0)(\mathbb R)$. Then $w\cdot g$ is rational if and only if $$g\in R_u(\overline{\mathbf P}_0)(\mathbb R)\cap(\mathbf H_{w}(\mathbb R)\cdot \mathbf P_0(\mathbb R)\cdot \mathcal K\cdot\Gamma).$$
\end{lemma}
\begin{proof}
If $w\cdot g$ is rational, then by Corollary~\ref{c23}, we have $$g\in\mathbf H_{w}(\mathbb R)\cdot\mathbf F_{w}(\mathbb Q)\subset\mathbf H_{w}(\mathbb R)\cdot\mathbf G(\mathbb Q)=\mathbf H_{w}(\mathbb R)\cdot\mathbf P_0(\mathbb Q)\cdot \mathcal K\cdot\Gamma\subset\mathbf H_{w}(\mathbb R)\cdot\mathbf P_0(\mathbb R)\cdot \mathcal K\cdot\Gamma.$$ Now suppose that $g\in R_u(\overline{\mathbf P}_0)(\mathbb R)\cap(\mathbf H_{w}(\mathbb R)\cdot \mathbf P_0(\mathbb R)\cdot \mathcal K\cdot\Gamma)$. Then there exist $h\in\mathbf H_{w}(\mathbb R)$ and $f\in\mathbf F_{w}(\mathbb R)$ such that $$g=h\cdot f,\quad f\in\mathbf H_{w}(\mathbb R)\cdot \mathbf P_0(\mathbb R)\cdot \mathcal K\cdot\Gamma.$$ We write $$f=h\cdot p\cdot x\cdot\gamma$$ for some $h\in\mathbf H_{w}(\mathbb R)$, $p\in\mathbf P_0(\mathbb R)$, $x\in \mathcal K$ and $\gamma\in\Gamma$. Let $\sigma\in\Gal(\mathbb R/\mathbb Q)$. Then we have $$\sigma(f)=\sigma(h)\cdot\sigma(p)\cdot x\cdot\gamma.$$ This implies that $x\cdot\gamma=p^{-1}\cdot h^{-1}\cdot f=\sigma(p)^{-1}\cdot\sigma(h_i)^{-1}\cdot\sigma(f)$. Since the product map $$\mathbf P_0\times\mathbf H_{w}\times\mathbf F_{w}\to\mathbf G$$ is injective, we have $\sigma(f)=f$ for any $\sigma\in\Gal(\mathbb R/\mathbb Q)$ and hence $f\in\mathbf F_{w}(\mathbb Q)$. By Corollary~\ref{c23}, $w\cdot g$ is rational.
\end{proof}

\begin{definition}\label{d25}
Let $g$ be a rational element in $\mathbf G(\mathbb R)$. The height $\height(g)$ of $g$ is defined to be the co-volume of the lattice $\rho(g)\mathbb Z^d\cap V_{\beta_0}$ in the vector space of real points in $V_{\beta_0}$.
\end{definition}

In the following lemma, we discuss the relation between the length of the shortest non-zero vector and the co-volume of the discrete subgroup $\rho(g)\cdot\mathbb Z^n\cap V_{\beta_0}$ for $g\in\mathbf G(\mathbb R)$. For convenience, we write $A\lesssim B$ $(A\gtrsim B)$ if there exists a constant $c>0$ such that $$A\leq c\cdot B\quad(A\geq c\cdot B).$$ If $A\lesssim B$ and $A\gtrsim B$, then we write $A\asymp B$. We will specify the implicit constants in the contexts if necessary.

\begin{lemma}\label{l26}
Let $w\in{_{\mathbb Q}\mathcal W}$ and $g\in R_u(\overline{\mathbf P}_0)(\mathbb R)$. Suppose that $w\cdot g$ is rational. Then $$\height(w\cdot g)\asymp\delta(\rho(w\cdot g)\cdot\mathbb Z^d\cap V_{\beta_0})^{\dim V_{\beta_0}}.$$ Here the implicit constant depends only on $\mathbf G$ and $\Gamma$.
\end{lemma}
\begin{proof}
By Corollary~\ref{c23}, there exist $h\in\mathbf H_{w}(\mathbb R)$ and $f\in\mathbf F_{w}(\mathbb Q)$ such that $g=h\cdot f$. Note that $$f\in\mathbf G(\mathbb Q)=w^{-1}\cdot\mathbf G(\mathbb Q)=w^{-1}\cdot\mathbf P_0(\mathbb Q)\cdot \mathcal K\cdot\Gamma$$ and there exist $p\in\mathbf P_0(\mathbb Q)$, $x\in \mathcal K$ and $\gamma\in\Gamma$ such that $f=w^{-1}\cdot p\cdot x\cdot\gamma$. We know that $\mathbf P_0=\mathbf M_0\cdot R_u(\mathbf P_0)$ where $\mathbf M_0$ is the Levi factor of $\mathbf P_0$. We write $\mathbf M_a$ for the maximal $\mathbb Q$-anisotropic subgroup in $\mathbf M_0$. Then $\mathbf M_0=\mathbf T\cdot\mathbf M_a$, $\mathbf T$ commutes with $\mathbf M_a$ and $\mathbf M_a(\mathbb R)/(\mathbf M_a(\mathbb Z)\cap\Gamma)$ is compact. So there exist $p_1\in\mathbf T$, $p_2$ in a compact fundamental domain of $\mathbf M_a(\mathbb Z)\cap\Gamma$ in $\mathbf M_a(\mathbb R)$, $p_3\in\mathbf M_a(\mathbb Z)\cap\Gamma$ and $u\in R_u(\mathbf P_0)$ such that $$p=u\cdot p_1\cdot p_2\cdot p_3.$$ Then we have $$w\cdot g=(w\cdot h\cdot w^{-1})\cdot u\cdot p_1\cdot p_2\cdot p_3\cdot x\cdot\gamma.$$ Since $w\mathbf H_{w}w^{-1}$ and $R_u(\mathbf P_0)$ fix every element in $V_{\beta_0}$, and $\mathbf T$ and $\mathbf M_a$ preserve the weight space $V_{\beta_0}$, we obtain $$\rho(w\cdot g)\cdot\mathbb Z^d\cap V_{\beta_0}=\rho(p_1\cdot p_2\cdot p_3\cdot x)\cdot\mathbb Z^d\cap V_{\beta_0}=\rho(p_1\cdot p_2)(\rho(p_3\cdot x)\cdot\mathbb Z^d\cap V_{\beta_0}).$$ The lemma then follows from the facts that $p_1\in\mathbf T$ acts as scalars in $V_{\beta_0}$, $p_2$ is in a fixed compact subset in $\mathbf M_a(\mathbb R)$ and $\rho(p_3\cdot x)\cdot\mathbb Z^d$ is commensurable with $\mathbb Z^d$ by $\Gamma$ and $\mathcal K$.
\end{proof}

Now we consider the case where $w=e$ is the identity element in $_{\mathbb Q}\mathcal W$. Then $\mathbf F_{w}=\mathbf F_e$ and $\mathbf H_{w}=\mathbf H_e$. Let $g$ be a rational element in $R_u(\overline{\mathbf P}_0)(\mathbb R)$. By Lemma~\ref{l24}, there exist $h\in\mathbf H_e(\mathbb R)$, $p\in\mathbf P_0(\mathbb R)$, $x\in\mathcal K$ and $\gamma\in\Gamma$ such that $$g=h\cdot p\cdot x\cdot\gamma.$$ Furthermore, we can write $$p=a\cdot m\cdot u$$ where $a\in\mathbf T(\mathbb R)$, $m\in\mathbf M_a(\mathbb R)$ and $u\in R_u(\mathbf P_0)(\mathbb R)$. Then we can compute the height $\height(g)$ of $g$ as follows: $$(\rho(g)\cdot\mathbb Z^d)\cap V_{\beta_0}=\rho(a\cdot m\cdot u\cdot x)\cdot\mathbb Z^d\cap V_{\beta_0}=\rho(a\cdot m\cdot x)\cdot\mathbb Z^d\cap V_{\beta_0}=\rho(a\cdot m)(\rho(x)\cdot\mathbb Z^d\cap V_{\beta_0})$$ and $$\height(g)=c_x\cdot e^{\beta_0(a)\cdot\dim V_{\beta_0}}$$ for some constant $c_x>0$ depending only on $x$. Here we use the facts that $\mathbf M_a$ and $R_u(\mathbf P_0)$ stabilize $V_{\beta_0}$ and preserve the volumes of sets in $V_{\beta_0}$, $\mathbf T$ acts as scalars in $V_{\beta_0}$ and $x\in\mathbf G(\mathbb Q)$. 

\begin{definition}\label{def27}
Let $\mathcal K=\{x_j\}_{j\in J}$. A rational element $g$ in $R_u(\overline{\mathbf P}_0)(\mathbb R)$ is called $j$-rational for some $j\in J$ if it can be written as $$g=h\cdot p\cdot x_j\cdot\gamma$$ for some $h\in\mathbf H_e(\mathbb R)$, $p\in\mathbf P_0(\mathbb R)$, $x_j\in\mathcal K$ and $\gamma\in\Gamma$. Furthermore, if $p=a\cdot m\cdot u$ where $a\in\mathbf T(\mathbb R)$, $m\in\mathbf M_a(\mathbb R)$ and $u\in R_u(\mathbf P_0)(\mathbb R)$, then $a$ is called the central coordinate of $g$ which we denote by $a_g$.
\end{definition}

\section{Counting rational elements}\label{counting}
In this section, we consider the problem of counting rational elements in $R_u(\overline{\mathbf P}_0)(\mathbb R)$. As discussed in \S\ref{intro}, to get a lower bound of the Hausdorff dimension of $S_\rho(\psi)^c\cap R_u(\overline{\mathbf P}_0)(\mathbb R)$, we need to select certain rational elements and construct a Cantor type subset from disjoint neighborhoods of these rational elements. For our purpose, in the following, we will select rational elements according to several algebraic conditions, and then count these rational elements by the mixing property of the flow $\{a_t\}_{t\in\mathbb R}$ on $\mathbf G(\mathbb R)^0/\Gamma$. The disjointness of the neighborhoods around these rational elements will follow from the transversal structure of some submanifolds in $\mathbf G(\mathbb R)^0/\Gamma$. 

For any $\mathbb Q$-algebraic group $\mathbf L$ in $\mathbf G$ and $\delta>0$, we write $B_{\mathbf L}(\delta)$ for the open ball of radius $\delta>0$ around the identity in $\mathbf L(\mathbb R)$. For any $t\in\mathbb R$, we write $$B_{\mathbf L}(\delta,t):=a_{-t}\cdot B_{\mathbf L}(\delta)\cdot a_t.$$ The identity component of $\mathbf L$ with respect to the Zariski topology is denoted by $\mathbf L^0$, and the identity component of $\mathbf L(\mathbb R)$ (as a Lie group) is denoted by $\mathbf L(\mathbb R)^0$.

In this section and \S\ref{lbd1}, we assume that the action of the one-parameter subgroup $\{a_t\}_{t\in\mathbb R}$ on $\mathbf G(\mathbb R)^0/\Gamma$ is mixing. This assumption holds when $\{a_t\}_{t\in\mathbb R}$ projects nontrivially into any $\mathbb Q$-simple factor of $\mathbf G$. Later in \S\ref{lbd2}, we will explain how to establish the results about counting rational elements and estimating Hausdorff dimensions in the case where $\{a_t\}_{t\in\mathbb R}$ projects trivially into some of the $\mathbb Q$-simple factors of $\mathbf G$.

Let $U$ be a small open bounded subset in $R_u(\overline{\mathbf P}_0)(\mathbb R)$ which projects injectively into $R_u(\overline{\mathbf P}_0)(\mathbb R)/R_u(\overline{\mathbf P}_0)(\mathbb R)\cap\Gamma$. For any $0<A<B$, define $$S(U,A,B)=\{q\in U: q \textup{ rational and } A\leq \height(q)\leq B\}.$$ The Lie algebra $\mathfrak a$ of $\mathbf T(\mathbb R)$ has the following direct sum decomposition $$\mathfrak a=\Lie(a_t)\oplus\ker(\beta_0)$$ where $\Lie(a_t)$ is the Lie algebra of the one parameter subgroup $\{a_t\}_{t\in\mathbb R}$ and we have $$\beta_0(a_t)=\lambda_0(a_t)<0\;(t>0).$$ Denote by $\pi_{\ker(\beta_0)}$ the projection of $\mathfrak a$ onto $\ker(\beta_0)$ along the linear subspace $\Lie(a_t)$. For convenience, we will write $\pi_{\ker(\beta_0)}(a)$ for $\pi_{\ker(\beta_0)}(\log (a))$ whenever $a\in\exp(\mathfrak a)$ (here $\log$ is the inverse of the exponential map $\exp$). For any compact subset $L$ in $\Lie(a_t)$ and any compact subset $K$ in $\ker(\beta_0)$, we denote by $$\mathfrak a_{L, K}=\{x\in\mathfrak a: x=y_1+y_2, y_1\in L, y_2\in K\}.$$ For any $x_j\in\mathcal K$ ($j\in J$), any compact subset $K_1$ in $\ker(\beta_0)$, any compact subset $K_2$ in $\mathbf H_e(\mathbb R)$, any compact subset $K_3\subset\mathbf M_a(\mathbb R)\cap\mathbf G(\mathbb R)^0$ and any compact subset $K_4\subset R_u(\mathbf P_0)(\mathbb R)$, we define $$S_{K_1,K_2,K_3,K_4,j}(U,A,B)$$ to be the set of all rational elements $q$ in $U$ such that
\begin{enumerate}
\item $A\leq\height(q)\leq B$ and $q$ is $j$-rational for $x_j\in\mathcal K$; 
\item $q=a\cdot h\cdot m\cdot u\cdot x_j\cdot\gamma$ for some $a\in\exp(\mathfrak a)$, $\pi_{\ker(\beta_0)}(a)\in K_1$ and $h\in K_2$, $m\in K_3$, $u\in K_4$ and $\gamma\in\Gamma$.
\end{enumerate}
Note that $S_{K_1,K_2,K_3,K_4,j}(U,A,B)\subset S(U,A,B)$, and by definition one can check that if $$S_{K_1,K_2,K_3,K_4,j}(U,A,B)\neq\emptyset,$$ then $x_j\in\mathbf G(\mathbb R)^0$ (and such elements exist as $e\in\mathcal K$). The elements in $S_{K_1,K_2,K_3,K_4,j}(U,A,B)$ are the rational elements we select when we construct a Cantor type subset in \S\ref{lbd1}. We denote by $$S_{K_1,K_2,K_3,K_4}(U,A,B):=\bigcup_{j\in J} S_{K_1,K_2,K_3,K_4,j}(U,A,B).$$

To count the rational elements in $S_{K_1,K_2,K_3,K_4,j}(U,A,B)$ $(j\in J)$, we need the following result about limiting distributions of translates of unipotent orbits pushed by the flow $\{a_t\}_{t\in\mathbb R}$ on $\mathbf G(\mathbb R)^0/\Gamma$, which is a direct consequence of the mixing property of $\{a_t\}_{t\in\mathbb R}$ on $\mathbf G(\mathbb R)^0/\Gamma$. 

\begin{proposition}\label{p31}
Let $x\in\mathbf G(\mathbb R)^0/\Gamma$ and $W\subset\mathbf G(\mathbb R)^0/\Gamma$ an open bounded subset whose boundary has measure zero with respect to the invariant probability measure $\mu_{\mathbf G^0(\mathbb R)/\Gamma}$ on $\mathbf G(\mathbb R)^0/\Gamma$. Let $\chi_{W}$ denote the characteristic function of $W$. Then for any bounded open subset $U$ in $R_u(\overline{\mathbf P}_0)(\mathbb R)$ we have $$\lim_{t\to\infty}\frac1{\mu_{R_u(\overline{\mathbf P}_0)}(U)}\int_{U}\chi_W(a_t\cdot nx)d\mu_{R_u(\overline{\mathbf P}_0)}(n)=\int_{\mathbf G(\mathbb R)^0/\Gamma}\chi_W d\mu_{\mathbf G(\mathbb R)^0/\Gamma}$$ where $\mu_{R_u(\overline{\mathbf P}_0)}$ is the Haar measure on $R_u(\overline{\mathbf P}_0)(\mathbb R)$.
\end{proposition}
\begin{remark}\label{r32}
Note that $R_u(\overline{\mathbf P}_0)(\mathbb R)$ is not necessarily the unstable horospherical subgroup of $\{a_t\}_{t\in\mathbb R}$ in $\mathbf G(\mathbb R)$. Here we can still apply the mixing property to obtain Proposition~\ref{p31} as long as $R_u(\overline{\mathbf P}_0)(\mathbb R)$ is contained in the group generated by the unstable subgroup and the centralizer of $\{a_t\}_{t\in\mathbb R}$ in $\mathbf G(\mathbb R)$. (See~\cite[\S2]{KM96}.)
\end{remark}

Now we define a measure on $R_u(\overline{\mathbf P}_0)(\mathbb R)$, which will be used to study the size of the subset $S_{K_1,K_2,K_3,K_4,j}(U,A,B)$ for some sufficiently large numbers $0<A<B$. Recall that $$R_u(\overline{\mathbf P}_0)=\mathbf H_e\cdot\mathbf F_e$$ where $\mathbf F_e=R_u(\overline{\mathbf P}_{\beta_0})$ and $\mathbf H_e=R_u(\overline{\mathbf P}_0)\cap\mathbf P_{\beta_0}$. By \cite[Proposition 5.26]{K86}, we fix a Haar measure $\mu_{\mathbf H_e}$ on $\mathbf H_e(\mathbb R)$ and a Haar measure $\mu_{\mathbf F_e}$ on $\mathbf F_e(\mathbb R)$ such that the product maps $$\mathbf F_e\times\mathbf H_e\to\mathbf F_e\cdot\mathbf H_e=R_u(\overline{\mathbf P}_0)\textup{ and }\mathbf H_e\times\mathbf F_e\to\mathbf H_e\cdot\mathbf F_e=R_u(\overline{\mathbf P}_0)$$ induce Haar measures on $R_u(\overline{\mathbf P}_0)(\mathbb R).$ For any $q\in\mathbf F_e(\mathbb Q)$, we define $m_{\mathbf H_eq}$ to be the locally finite measure defined on $R_u(\overline{\mathbf P}_0)(\mathbb R)$ which is supported on $\mathbf H_e(\mathbb R)\cdot q$ and induced by $\mu_{\mathbf H_e}$ via the product map $$\mathbf H_e(\mathbb R)\times\{q\}\to\mathbf H_e(\mathbb R)\cdot q\subset R_u(\overline{\mathbf P}_0)(\mathbb R).$$ Then we define $$m_{\mathbf H_e}:=\sum_{q\in\mathbf F_e(\mathbb Q)}m_{\mathbf H_eq}.$$ Note that $m_{\mathbf H_e}$ is not a locally finite measure on $R_u(\overline{\mathbf P}_0)(\mathbb R)$, and it is defined by the leaf-wise measures on the countable leaves $\mathbf H_e(\mathbb R)\cdot q$ $(q\in\mathbf F_e(\mathbb Q))$ in the foliation $\mathcal F_{\mathbf H_e}$ induced by the group action of $\mathbf H_e(\mathbb R)$ on $R_u(\overline{\mathbf P}_0)(\mathbb R)$. Note also that by Corollary~\ref{c23}, we have $$S_{K_1,K_2,K_3,K_4,j}(U,A,B)\subset\bigcup_{q\in\mathbf F_e(\mathbb Q)}\mathbf H_e(\mathbb R)\cdot q=\mathbf H_e(\mathbb R)\cdot\mathbf F_e(\mathbb Q).$$ In the following, we use the arguments in \cite{FZ22, Z19} and estimate the $m_{\mathbf H_e}$-measure of the subset $$S_{K_1,K_2,K_3,K_4,j}(U,(l/2)^{\dim V_{\beta_0}},l^{\dim V_{\beta_0}})$$ for $j\in J$ with $x_j\in\mathbf G(\mathbb R)^0\cap\mathcal K$ and for sufficiently large $l>0$. For convenience, we write $$A_l=(l/2)^{\dim V_{\beta_0}}\textup{ and }B_l=l^{\dim V_{\beta_0}}\quad(\forall l>0).$$

For any $l>1$, let $T=T(l)>0$  such that $$\beta_0(a_{T})=-\ln l.$$ Let $q$ be a rational element in $U\subset R_u(\overline{\mathbf P}_0)(\mathbb R)$. By Lemma~\ref{l24}, we may write $$q=a\cdot h\cdot m\cdot u\cdot x_k\cdot\gamma\in\mathbf H_e(\mathbb R)\cdot\mathbf P_0(\mathbb R)\cdot x_k\cdot\Gamma$$ for some $a\in\mathbf T(\mathbb R)$, $h\in\mathbf H_e(\mathbb R)$, $m\in\mathbf M_a(\mathbb R)$, $u\in R_u(\mathbf P_0)(\mathbb R)$, $x_k\in\mathcal K$ and $\gamma\in\Gamma$. Then
\begin{align*}
&q\in S_{K_1,K_2,K_3,K_4,j}(U,A_l,B_l)\\
\iff& q\textup{ rational in }U,\;A_l\leq\height(q)\leq B_l,\textrm{ and}\\
&\quad a\in\exp(\mathfrak a), \pi_{\ker(\beta_0)}(a)\in K_1, h\in K_2, m\in K_3, u\in K_4, k=j.\\
\iff& a_{T}\cdot q\Gamma\in a_{T}\cdot U\Gamma/\Gamma\textrm{ and } a_{T}\cdot q\Gamma\in\exp(\mathfrak a_{I_0,K_1})\cdot K_2\cdot K_3\cdot K_4\cdot x_j\cdot\Gamma/\Gamma
\end{align*}
where $I_0$ is the following compact interval in the Lie algebra of $\left\{a_t\right\}_{t\in\mathbb R}$ $$I_0:=\left\{x\in\mathbb\Lie(a_t): -\ln(2/{c_{x_j}^{\frac1{\dim V_{\beta_0}}}})\leq\beta_0(x)\leq-\ln(1/{c_{x_j}^{\frac1{\dim V_{\beta_0}}}})\right\}.$$ Here we use the formula $$d(q)=c_{x_j}e^{\beta_0(a)\cdot\dim V_{\beta_0}}.$$ This implies that $$a_T\cdot S_{K_1,K_2,K_3,K_4,j}(U,A_l,B_l)\Gamma/\Gamma=a_{T}\cdot U\Gamma/\Gamma\cap\exp(\mathfrak a_{I_0,K_1})K_2K_3K_4x_j\Gamma/\Gamma.$$ Since $\exp(\mathfrak a_{I_0,K_1})K_2K_3K_4$ is a compact subset in $\mathbf P_{\beta_0}(\mathbb R)$ and $$\mathbf F_e(\mathbb R)\cap\mathbf P_{\beta_0}(\mathbb R)=R_u(\overline{\mathbf P}_{\beta_0})(\mathbb R)\cap\mathbf P_{\beta_0}(\mathbb R)=\{e\},$$ there exists a small neighborhood of identity $B_{\mathbf F_e}(\delta_0)$ in $\mathbf F_e(\mathbb R)$ (for some $\delta_0>0$) such that $$B_{\mathbf F_e}(\delta_0)\times\exp(\mathfrak a_{I_0,K_1})K_2K_3K_4x_j\Gamma/\Gamma\to B_{\mathbf F_e}(\delta_0)\exp(\mathfrak a_{I_0,K_1})K_2K_3K_4x_j\Gamma/\Gamma$$ is a homeomorphism, and hence the following product map 
\begin{align}\label{eqnmap}
B_{\mathbf F_e}(\delta_0)\times a_T\cdot S_{K_1,K_2,K_3,K_4,j}(U,A_l,B_l)\Gamma/\Gamma\to B_{\mathbf F_e}(\delta_0)\cdot a_T\cdot S_{K_1,K_2,K_3,K_4,j}(U,A_l,B_l)\Gamma/\Gamma
\end{align}
is bijective. Consequently, for any $p,q\in S_{K_1,K_2,K_3,K_4,j}(U,A_l,B_l)$, the subsets $B_{\mathbf F_e}(\delta_0,T)\cdot p\Gamma$ and $B_{\mathbf F_e}(\delta_0,T)\cdot q\Gamma$ are disjoint, where $$B_{\mathbf F_e}(\delta_0,T):=a_{-T}\cdot B_{\mathbf F_e}(\delta_0)\cdot a_{T}.$$ 

Now we estimate the $m_{\mathbf H_e}$-measure of $S_{K_1,K_2,K_3, K_4,j}(U,A_l,B_l)$, following the arguments in \cite[\S4]{FZ22}. First we prove an upper bound for $m_{\mathbf H_e}(S_{K_1,K_2,K_3,K_4,j}(U,A_l,B_l))$. Fix a sufficiently small number $0<\epsilon<\delta_0$ such that $$\mu_{R_u(\overline{\mathbf P}_0)}(U)\leq\mu_{R_u(\overline{\mathbf P}_0)}(B_{\mathbf F_e}(\epsilon)\cdot U)\leq 2\mu_{R_u(\overline{\mathbf P}_0)}(U).$$ Then for sufficiently large $l>0$, we have $$B_{\mathbf F_e}(\epsilon,T)\subset  B_{\mathbf F_e}(\epsilon)\textup{ and }B_{\mathbf F_e}(\epsilon,T)\cdot S_{K_1,K_2,K_3,K_4,j}(U,A_l,B_l)\Gamma\subset B_{\mathbf F_e}(\epsilon)\cdot U\Gamma/\Gamma.$$ Since the map~\eqref{eqnmap} is bijective, the following product map $$B_{\mathbf F_e}(\delta_0,T)\times S_{K_1,K_2,K_3,K_4,j}(U,A_l,B_l)\Gamma/\Gamma\to B_{\mathbf F_e}(\delta_0,T)\cdot S_{K_1,K_2,K_3,K_4,j}(U,A_l,B_l)\Gamma/\Gamma$$ is also bijective. By Proposition~\ref{p31} and the fact that $U$ projects injectively into the quotient space $R_u(\overline{\mathbf P}_0)(\mathbb R)/R_u(\overline{\mathbf P}_0)(\mathbb R)\cap\Gamma$, we deduce that 
\begin{align*}
&m_{\mathbf H_e}(S_{K_1,K_2,K_3,K_4,j}(U,A_l,B_l))\cdot\mu_{\mathbf F_e}(B_{\mathbf F_e}(\epsilon,T))\\
=&\mu_{R_u(\overline{\mathbf P}_0)}(B_{\mathbf F_e}(\epsilon,T)\cdot S_{K_1,K_2,K_3,K_4,j}(U,A_l,B_l))\\
\leq&\int_{B_{\mathbf F_e}(\epsilon)\cdot U} \chi_{B_{\mathbf F_e}(\epsilon)\exp(\mathfrak a_{I_0,K_1})K_2K_3K_4x_j\Gamma/\Gamma}(a_{T}u\Gamma)d\mu_{R_u(\overline{\mathbf P}_0)}(u)\\
\asymp&\mu_{R_u(\overline{\mathbf P}_0)}(B_{\mathbf F_e}(\epsilon)\cdot U)\cdot\mu_{\mathbf G(\mathbb R)^0/\Gamma}(B_{\mathbf F_e}(\epsilon)\exp(\mathfrak a_{I_0,K_1})K_2K_3K_4x_j\Gamma/\Gamma)\\
\asymp&\mu_{R_u(\overline{\mathbf P}_0)}(U)\cdot\mu_{\mathbf F_e}(B_{\mathbf F_e}(\epsilon))
\end{align*}
as $T\to\infty$. Here the implicit constant in the last equation depends only on the parameter $\delta_0>0$ (as $0<\epsilon<\delta_0$), and hence depends only on the compact subsets $K_i$'s. Note that 
\begin{align*}
\mu_{\mathbf F_e}(B_{\mathbf F_e}(\epsilon,T))=&\mu_{\mathbf F_e}(B_{\mathbf F_e}(\epsilon))\cdot e^{-\sum_{\alpha\in\Phi(\mathbf F_e)}\alpha(a_T)}\\
=&\mu_{\mathbf F_e}(B_{\mathbf F_e}(\epsilon))\cdot l^{\sum_{\alpha\in\Phi(\mathbf F_e)}\alpha(a_T)/\beta_0(a_T)}\\
=&\mu_{\mathbf F_e}(B_{\mathbf F_e}(\epsilon))\cdot l^{\sum_{\alpha\in\Phi(\mathbf F_e)}\alpha(a_1)/\beta_0(a_1)}.
\end{align*}
So for sufficiently large $l>0$ we have $$m_{\mathbf H_e}(S_{K_1,K_2,K_3,K_4,j}(U,A_l,B_l))\lesssim l^{-\sum_{\alpha\in\Phi(\mathbf F_e)}\alpha(a_1)/\beta_0(a_1)}\cdot\mu_{R_u(\overline{\mathbf P}_0)}(U)$$ where the implicit constant depends only on $K_1$, $K_2$, $K_3$, $K_4$, $I_0$, $\mathbf G$ and $\Gamma$. Similarly, using the arguments in \cite[\S4]{FZ22}, we can prove a lower bound for $m_{\mathbf H_e}(S_{K_1,K_2,K_3,K_4,j}(U,A_l,B_l))$ and for any sufficiently large $l>0$ we have $$m_{\mathbf H_e}(S_{K_1,K_2,K_3,K_4,j}(U,A_l,B_l))\gtrsim l^{-\sum_{\alpha\in\Phi(\mathbf F_e)}\alpha(a_1)/\beta_0(a_1)}\cdot\mu_{R_u(\overline{\mathbf P}_0)}(U).$$ To sum up, we have the following

\begin{proposition}\label{p33}
Let $U$ be a small open bounded subset in $R_u(\overline{\mathbf P}_0)(\mathbb R)$ which projects injectively into $R_u(\overline{\mathbf P}_0)(\mathbb R)/R_u(\overline{\mathbf P}_0)(\mathbb R)\cap\Gamma$. Then for any $j\in J$ with $x_j\in\mathcal K\cap\mathbf G(\mathbb R)^0$, we have $$m_{\mathbf H_e}(S_{K_1,K_2,K_3,K_4,j}(U,A_l,B_l))\asymp l^{-\sum_{\alpha\in\Phi(\mathbf F_e)}\alpha(a_1)/\beta_0(a_1)}\cdot\mu_{R_u(\overline{\mathbf P}_0)}(U)$$ as $l\to\infty$. Consequently, we have $$m_{\mathbf H_e}(S_{K_1,K_2,K_3,K_4}(U,A_l,B_l))\asymp l^{-\sum_{\alpha\in\Phi(\mathbf F_e)}\alpha(a_1)/\beta_0(a_1)}\cdot\mu_{R_u(\overline{\mathbf P}_0)}(U).$$ Here the implicit constants depend only on the compact subsets $K_i$'s, $\mathbf G$ and $\Gamma$.
\end{proposition}

\section{Proof of Theorem~\ref{mthm11}: The case where $0\leq\tau(\psi)<\beta_0(a_{-1})$ and the flow $\{a_t\}_{t\in\mathbb R}$ is mixing}\label{lbd1}
In this section, we prove Theorem~\ref{mthm11} under the assumption that $0\leq\tau(\psi)<\beta_0(a_{-1})$ and the action of $\{a_t\}_{t\in\mathbb R}$ on $\mathbf G(\mathbb R)^0/\Gamma$ is mixing. In the following, we fix some compact subsets $K_1\subset\ker\beta_0$, $K_2\subset\mathbf H_e(\mathbb R)$, $K_3\subset\mathbf M_a(\mathbb R)\cap\mathbf G(\mathbb R)^0$ and $K_4\subset R_u(\mathbf P_0)(\mathbb R)$.

\begin{lemma}\label{l41}
Fix $j\in J$ with $x_j\in\mathbf G(\mathbb R)^0$. Let $U$ be a small open bounded subset in $R_u(\overline{\mathbf P}_0)(\mathbb R)$ and $$\beta_0(a_T)=-\ln l$$ for some $T>0$ and $l>1$. Let $\mathcal F_q=\mathbf H_e(\mathbb R)\cdot q$ be the leaf through $q\in\mathbf F_e(\mathbb Q)$ such that $$\mathcal F_q\cap S_{K_1,K_2,K_3,K_4,j}(U,A_l, B_l)\neq\emptyset.$$ Then there exist $\theta_1>0$ and $\theta_2>0$ such that for any $p\in\mathcal F_q\cap S_{K_1,K_2,K_3,K_4,j}(U,A_l, B_l)$ $$B_{\mathbf H_e}(\theta_1,T)\cdot p\cap U\subset\mathcal F_q\cap S_{K_1,\tilde K_2,K_3,K_4,j}(U,A_l, B_l)$$ where $\tilde K_2=B_{\mathbf H_e}(\theta_2)\cdot K_2$ and $B_{\mathbf H_e}(\theta_1,T)=a_{-T}\cdot B_{\mathbf H_e}(\theta_1)\cdot a_{T}$. Here the constants $\theta_1$ and $\theta_2$ depend only on $K_i$'s, $\mathbf G$ and $\Gamma$.
\end{lemma}
\begin{proof}
By the discussion in \S\ref{counting}, we know that $$a_T\cdot S_{K_1,K_2,K_3,K_4,j}(U,A_l,B_l)\Gamma=a_{T}\cdot U\Gamma/\Gamma\cap\exp(\mathfrak a_{I_0,K_1})K_2K_3K_4x_j\Gamma/\Gamma.$$ Choose $\theta_1,\theta_2>0$ sufficiently small so that for any $a\in\exp(\mathfrak a_{I_0,K_1})$ we have $$a^{-1}B_{\mathbf H_e}(\theta_1)a\subset B_{\mathbf H_e}(\theta_2).$$ Now for any $p\in\mathcal F_q\cap S_{K_1,K_2,K_3,K_4,j}(U,A_l, B_l)$, there exist $a\in\exp(\mathfrak a_{I_0,K_1})$, $h\in K_2$, $m\in K_3$, $u\in K_4$ and $\gamma\in\Gamma$ such that $$a_T\cdot p=a\cdot h\cdot m\cdot u\cdot x_j\cdot\gamma.$$ Then $$a_T\cdot B_{\mathbf H_e}(\theta_1,T)\cdot p=B_{\mathbf H_e}(\theta_1)\cdot a\cdot h\cdot m\cdot u\cdot x_j\cdot\gamma\subset\exp(\mathfrak a_{I_0,K_1})\tilde K_2K_3K_4x_j\Gamma$$ where $\tilde K_2=B_{\mathbf H_e}(\theta_2)\cdot K_2$. By definition, we have $$B_{\mathbf H_e}(\theta_1,T)\cdot p\cap U\subset S_{K_1,\tilde K_2,K_3,K_4,j}(U,A_l, B_l).$$ This completes the proof of the lemma.
\end{proof}

Let $X$ be a Riemannian manifold, $m$ a volume form on $X$ and $E$ a compact subset of $X$. Denote by $\diam(S)$ the diameter of a set $S\subset X$. A collection $\mathcal A$ of compact subsets of $E$ is said to be tree-like if $\mathcal A$ is the union of finite sub-collections $\mathcal A_k$ $(k\in\mathbb N)$ such that\vspace{2mm}
\begin{enumerate}
\item $\mathcal A_0=\left\{E\right\}$;\vspace{2mm}
\item For any $k$ and $S_1,S_2\in\mathcal A_k$, either $S_1=S_2$ or $S_1\cap S_2=\emptyset$;\vspace{2mm}
\item For any $k$ and $S_1\in\mathcal A_{k+1}$, there exists $S_2\in\mathcal A_k$ such that $S_1\subset S_2$; \vspace{2mm}
\item $d_k(\mathcal A):=\sup_{S\in\mathcal A_k}\diam(S)\to0$ as $k\to\infty$.\vspace{2mm}
\end{enumerate}
We write $\mathbf A_k:=\bigcup_{A\in\mathcal A_k}A$ and define a Cantor-type subset $\mathbf A_\infty:=\bigcap_{k\in\mathbb N}\mathbf A_k$. We also define $$\Delta_k(\mathcal A):=\inf_{S\in\mathcal A_k}\frac{m(\mathbf A_{k+1}\cap S)}{m(S)}.$$ We will use the following result to compute the lower bound of $\dim_HS_\rho(\psi)^c$ in Theorem~\ref{mthm11}.

\begin{theorem}[\cite{KM96,M87,U91}]\label{thm43}
Let $(X,m)$ be a Riemannian manifold, where $m$ is the volume form on $X$. Then for any tree-like collection $\mathcal A$ of subsets of $E$ $$\dim_H(\mathbf A_\infty)\geq\dim_{H} X-\limsup_{k\to\infty}\frac{\sum_{i=0}^k\log(\Delta_i(\mathcal A))}{\log(d_{k+1}(\mathcal A))}.$$ 
\end{theorem}

\begin{proof}[Proof of Theorem~\ref{mthm11} assuming that $0\leq\tau(\psi)<\beta_0(a_{-1})$ and the flow $\{a_t\}_{t\in\mathbb R}$ is mixing]
We first get a lower bound for the Hausdorff dimension of $S_\rho(\psi)^c\cap R_u(\overline{\mathbf P}_0)(\mathbb R)$. We fix some compact subsets $K_1\subset\ker(\beta_0),K_2\subset\mathbf H_e(\mathbb R),K_3\subset\mathbf M_a(\mathbb R)\cap\mathbf G(\mathbb R)^0$ and $K_4\subset R_u(\mathbf P_0)(\mathbb R)$ as in the begining of this section, and also fix some index $j\in J$ with $x_j\in\mathcal K\cap\mathbf G(\mathbb R)^0$. Assume that $0\leq\tau=\tau(\psi)<\beta_0(a_{-1})$. Let $\epsilon>0$ be a sufficiently small number such that $\tau+\epsilon<\beta_0(a_{-1})$. Let $\nu_0$ be any $\mathbb Q$-root relative to $\mathbf T$ in $R_u(\overline{\mathbf P}_0)$ such that $$\nu_0(a_1)=\max\{\alpha(a_1):\alpha\in\Phi(R_u(\overline{\mathbf P}_0))\}.$$ Since the unstable horospherical subgroup of $\{a_t\}_{t\in\mathbb R}$ is contained in $R_u(\overline{\mathbf P}_0)$, we have $\nu_0(a_1)>0$.

In the following, we construct a tree-like collection $\mathcal A=\{\mathcal A_k\}_{k\in\mathbb N}$ of compact subsets in $R_u(\overline{\mathbf P}_0)(\mathbb R)$. We start with a small open bounded box $U_0$ in $R_u(\overline{\mathbf P}_0)(\mathbb R)$ which projects injectively into $R_u(\overline{\mathbf P}_0)(\mathbb R)/(R_u(\overline{\mathbf P}_0)(\mathbb R)\cap\Gamma)$.

For $k=0$, we set $\mathcal A_0=\left\{U_0\right\}$. 

For $k=1$, we choose sufficiently large numbers $l_1>0$ and $T_1>0$ such that $$\beta_0(a_{T_1})=-\ln l_1.$$ By Proposition~\ref{p33}, we know that $$m_{\mathbf H_e}(S_{K_1,K_2,K_3,K_4,j}(U_0,A_{l_1}, B_{l_1}))\asymp l_1^{-\sum_{\alpha\in\Phi(\mathbf F_e)}\alpha(a_1)/\beta_0(a_1)}\cdot\mu_{R_u(\overline{\mathbf P}_0)}(U_0)$$ and $$m_{\mathbf H_e}(S_{K_1,\tilde K_2,K_3,K_4,j}(U_0,A_{l_1}, B_{l_1}))\asymp l_1^{-\sum_{\alpha\in\Phi(\mathbf F_e)}\alpha(a_1)/\beta_0(a_1)}\cdot\mu_{R_u(\overline{\mathbf P}_0)}(U_0)$$ where $\tilde K_2=B_{\mathbf H_e}(\theta_2)\cdot K_2$ for some $\theta_2>0$ defined in Lemma~\ref{l41}. Now for any leaf $\mathcal F_q=\mathbf H_e(\mathbb R)\cdot q$ ($q\in\mathbf F_e(\mathbb Q)$) in the foliation $\mathcal F_{\mathbf H_e}$ such that $$\mathcal F_q\cap S_{K_1,K_2,K_3,K_4,j}(U_0,A_{l_1}, B_{l_1})\neq\emptyset,$$ we divide the region $\mathcal F_q\cap U_0$ into small cubes of side length $$\theta_1\cdot\exp(-\nu_{0}(a_{T_1}))/10$$ where $\theta_1>0$ is the constant defined in Lemma~\ref{l41}. Then we collect those cubes $R\subset\mathcal F_q\cap U_0$ which intersect $S_{K_1,K_2,K_3,K_4,j}(U_0,A_{l_1}, B_{l_1})$, and denote the corresponding collection by $$\mathcal G_{1,q}:=\{R: R\subset\mathcal F_q\cap U_0,\;R\cap S_{K_1,K_2,K_3,K_4,j}(U_0,A_{l_1}, B_{l_1})\neq\emptyset\}.$$ Note that $\theta_1\cdot\exp(-\nu_0(a_{T_1}))/10$ is smaller than the minimum side length of the rectangle $B_{\mathbf H_e}(\theta_1, T_1)$, so by Lemma~\ref{l41}, we know that each cube in $\mathcal G_{1,q}$ is contained in $$S_{K_1,\tilde K_2,K_3,K_4,j}(U_0,A_{l_1}, B_{l_1})$$ where $\tilde K_2=B_{\mathbf H_e}(\theta_2)\cdot K_2$. For any leaf $\mathcal F_q=\mathbf H_e(\mathbb R)\cdot q$ $(q\in\mathbf F_e(\mathbb Q))$ in $\mathcal F_{\mathbf H_e}$ such that $$\mathcal F_q\cap S_{K_1,K_2,K_3,K_4,j}(U_0,A_{l_1}, B_{l_1})=\emptyset,$$ we set $\mathcal G_{1,q}=\emptyset$. 

Now let $$\mathcal H_1=\bigcup_{q\in\mathbf F_e(\mathbb Q)}\bigcup_{R\in\mathcal G_{1,q}} R.$$ Then we have 
\begin{align*}
& l_1^{-\sum_{\alpha\in\Phi(\mathbf F_e)}\alpha(a_1)/\beta_0(a_1)}\cdot\mu_{R_u(\overline{\mathbf P}_0)}(U_0)\\
&\asymp m_{\mathbf H_e}(S_{K_1,K_2,K_3,K_4,j}(U_0,A_{l_1}, B_{l_1}))\\
&\leq m_{\mathbf H_e}(\bigcup_{q\in\mathbf F_e(\mathbb Q)}\bigcup_{R\in\mathcal G_{1,q}} R )\\
&\leq m_{\mathbf H_e}(S_{K_1,\tilde K_2,K_3,K_4,j}(U_0,A_{l_1}, B_{l_1}))\\
&\asymp l_1^{-\sum_{\alpha\in\Phi(\mathbf F_e)}\alpha(a_1)/\beta_0(a_1)}\cdot\mu_{R_u(\overline{\mathbf P}_0)}(U_0)
\end{align*}
and hence $$m_{\mathbf H_e}(\mathcal H_1 )\asymp l_1^{-\sum_{\alpha\in\Phi(\mathbf F_e)}\alpha(a_1)/\beta_0(a_1)}\cdot\mu_{R_u(\overline{\mathbf P}_0)}(U_0).$$ Note that each cube $R$ in $\mathcal H_1$ is contained in $S_{K_1,\tilde K_2,K_3,K_4,j}(U_0,A_{l_1}, B_{l_1})$. Let $$t_1=\frac{\ln l_1}{\beta_0(a_{-1})-(\tau+\epsilon)}.$$ By the computations in \S\ref{counting}, we can choose a sufficiently small number $\tilde\delta_0>0$ such that $$B_{\mathbf F_e}(\tilde\delta_0)\times\exp(\mathfrak a_{I_0,K_1})\tilde K_2K_3K_4x_j\Gamma/\Gamma\to B_{\mathbf F_e}(\tilde\delta_0)\exp(\mathfrak a_{I_0,K_1})\tilde K_2K_3K_4x_j\Gamma/\Gamma$$ is a homeomorphism, and the subsets in the collection 
\begin{align*}
\mathcal F_1(U_0):=\left\{\left(a_{-t_1}\cdot B_{\mathbf F_e}(\tilde{\delta}_0)\cdot a_{t_1}\right)\cdot q: q\in\mathcal H_1\right\}
\end{align*}
are disjoint by this homeomorphism as $$a_{-t_1}\cdot B_{\mathbf F_e}(\tilde{\delta}_0)\cdot a_{t_1}\subset B_{\mathbf F_e}(\tilde\delta_0,T_1)\textup{ and }a_{T_1}\cdot q\Gamma\in\exp(\mathfrak a_{K_1, I_0})\tilde K_2K_3K_4x_j\Gamma/\Gamma\;(\forall q\in\mathcal H_1).$$ We write $$\Phi(\mathbf F_e)=\Phi^0(\mathbf F_e)\cup\Phi^1(\mathbf F_e)$$ where $$\Phi^0(\mathbf F_e)=\{\alpha\in\Phi(\mathbf F_e):\alpha(a_1)=0\}\textup{ and }\Phi^1(\mathbf F_e)=\{\alpha\in\Phi(\mathbf F_e):\alpha(a_1)\neq0\}.$$ We denote by $$\mathcal P_1=\bigcup_{E\in\mathcal F_1(U_0)} E.$$ Now we can divide the subset $\mathcal P_1$ into cubes of side length $$\tilde\delta_0\cdot l_1^{\frac{\nu_0(a_{-1})}{\beta_0(a_{-1})-(\tau+\epsilon)}}$$ which is smaller than $\theta_1\cdot\exp(-\nu_0(a_{T_1}))/10$ and the minimum side length of the rectangle $$a_{-t_1}\cdot B_{\mathbf F_e}(\tilde{\delta}_0)\cdot a_{t_1}$$ if $T_1>0$ and $l_1>0$ are sufficiently large. Note that the set $\Phi^0(\mathbf F_e)$ may not be empty and the diameter of the set $$a_{-t_1}\cdot B_{\mathbf F_e}(\tilde{\delta}_0)\cdot a_{t_1}$$ may be larger than the diameter of $U_0$, so some cubes we obtain here (by dividing the subset $\mathcal P_1$) may be outside the set $U_0$. For our purpose, we collect only those cubes which are inside the subset $U_0$. In this manner, we obtain a family of disjoint cubes inside $U_0$ constructed from $\mathcal P_1$, which we denote by $\mathcal A_1$. 

We remark here that if $\Phi^0(\mathbf F_e)=\emptyset$, then all the cubes we obtain by dividing $\mathcal P_1$ are inside the subset $U_0$ if $T_1>0$ and $l_1>0$ are chosen to be sufficiently large (and also if we enlarge $U_0$ slightly to avoid the complexity caused by the boundary of $U_0$). We will see later that in the computations there are no significant differences between the case $\Phi^0(\mathbf F_e)\neq\emptyset$ and the case $\Phi^0(\mathbf F_e)=\emptyset$. Indeed, when we apply Theorem~\ref{thm43} and calculate $\Delta_i(\mathcal A)$ and $d_i(\mathcal A)$, the differences between these two cases may affect the value of the formula in finite steps, but eventually when we take the limit, these differences will disappear since we keep choosing sufficiently large numbers $T_1,l_1,T_2,l_2,T_3,l_3,\dots$ to offset the effects by the differences in the early stages.

In a similar way, we can construct $\mathcal A_k$ for any $k\in\mathbb N$ inductively. For $k>1$ we choose sufficiently large numbers $l_k>0$ and $T_k>0$ such that $\beta_0(a_{T_{k}})=-\ln l_k$. For each cube $S\in\mathcal A_{k-1}$, by Proposition~\ref{p33}, we know that $$m_{\mathbf H_e}(S_{K_1,K_2,K_3,K_4,j}(S,A_{l_k}, B_{l_k}))\asymp l_k^{-\sum_{\alpha\in\Phi(\mathbf F_e)}\alpha(a_1)/\beta_0(a_1)}\cdot\mu_{R_u(\overline{\mathbf P}_0)}(S)$$ and $$m_{\mathbf H_e}(S_{K_1,\tilde K_2,K_3,K_4,j}(S,A_{l_k}, B_{l_k}))\asymp l_k^{-\sum_{\alpha\in\Phi(\mathbf F_e)}\alpha(a_1)/\beta_0(a_1)}\cdot\mu_{R_u(\overline{\mathbf P}_0)}(S)$$ where $\tilde K_2=B_{\mathbf H_e}(\theta_2)\cdot K_2$ for $\theta_2>0$ as defined in Lemma~\ref{l41}. Now for any leaf $\mathcal F_q=\mathbf H_e(\mathbb R)\cdot q$ ($q\in\mathbf F_e(\mathbb Q)$) in the foliation $\mathcal F_{\mathbf H_e}$ such that $$\mathcal F_q\cap S_{K_1,K_2,K_3,K_4,j}(S,A_{l_k}, B_{l_k})\neq\emptyset,$$ we divide the region $\mathcal F_q\cap S$ into small cubes of side length $$\theta_1\cdot\exp(-\nu_0(a_{T_k}))/10$$ where $\theta_1>0$ is the constant defined in Lemma~\ref{l41}. Then we collect those cubes $R\subset\mathcal F_q\cap S$ which intersect $S_{K_1,K_2,K_3,K_4,j}(S,A_{l_k}, B_{l_k})$, and denote the corresponding collection by $$\mathcal G_{k,q,S}=\{R: R\subset\mathcal F_q\cap S,\; R\cap S_{K_1,K_2,K_3,K_4,j}(S,A_{l_k}, B_{l_k})\neq\emptyset\}.$$ Note that $\theta_1\cdot\exp(-\nu_0(a_{T_k}))/10$ is smaller than the minimum side length of the rectangle $B_{\mathbf H_e}(\theta_1, T_k)$, so by Lemma~\ref{l41}, we know that each cube in $\mathcal G_{k,q,S}$ is contained in $$S_{K_1,\tilde K_2,K_3,K_4,j}(S,A_{l_k}, B_{l_k})$$ where $\tilde K_2=B_{\mathbf H_e}(\theta_2)\cdot K_2$. For any leaf $\mathcal F_q=\mathbf H_e(\mathbb R)\cdot q$ $(q\in\mathbf F_e(\mathbb Q))$ in $\mathcal F_{\mathbf H_e}$ such that $$\mathcal F_q\cap S_{K_1,K_2,K_3,K_4,j}(S,A_{l_k}, B_{l_k})=\emptyset,$$ we set $\mathcal G_{k,q,S}=\emptyset$.

Let $$\mathcal H_{k,S}=\bigcup_{q\in\mathbf F_e(\mathbb Q)}\bigcup_{R\in\mathcal G_{k,q,S}} R.$$ Then we have 
\begin{align*}
& l_k^{-\sum_{\alpha\in\Phi(\mathbf F_e)}\alpha(a_1)/\beta_0(a_1)}\cdot\mu_{R_u(\overline{\mathbf P}_0)}(S)\\
&\asymp m_{\mathbf H_e}(S_{K_1,K_2,K_3,K_4,j}(S,A_{l_k}, B_{l_k}))\\
&\leq m_{\mathbf H_e}(\bigcup_{q\in\mathbf F_e(\mathbb Q)}\bigcup_{R\in\mathcal G_{k,q,S}} R )\\
&\leq m_{\mathbf H_e}(S_{K_1,\tilde K_2,K_3,K_4,j}(S,A_{l_k}, B_{l_k}))\\
&\asymp l_k^{-\sum_{\alpha\in\Phi(\mathbf F_e)}\alpha(a_1)/\beta_0(a_1)}\cdot\mu_{R_u(\overline{\mathbf P}_0)}(S)
\end{align*}
and hence $$ m_{\mathbf H_e}(\mathcal H_{k,S})\asymp l_k^{-\sum_{\alpha\in\Phi(\mathbf F_e)}\alpha(a_1)/\beta_0(a_1)}\cdot\mu_{R_u(\overline{\mathbf P}_0)}(S).$$ Note that each cube in $\mathcal H_{k,S}$ is contained in $S_{K_1,\tilde K_2,K_3,K_4,j}(S,A_{l_k}, B_{l_k})$. Let $$t_k=\frac{\ln l_k}{\beta_0(a_{-1})-(\tau+\epsilon)}.$$ By the computations in \S\ref{counting}, we can choose a sufficiently small number $\tilde\delta_0>0$ such that $$B_{\mathbf F_e}(\tilde\delta_0)\times\exp(\mathfrak a_{I_0,K_1})\tilde K_2K_3K_4x_j\Gamma/\Gamma\to B_{\mathbf F_e}(\tilde\delta_0)\exp(\mathfrak a_{I_0,K_1})\tilde K_2K_3K_4x_j\Gamma/\Gamma$$ is a homeomorphism, and the subsets in the collection 
\begin{align*}
\mathcal F_{k}(S):=\left\{(a_{-t_k}\cdot B_{\mathbf F_e}(\tilde{\delta}_0)\cdot a_{t_k})\cdot q: q\in\mathcal H_{k,S}\right\}
\end{align*}
are disjoint by this homeomorphism as $$a_{-t_k}\cdot B_{\mathbf F_e}(\tilde{\delta}_0)\cdot a_{t_k}\subset B_{\mathbf F_e}(\tilde\delta_0,T_k)\textup{ and }a_{T_k}\cdot q\Gamma\in\exp(\mathfrak a_{I_0,K_1})\tilde K_2K_3K_4x_j\Gamma/\Gamma\;(\forall q\in\mathcal H_{k,S}).$$ We denote by $$\mathcal P_{k,S}=\bigcup_{E\in\mathcal F_k(S)} E.$$ Now we can divide the subset $\mathcal P_{k,S}$ into cubes of side length $$\tilde\delta_0\cdot l_k^{\frac{\nu_0(a_{-1})}{\beta_0(a_{-1})-(\tau+\epsilon)}}$$ which is smaller than $\theta_1\cdot\exp(-\nu_0(a_{T_k}))/10$ and the minimum side length of the rectangle $$a_{-t_k}\cdot B_{\mathbf F_e}(\tilde{\delta}_0)\cdot a_{t_k}$$ if $T_k>0$ and $l_k>0$ are sufficiently large. Note that $\Phi^0(\mathbf F_e)$ may not be empty, and some of the cubes in $\mathcal P_{k,S}$ may be outside $S$, and here we collect only those cubes in $\mathcal P_{k,S}$ which are inside $S$. Thus, we obtain a family of disjoint cubes inside $S$ constructed from $\mathcal P_{k,S}$. We denote by $\mathcal A_k$ the collection of all the disjoint cubes inside $S$ constructed from $\mathcal P_{k,S}$ where $S$ ranges over all elements in $\mathcal A_{k-1}$. 

In this manner, we obtain an increasing sequence of sufficiently large numbers $\left\{l_k\right\}_{k\in\mathbb N}$ and a tree-like collection $\mathcal A=\left\{\mathcal A_k\right\}_{k\in\mathbb N}$ of cubes in $U_0$. Using the notation in Theorem~\ref{thm43}, we have $$d_k(\mathcal A)\asymp l_k^{\frac{\nu_0(a_{-1})}{\beta_0(a_{-1})-(\tau+\epsilon)}}$$ where $d_k(\mathcal A)$ is the diameter of the family $\mathcal A_k$. Moreover, one can compute that
$$\begin{cases} 
\Delta_k(\mathcal A)\asymp  l_{k+1}^{-\sum_{\alpha\in\Phi(\mathbf F_e)}\frac{\alpha(a_1)}{\beta_0(a_1)}}\cdot\prod_{\alpha\in\Phi(\mathbf F_e)} l_{k+1}^{\frac{\alpha(a_{-1})}{\beta_0(a_{-1})-(\tau+\epsilon)}}, & \Phi^0(\mathbf F_e)=\emptyset\\ 
\Delta_k(\mathcal A)\asymp  l_{k+1}^{-\sum_{\alpha\in\Phi(\mathbf F_e)}\frac{\alpha(a_1)}{\beta_0(a_1)}}\cdot\prod_{\alpha\in\Phi^1(\mathbf F_e)} l_{k+1}^{\frac{\alpha(a_{-1})}{\beta_0(a_{-1})-(\tau+\epsilon)}}\cdot\prod_{\alpha\in\Phi^0(\mathbf F_e)} l_k^{\frac{\nu_0(a_{-1})}{\beta_0(a_{-1})-(\tau+\epsilon)}}, & \Phi^0(\mathbf F_e)\neq\emptyset.
\end{cases}
$$
Now let $\mathbf A_\infty=\bigcap_{k\in\mathbb N}\mathbf A_k$. By Theorem~\ref{thm43}, we can compute (assuming that $l_{k+1}$ is much larger than $l_k$ for any $k\in\mathbb N$) that 
\begin{align*}\dim_H(\mathbf A_\infty)\geq& \dim_{H} X-\limsup_{k\to\infty}\frac{\sum_{i=0}^k\log(\Delta_i(\mathcal A))}{\log(d_{k+1}(\mathcal A))} \\
=&\dim R_u(\overline{\mathbf P}_0)-\sum_{\alpha\in\Phi^1(\mathbf F_e)}\frac{\frac{\alpha(a_1)}{\beta_0(a_{-1})-(\tau+\epsilon)}-\frac{\alpha(a_1)}{\beta_0(a_{-1})}}{\frac{\nu_0(a_1)}{\beta_0(a_{-1})-(\tau+\epsilon)}}\\
=&\dim R_u(\overline{\mathbf P}_0)-\sum_{\alpha\in\Phi(\mathbf F_e)}\frac{\alpha(a_1)(\tau+\epsilon)}{\beta_0(a_{-1})\nu_0(a_1)}.
\end{align*}

By the definition of $\tau=\tau(\psi)$, there exists a divergent sequence $\{s_k\}_{k\in\mathbb N}\subset\mathbb R_+$ such that $$\tau=\tau(\psi)=\lim_{k\to\infty}\left(-\frac{\ln\psi(s_k)}{s_k}\right).$$ Now we choose the sequence $\{l_k\}_{k\in\mathbb N}$ in the construction of $\mathbf A_\infty$ so that the sequence $\{t_k\}_{k\in\mathbb N}$ equals the sequence $\{s_k\}_{k\in\mathbb N}$ defined as above. Then we can prove the following

\begin{lemma}\label{l44}
We have $\mathbf A_\infty\subset S_\rho(\psi)^c\cap R_u(\overline{\mathbf P}_0)(\mathbb R)$.
\end{lemma}
\begin{proof}
By the construction, for any $p\in\mathbf A_\infty$, there exists a sequence of rational elements $$q_k\in S_{K_1,K_2,K_3,K_4,j}(S_{k-1},A_{l_k},B_{l_k}),\quad S_{k-1}\in\mathcal A_{k-1}$$ such that $$p\in(a_{-t_k}\cdot B_{\mathbf F_e}(\tilde{\delta}_0)\cdot a_{t_k})\cdot q_k$$ where $t_k=\frac{\ln l_k}{\beta_0(a_{-1})-(\tau+\epsilon)}$. Then
\begin{align*}
a_{t_k}\cdot p\in B_{\mathbf F_e}(\tilde\delta_0)\cdot (a_{t_k}\cdot q_k).
\end{align*}
Note that by definition, the height of the rational element $a_{t_k}\cdot q_k$ is equal to the co-volume of the lattice $\rho(a_{t_k}\cdot q_k)\cdot\mathbb Z^d\cap V_{\beta_0}$ in $V_{\beta_0}$ and $$\height(a_{t_k}\cdot q_k)=e^{\beta_0(a_{t_k})\cdot\dim V_{\beta_0}}\cdot d(q_k).$$ By Lemma~\ref{l26}, we deduce that for any $k\in\mathbb N$
\begin{align*}
\delta(\rho(a_{t_k}\cdot p)\mathbb Z^d)&\asymp\delta(\rho(a_{t_k}\cdot q_k)\mathbb Z^d)\lesssim \height(a_{t_k}\cdot q_k)^{\frac1{\dim V_{\beta_0}}}\\
&\lesssim e^{\beta_0(a_{t_k})}\cdot l_k\leq e^{-(\tau+\epsilon)t_k}.
\end{align*}
Note that by our choice of $\{t_k\}_{k\in\mathbb N}$, for $\epsilon>0$ and for sufficiently large $k\in\mathbb N$, $$-\frac{\ln\psi(t_k)}{t_k}\leq\tau+\frac\epsilon2\implies\psi(t_k)\geq e^{-(\tau+\epsilon/2)t_k}.$$
So we have
\begin{align*}
\delta(\rho(a_{t_k}\cdot p)\mathbb Z^d)&\lesssim e^{-(\tau+\epsilon)t_k}\leq\psi(t_k)\cdot e^{-\frac\epsilon2 t_k}
\end{align*}
and by definition $p\in S_\rho(\psi)^c\cap R_u(\overline{\mathbf P}_0)(\mathbb R)$. This completes the proof of the lemma.
\end{proof}

By Lemma~\ref{l44} and the computation for $\dim_H\mathbf A_\infty$, we have $$\dim_H(S_\rho(\psi)^c\cap R_u(\overline{\mathbf P}_0)(\mathbb R))\geq\dim R_u(\overline{\mathbf P}_0)(\mathbb R)-\sum_{\alpha\in\Phi(\mathbf F_e)}\frac{\alpha(a_1)(\tau+\epsilon)}{\beta_0(a_{-1})\nu_0(a_1)}.$$ By taking $\epsilon\to0$, we obtain
\begin{align*}
\dim_H(S_\rho(\psi)^c\cap R_u(\overline{\mathbf P}_0)(\mathbb R))\geq&\dim R_u(\overline{\mathbf P}_0)(\mathbb R)-\sum_{\alpha\in\Phi(\mathbf F_e)}\frac{\alpha(a_1)\tau}{\beta_0(a_{-1})\nu_0(a_1)}.
\end{align*}
By definition, if an element $g\in\mathbf G(\mathbb R)$ belongs to $S_\rho(\psi)^c$, then for any $h\in\mathbf P_0(\mathbb R)$, $h\cdot g$ also belongs to $S_\rho(\psi)^c$.
Using the same argument as in \cite[\S10]{FZ22}, we can conclude that
\begin{align*}
\dim_H(S_\rho(\psi)^c)\geq&\dim\mathbf P_0(\mathbb R)+\dim_H(S_\rho(\psi)^c\cap R_u(\overline{\mathbf P}_0)(\mathbb R))\\
\geq&\dim\mathbf G(\mathbb R)-\sum_{\alpha\in\Phi(\mathbf F_e)}\frac{\alpha(a_1)\tau}{\beta_0(a_{-1})\nu_0(a_1)}.
\end{align*}
This completes the proof of Theorem~\ref{mthm11} in the case where $0\leq\tau(\psi)<\beta_0(a_{-1})$ and $\{a_t\}_{t\in\mathbb R}$ is mixing on $\mathbf G(\mathbb R)^0/\Gamma$.
\end{proof}

\section{Proof of Theorem~\ref{mthm11}: The other cases}\label{lbd2}
In this section, we discuss the case where $0\leq\tau(\psi)<\beta_0(a_{-1})$ and the action of $\{a_t\}_{t\in\mathbb R}$ on $\mathbf G(\mathbb R)^0/\Gamma$ is not mixing, and explain how to modify the arguments in \S\ref{counting} and \S\ref{lbd1} to prove Theorem~\ref{mthm11}. We study the case $\tau(\psi)=\beta_0(a_{-1})$ at the end of this section.

We first discuss the ergodic properties of group actions on homogeneous spaces, for which one may refer to \cite{M91,M15,R72}. Let $\mathbf G_i$ $(1\leq i\leq k)$ be the $\mathbb Q$-simple factors of $\mathbf G$. Then $\mathbf G$ is an almost direct product of $\mathbf G_i$ $(1\leq i\leq k)$. Without loss of generality, we may assume that $\{a_t\}_{t\in\mathbb R}$ projects nontrivially into $\mathbf G_i(\mathbb R)^0$ ($1\leq i\leq s$) for some $s<k$. It is known that for each $1\leq i\leq k$, any arithmetic lattice $\Gamma_i$ inside $\mathbf G_i(\mathbb Z)\cap\mathbf G_i(\mathbb R)^0$ is irreducible in $\mathbf G_i(\mathbb R)^0$ (and we fix such an arithmetic lattice $\Gamma_i$ for later use). Moreover, since $\{a_t\}_{t\in\mathbb R}$ projects nontrivially into $\mathbf G_i(\mathbb R)^0$ $(1\leq i\leq s)$, we have $\mathbf T\cap\mathbf G_i\neq\{e\}$ $(1\leq i\leq s)$. So $\mathbf G_i$ is $\mathbf Q$-isotropic, and $\mathbf G_i(\mathbb R)^0/\Gamma_i$ is not compact. By Godement compactness criterion (Cf.~\cite[Theorem 11.6]{BH62}) and \cite[6.21]{BT65}, every simple factor of $\mathbf G_i(\mathbb R)^0$ is not compact. Let $\{a_t^{i}\}_{t\in\mathbb R}$ be the projection of $\{a_t\}_{t\in\mathbb R}$ in $\mathbf G_i(\mathbb R)^0$ $(1\leq i\leq s)$. Then $\{a_t^i\}_{t\in\mathbb R}\subset \mathbf T(\mathbb R)\cap\mathbf G_i(\mathbb R)$ is a non-compact subgroup in $\mathbf G_i(\mathbb R)^0$. Using Moore's ergodicity theorem \cite{M66} and Mautner phenomenon \cite{M57,M80}, one can conclude that the action of $\{a_t^i\}_{t\in\mathbb R}$ on $\mathbf G_i(\mathbb R)^0/\Gamma_i$ is mixing $(1\leq i\leq s)$. Consequently, the action of $\{a_t\}_{t\in\mathbb R}$ on $\prod_{i=1}^s\mathbf G_i(\mathbb R)^0/\prod_{i=1}^s\Gamma_i$ is mixing. 

We denote by $$\tilde{\mathbf G}=\prod_{i=1}^s\mathbf G_i=\mathbf G_1\cdot\mathbf G_2\cdot\cdots\cdot\mathbf G_s.$$ Note that in \S\ref{pre}, \S\ref{counting} and \S\ref{lbd1}, we don't use any explicit expression of the lattice $\Gamma$ in $\mathbf G(\mathbb R)^0$. So we may choose $\Gamma$ such that $$\Gamma\cap\tilde{\mathbf G}(\mathbb R)^0=\prod_{i=1}^s\Gamma_i.$$ In the follwing, we fix the lattices $\Gamma_i$'s and $\Gamma$ and denote by $$\tilde\Gamma=\Gamma\cap\tilde{\mathbf G}(\mathbb R)^0.$$ So the action of $\{a_t\}_{t\in\mathbb R}$ on $\tilde{\mathbf G}(\mathbb R)^0/\tilde\Gamma$ is mixing.

Now all the arguments in \S\ref{pre}, \S\ref{counting} and \S\ref{lbd1} can be carried over to $\tilde{\mathbf G}$ and the homogeneous subspace $\tilde{\mathbf G}(\mathbb R)^0/\tilde\Gamma$ almost verbatim. Indeed, denote by $$\tilde{\mathbf P}_0=\mathbf P_0\cap\tilde{\mathbf G}$$ which is a minimal parabolic $\mathbb Q$-subgroup of $\tilde{\mathbf G}$, and let $\overline{\tilde{\mathbf P}}_0=\overline{\mathbf P}_0\cap\tilde{\mathbf G}$. Let $U$ be a small open bounded subset in $R_u(\overline{\tilde{\mathbf P}}_0)(\mathbb R)$ and define $$\tilde S(U,A,B)=\{q\in U: q\textup{ rational and }A\leq d(q)\leq B\}.$$ Let $\tilde{\mathbf T}=\mathbf T\cap\tilde{\mathbf G}$ which is a maximal $\mathbb Q$-split torus in $\tilde{\mathbf G}$, and $$\tilde{\mathbf H}_e=\mathbf H_e\cap\tilde{\mathbf G},\quad\tilde{\mathbf F}_e=\mathbf F_e\cap\tilde{\mathbf G}.$$ We also need results in the reduction theory about $\tilde{\mathbf G}(\mathbb R)/\tilde\Gamma$. Let $\tilde K$ be a maximal compact subgroup in $\tilde{\mathbf G}(\mathbb R)$. Let $\tilde M$ be the identity component in the unique maximal $\mathbb Q$-anisotropic subgroup in $Z_{\tilde{\mathbf G}(\mathbb R)}(\tilde{\mathbf T}(\mathbb R))$ (= the centralizer of $\tilde{\mathbf T}(\mathbb R)$ in $\tilde{\mathbf G}(\mathbb R)$). Denote by $$\tilde T_\eta=\{a\in\tilde{\mathbf T}(\mathbb R):\lambda(a)\leq\eta,\;\lambda\textup{ a simple root in }\tilde{\Delta}\}$$ where $\tilde\Delta$ is the set of positive $\mathbb Q$-simple roots in $\tilde{\mathbf G}$. A Siegel set in $\tilde{\mathbf G}(\mathbb R)$ is a subset of the form $\tilde S_{\eta,\Omega}=\tilde K\cdot\tilde T_\eta\cdot\tilde\Omega$ for some $\eta\in\mathbb R$ and a relatively compact open subset $\tilde\Omega$ containing identity in $\tilde M\cdot R_u(\tilde{\mathbf P}_0)(\mathbb R)$, and the group $\tilde{\mathbf G}(\mathbb R)$ can be written as $$\tilde{\mathbf G}(\mathbb R)=\tilde S_{\eta,\Omega}\cdot\tilde{\mathcal K}\cdot\tilde\Gamma$$ for some Siegel set $\tilde S_{\eta,\Omega}$ and some finite subset $\tilde{\mathcal K}\subset\tilde{\mathbf G}(\mathbb Q)$. Moreover, the finite set $\tilde{\mathcal K}$ satisfies the property that $$\tilde{\mathbf G}(\mathbb Q)=\tilde{\mathbf P}_0(\mathbb Q)\cdot\tilde{\mathcal K}\cdot\tilde\Gamma.$$ Denote by $\tilde{\mathcal K}=\{\tilde x_j\}_{j\in\tilde J}\subset\tilde{\mathbf G}(\mathbb Q)$ and we may assume that $e\in\tilde{\mathcal K}$.

The proof of the following lemma is similar to those of Corollary~\ref{c23} and Lemma~\ref{l24}.
\begin{lemma}\label{l51}
We have
\begin{enumerate}
\item An element $g\in R_u(\overline{\tilde{\mathbf P}}_0)(\mathbb R)$ is rational if and only if $g\in\tilde{\mathbf H}_e(\mathbb R)\cdot\tilde{\mathbf F}_e(\mathbb Q)$.
\item An element $g\in R_u(\overline{\tilde{\mathbf P}}_0)(\mathbb R)$ is rational if and only if $$g\in R_u(\overline{\tilde{\mathbf P}}_0)\cap (\tilde{\mathbf H}_e(\mathbb R)\cdot \tilde{\mathbf P}_0(\mathbb R)\cdot\tilde{\mathcal{K}}\cdot\tilde{\Gamma}).$$
\end{enumerate}
\end{lemma}

\begin{definition}\label{def52}
Let $j\in\tilde J$. A rational element $g$ in $R_u(\overline{\tilde{\mathbf P}}_0)(\mathbb R)$ is called $j$-rational if it can be written as $$g=h\cdot p\cdot \tilde x_j\cdot\gamma$$ for some $h\in\tilde{\mathbf H}_e(\mathbb R)$, $p\in\tilde{\mathbf P}_0(\mathbb R)$ and $\gamma\in\tilde\Gamma$.
\end{definition}

Let $\tilde{\mathfrak a}$ be the Lie algebra of $\tilde{\mathbf T}(\mathbb R)$. For any $\tilde x_j\in\tilde{\mathcal K}$, any compact subset $K_1$ in $\ker(\beta_0)\cap\tilde{\mathfrak a}$, any compact subset $K_2$ in $\tilde{\mathbf H}_e(\mathbb R)$, any compact subset $K_3\subset\mathbf M_a(\mathbb R)\cap\tilde{\mathbf G}(\mathbb R)^0$ and any compact subset $K_4\subset R_u(\tilde{\mathbf P}_0)(\mathbb R)$, we define $$\tilde S_{K_1,K_2,K_3,K_4,j}(U,A,B)$$ to be the set of all rational elements $q$ in $U$ such that
\begin{enumerate}
\item $A\leq d(q)\leq B$ and $q$ is $j$-rational for $\tilde x_j\in\tilde C$; 
\item $q=a\cdot h\cdot m\cdot u\cdot\tilde x_j\cdot\gamma$ for some $a\in\exp(\tilde{\mathfrak a})$, $\pi_{\ker(\beta_0)}(a)\in K_1$, $h\in K_2$, $m\in K_3$, $u\in K_4$ and $\gamma\in\tilde\Gamma$.
\end{enumerate}
Note that $\tilde S_{K_1,K_2,K_3,K_4,j}(U,A,B)\neq\emptyset$ implies that $\tilde x_j\in\tilde{\mathbf G}(\mathbb R)^0$.

We fix a Haar measure $\mu_{\tilde{\mathbf H}_e}$ on $\tilde{\mathbf H}_e(\mathbb R)$ and a Haar measure $\mu_{\tilde{\mathbf F}_e}$ on $\tilde{\mathbf F}_e(\mathbb R)$. Then for any $q\in\tilde{\mathbf F}_e(\mathbb Q)$, we define $m_{\tilde{\mathbf H}_eq}$ to be the locally finite measure defined on $R_u(\overline{\tilde{\mathbf P}}_0)(\mathbb R)$ which is supported on $\tilde{\mathbf H}_e(\mathbb R)\cdot q$ and induced by $\mu_{\tilde{\mathbf H_e}}$ via the product map $$\tilde{\mathbf H}_e(\mathbb R)\times\{q\}\to\tilde{\mathbf H}_e(\mathbb R)\cdot q\subset R_u(\overline{\tilde{\mathbf P}}_0)(\mathbb R).$$ We define $$m_{\tilde{\mathbf H}_e}:=\sum_{q\in\tilde{\mathbf F}_e(\mathbb Q)}m_{\tilde{\mathbf H}_eq}.$$

Then using the same argument as in Proposition~\ref{p33}, we have
\begin{proposition}\label{p53}
Let $U$ be a small open bounded subset in $R_u(\overline{\tilde{\mathbf P}}_0)(\mathbb R)$ which projects injectively into $R_u(\overline{\tilde{\mathbf P}}_0)(\mathbb R)/(R_u(\overline{\tilde{\mathbf P}}_0)(\mathbb R)\cap\tilde\Gamma)$. Then for any $j\in\tilde J$ with $\tilde x_j\in\tilde{\mathbf G}(\mathbb R)^0$ and any sufficiently large $l>0$, we have $$m_{\tilde{\mathbf H}_e}(\tilde S_{K_1,K_2,K_3,K_4,j}(U,A_l,B_l))\asymp l^{-\sum_{\alpha\in\Phi(\tilde{\mathbf F}_e)}\alpha(a_1)/\beta_0(a_1)}\cdot\mu_{R_u(\overline{\tilde{\mathbf P}}_0)}(U)$$ and $$m_{\tilde{\mathbf H}_e}(\tilde S_{K_1,K_2,K_3,K_4}(U,A_l,B_l))\asymp l^{-\sum_{\alpha\in\Phi(\tilde{\mathbf F}_e)}\alpha(a_1)/\beta_0(a_1)}\cdot\mu_{R_u(\overline{\tilde{\mathbf P}}_0)}(U).$$ Here the implicit constants depend only on the compact subsets $K_i$'s, $\tilde{\mathbf G}$ and $\tilde{\Gamma}$.
\end{proposition}

The analogue of Lemma~\ref{l41} holds as well in $\tilde{\mathbf G}(\mathbb R)^0/\tilde\Gamma$. 
\begin{lemma}\label{l54}
Fix $j\in\tilde J$ with $\tilde x_j\in\tilde{\mathbf G}(\mathbb R)^0$. Let $U$ be an open bounded subset in $R_u(\overline{\tilde{\mathbf P}}_0)(\mathbb R)$ and $$\beta_0(a_T)=-\ln l$$ for some $T>0$ and $l>1$. Let $\tilde{\mathcal F}_q=\tilde{\mathbf H}_e(\mathbb R)\cdot q$ be the leaf through $q\in\tilde{\mathbf F}_e(\mathbb Q)$ such that $$\tilde{\mathcal F}_q\cap\tilde S_{K_1,K_2,K_3,K_4,j}(U,A_l, B_l)\neq\emptyset.$$ Then there exist $\theta_1>0$ and $\theta_2>0$ such that for any $p\in\tilde{\mathcal F}_q\cap\tilde S_{K_1,K_2,K_3,K_4,j}(U,A_l, B_l)$ $$B_{\tilde{\mathbf H}_e}(\theta_1,T)\cdot p\cap U\subset\tilde{\mathcal F}_q\cap \tilde S_{K_1,\tilde K_2,K_3,K_4,j}(U,A_l, B_l)$$ where $\tilde K_2=B_{\tilde{\mathbf H}_e}(\theta_2)\cdot K_2$ and $B_{\tilde{\mathbf H}_e}(\theta_1,T)=a_{-T}\cdot B_{\tilde{\mathbf H}_e}(\theta_1)\cdot a_{T}$. Here the constants $\theta_1$ and $\theta_2$ depend only on $K_i$'s, $\tilde{\mathbf G}$ and $\tilde\Gamma$.
\end{lemma}


Consequently, using the same argument as in \S\ref{lbd1}, we can conclude that for any $0\leq\tau(\psi)<\beta_0(a_{-1})$ $$\dim_H(S_\rho(\psi)^c\cap\tilde{\mathbf G}(\mathbb R))\geq\dim\tilde{\mathbf G}-\sum_{\alpha\in\Phi(\tilde{\mathbf F}_e)}\frac{\alpha(a_1)\tau(\psi)}{\beta_0(a_{-1})\tilde{\nu}_0(a_1)}$$ where $\tilde{\nu}_0$ is any $\mathbb Q$-root in $R_u(\overline{\tilde{\mathbf P}}_0)$ such that $$\tilde{\nu}_0(a_1)=\max\{\alpha(a_1):\alpha\in\Phi(R_u(\overline{\tilde{\mathbf P}}_0))\}.$$ Note that $\mathbf G_i(\mathbb R)$ $(s+1\leq i\leq k)$ commutes with $\{a_t\}_{t\in\mathbb R}$, and an element $g\in S_\rho(\psi)^c$ if and only if $h\cdot g\in S_\rho(\psi)^c$ for any $h\in\mathbf G_i(\mathbb R)$ $(s+1\leq i\leq k)$. Therefore, we have
\begin{align*}
\dim_H(S_\rho(\psi)^c)=&\dim_H(S_\rho(\psi)^c\cap\tilde{\mathbf G}(\mathbb R))+\sum_{i=s+1}^k\dim\mathbf G_i(\mathbb R)\\
\geq&\dim\mathbf G-\sum_{\alpha\in\Phi(\mathbf F_e)}\frac{\alpha(a_1)\tau(\psi)}{\nu_0(a_1)\beta_0(a_{-1})}
\end{align*}
where $\nu_0(a_1):=\max\{\alpha(a_1):\alpha\in\Phi(\mathbf F_e)\}$. Here we use the fact that $\nu_0(a_1)=\tilde{\nu}_0(a_1)$ and $$\sum_{\alpha\in\Phi(\tilde{\mathbf F}_e)}\alpha(a_1)/\beta_0(a_1)=\sum_{\alpha\in\Phi(\mathbf F_e)}\alpha(a_1)/\beta_0(a_1).$$ This completes the proof of Theorem~\ref{mthm11}.

\begin{remark}\label{r56}
One can see from the arguments in \S\ref{lbd1} and \S\ref{lbd2} that we actually prove that for any open bounded subset $U$ in $R_u(\overline{\mathbf P}_0)(\mathbb R)$ $$\dim_H(S_\rho(\psi)^c\cap U)\geq\dim R_u(\overline{\mathbf P}_0)(\mathbb R)-\sum_{\alpha\in\Phi(\mathbf F_e)}\frac{\alpha(a_1)\tau(\psi)}{\nu_0(a_1)\beta_0(a_{-1})}.$$ Note that $\mathbf P_0\cdot R_u(\overline{\mathbf P}_0)$ is a Zariski open dense subset in $\mathbf G$, and for any open bounded subset $\tilde U\subset\mathbf G(\mathbb R)$, one can find open subsets $U_1$ and $U_2$ in $R_u(\overline{\mathbf P}_0)(\mathbb R)$ and $\mathbf P_0(\mathbb R)$ respectively such that $$U_2\cdot U_1\subset\tilde U.$$ Moreover, by definition, if $x\in S_\rho(\psi)^c$, then for any $g\in\mathbf P_0(\mathbb R)$, $g\cdot x$ is also an element in $S_\rho(\psi)^c$. This implies that $$\dim_H(S_\rho(\psi)^c\cap\tilde U)\geq\dim U_2+\dim_H(S_\rho(\psi)^c\cap U_1)\geq\dim\mathbf G-\sum_{\alpha\in\Phi(\mathbf F_e)}\frac{\alpha(a_1)\tau(\psi)}{\beta_0(a_{-1})\nu_0(a_1)}.$$ 
\end{remark}

Finally, we study the case $\tau(\psi)=\beta_0(a_{-1})$ and conclude the proof of Theorem~\ref{mthm11}.

\begin{proposition}\label{p56}
Let $\mathbf G$ be a connected semisimple algebraic group, $\rho$ a $\mathbb Q$-rational irreducible representation of $\mathbf G$, $\{a_t\}_{t\in\mathbb R}$ a one-parameter subgroup in $\mathbf T(\mathbb R)$ and $\beta_0$ the highest weight defined as in \S\ref{pre}. Let $\psi:\mathbb R_+\to\mathbb R_+$ with $\tau(\psi)=\beta_0(a_{-1})$. Then
\begin{enumerate}
\item If $\psi(t)\cdot e^{\beta_0(a_{-1})t}$ is bounded, then $S_\rho(\psi)^c=\emptyset$.
\item If $\psi(t)\cdot e^{\beta_0(a_{-1})t}$ is unbounded, then $S_\rho(\psi)^c\neq\emptyset$ and $$\dim_HS_\rho(\psi)^c\geq\dim\mathbf G-\frac{\tau(\psi)}{\beta_0(a_{-1})\nu_{0}(a_1)}\cdot\sum_{\alpha\in\Phi(R_u(\overline{\mathbf P}_{\beta_0}))}\alpha(a_1).$$
\end{enumerate}
\end{proposition}
\begin{proof}
Suppose that $\psi(t)\cdot e^{\beta_0(a_{-1})t}$ is bounded. Then there exists $C>0$ such that $$\psi(t)\leq C\cdot e^{-\beta_0(a_{-1})t}\quad(t>0).$$ Let $g\in\mathbf G(\mathbb R)$. Since $\beta_0(a_{-1})>0$ is the fastest contracting rate of the $\{a_t\}_{t\in\mathbb R}$-action on $V$, we have $$\|\rho(a_t)v\|\geq e^{-\beta_0(a_{-1})t}\|v\|\quad(\forall v\in\rho(g)\cdot\mathbb Z^d\setminus\{0\})$$ and $$\delta(\rho(a_t\cdot g)\cdot\mathbb Z^d)\geq e^{-\beta_0(a_{-1})t}\cdot\delta(\rho(g)\cdot\mathbb Z^d).$$ One can find a constant $C_g>0$ depending on $g$ such that $$\delta(\rho(a_t\cdot g)\cdot\mathbb Z^d)\geq C_g\cdot\psi(t)\quad(t>0)$$ which shows that $g\in S_\rho(\psi)$ and $S_\rho(\psi)=\mathbf G(\mathbb R)$. This proves the first claim.

Now suppose that $\psi(t)\cdot e^{\beta_0(a_{-1})t}$ is unbounded. We first consider the case where $\{a_t\}_{t\in\mathbb R}$ on $\mathbf G(\mathbb R)^0/\Gamma$ is mixing. We may write $$\psi(t)\cdot e^{\beta_0(a_{-1})t}=e^{h(t)}$$ where $h(t)$ is not bounded from above. Since $\tau(\psi)=\beta_0(a_{-1})$, we have $$\limsup_{t\to\infty}\frac{h(t)}t=0.$$ Let $\{s_k\}_{k\in\mathbb N}$ be a divergent sequence in $\mathbb R_+$ such that $\lim_{k\to\infty}h(s_k)=\infty.$ Then $$\lim_{k\to\infty}\frac{h(s_k)}{s_k}=0.$$ In \S\ref{lbd1}, we construct a Cantor-type subset $\mathbf A_\infty$ in $R_u(\overline{\mathbf P}_0)(\mathbb R)$ with parameters $l_k$ and $t_k$ $(k\in\mathbb N)$. Here to prove the proposition with $\tau(\psi)=\beta_0(a_{-1})$, we may take $l_k=e^{h(s_k)/2}$, $t_k=s_k$ and $\tau_k=\beta_0(a_{-1})-h(s_k)/(2s_k)$ $(k\in\mathbb N)$. Then $$t_k=\frac{\ln l_k}{\beta_0(a_{-1})-\tau_k}.$$ Using the same argument in Lemma~\ref{l44}, one can prove that for any $p\in\mathbf A_{\infty}$
\begin{align*}
\delta(\rho(a_{t_k}\cdot p)\mathbb Z^d)&\lesssim e^{-\tau_k\cdot t_k}=e^{-\beta_0(a_{-1})t_k+h(t_k)/2}=\psi(t_k)\cdot e^{-h(t_k)/2}\quad(k\in\mathbb N)
\end{align*}
and $\mathbf A_\infty\subset S_{\rho}(\psi)^c\cap R_u(\overline{\mathbf P}_0)(\mathbb R)$. By Theorem~\ref{thm43}, we may estimate the Hausdorff dimension of $\mathbf A_\infty$ and obtain $$\dim_HS_\rho(\psi)^c\geq\dim\mathbf G-\frac{\tau(\psi)}{\beta_0(a_{-1})\nu_{0}(a_1)}\cdot\sum_{\alpha\in\Phi(R_u(\overline{\mathbf P}_{\beta_0}))}\alpha(a_1).$$

If the action of $\{a_t\}_{t\in\mathbb R}$ on $\mathbf G(\mathbb R)^0/\Gamma$ is not mixing, then we may repeat the arguments in this section and reduce the problem in the homogeneous space $\tilde{\mathbf G}(\mathbb R)^0/\tilde{\Gamma}$. This completes the proof of the proposition.
\end{proof}

\section{An upper bound for the Hausdorff dimension of $S_\rho(\psi)^c$}\label{ubd}
In this section, we prove Theorem~\ref{mthm12}. We will first compute an upper bound for the Hausdorff dimension of $S_\rho(\psi)^c\cap R_u(\overline{\mathbf P}_0)(\mathbb R)$ by constructing open covers of the subset $S_\rho(\psi)^c\cap R_u(\overline{\mathbf P}_0)(\mathbb R)$, and then derive Theorem~\ref{mthm12} from the upper bound of $\dim_H S_\rho(\psi)^c\cap R_u(\overline{\mathbf P}_0)(\mathbb R)$. We denote by $$\mathcal C:=\bigcup_{w\in{_{\mathbb Q}\overline{\mathcal W}}}w\mathbf F_{w}(\mathbb R).$$

\begin{proposition}\label{p61}
Suppose that $$\delta(\rho(g)\cdot\mathbb Z^d)\leq r$$ for some $g\in\mathbf G(\mathbb R)$ and $r>0$. Then there exist $\tilde g\in\mathcal C$ and a discrete primitive subgroup of rank $\dim V_{\beta_0}$ $$\Lambda_g\subset\rho(g)\cdot\mathbb Z^d$$ such that $$\Vol(\Lambda_g)\leq C\cdot r^{\dim V_{\beta_0}}\textup{ and }\rho(\tilde g)\cdot\Lambda_g\subset V_{\beta_0}\textup{ is Zariski dense in }V_{\beta_0}$$ where the constant $C>0$ depends only on $\mathbf G$, $\Gamma$ and $\rho$.
\end{proposition}
\begin{proof}
By the reduction theory of arithmetic subgroups, $\mathbf G(\mathbb R)$ can be written as $$\mathbf G(\mathbb R)=S_{\eta,\Omega}\cdot\mathcal K\cdot\Gamma$$ for some Siegel set $S_{\eta,\Omega}=K\cdot T_\eta\cdot\Omega$ and some finite subset $\mathcal K\subset\mathbf G(\mathbb Q)$, where $$T_{\eta}=\{a\in T(\mathbb R):\lambda(a)\leq\eta,\;\lambda\textup{ a simple root in }\Delta\}$$ and the finite set $\mathcal K$ satisfies the property that $$\mathbf G(\mathbb Q)=\mathbf P_0(\mathbb Q)\cdot\mathcal K\cdot\Gamma.$$ Then there exist $k\in K$, $a\in T_\eta$, $u\in\Omega$, $\gamma\in\Gamma$ and $x\in\mathcal K\subset\mathbf G(\mathbb Q)$ such that $$g=k\cdot a\cdot u\cdot x\cdot\gamma.$$
Note that by the definition of $T_\eta$, the subset $$\{aua^{-1}: a\in T_\eta\textup{ and }u\in\Omega\}$$ is relatively compact.  So we can write $$g=\tilde k\cdot a\cdot x\cdot\gamma$$ for $\tilde k=k\cdot (a\cdot u\cdot a^{-1})$ in a fixed compact subset in $\mathbf G(\mathbb R)$.

Now since $\delta(\rho(g)\cdot\mathbb Z^d)\leq r$, there exists $y\in\mathbb Z^d\setminus\{0\}$ such that $\rho(\tilde k\cdot a\cdot x)\cdot y$ is a shortest vector in $\rho(g)\cdot\mathbb Z^d$ and $$\|\rho(a\cdot x)\cdot y\|\asymp\|\rho(\tilde k\cdot a\cdot x)\cdot y\|\leq r.$$ We write $$\rho(x)\cdot y=\sum_{\text{weight }\beta}y_\beta$$ according to the decomposition of $V=\oplus_{\beta}V_{\beta}$ in $\rho$ into the weight spaces $V_\beta$ relative to $\mathbf T$, where $y_\beta\in V_{\beta}$. Let $\tilde\beta$ be a weight among the weights $\beta$'s in the summation above such that $y_{\tilde\beta}\neq 0$. Then
\begin{align}\label{eqn1}
\|\rho(a\cdot x)\cdot y\|=\|e^{\tilde\beta(a)}y_{\tilde\beta}+\cdots\|\lesssim r.
\end{align}
 Note that $y\in\mathbb Z^d$ and $x\in\mathcal K\subset\mathbf G(\mathbb Q)$, and there is a positive lower bound for $\|y_{\tilde\beta}\|$ depending only on $\rho$ and $\mathcal K$. So by equation~\eqref{eqn1}, we have $$e^{\tilde\beta(a)}\lesssim r.$$ On the other hand, there exists a unique discrete primitive subgroup $\tilde\Lambda\subset\mathbb Z^d$ such that $\rho(x)\cdot\tilde\Lambda$ is a discrete and Zariski-dense subgroup in $V_{\beta_0}$, and for any $w\in\tilde\Lambda\setminus\{0\}$ we have 
\begin{align}\label{eqn2}
\|\rho(\tilde k\cdot a\cdot x)\cdot w\|\asymp\|\rho(a)\cdot\rho(x)\cdot w\|=e^{\beta_0(a)}\|\rho(x)\cdot w\|.
\end{align} 
By the condition $a\in T_\eta$ and the relation between $\tilde\beta$ and $\beta_0$ according to the structure of the representation $\rho$, one can compute that $$e^{\beta_0(a)}\lesssim_{\eta} e^{\tilde\beta(a)}\lesssim r.$$ Consequently, from equation~\eqref{eqn2} one can deduce that for any $w\in\tilde\Lambda\setminus\{0\}$ $$\|\rho(\tilde k\cdot a\cdot x)w\|\lesssim r\cdot\|\rho(x)w\|.$$ Let $\Lambda_g=\rho(g)(\rho(\gamma^{-1})\tilde\Lambda)$. Then we have $$\Lambda_g\subset\rho(g)\cdot\mathbb Z^d,\quad\Vol(\Lambda_g)\lesssim r^{\rank\Lambda_g}=r^{\dim V_{\beta_0}}\textup{ and }\rho(\tilde k^{-1})\cdot\Lambda_g\subset V_{\beta_0}.$$ We know that $$\tilde k^{-1}\in\mathbf G(\mathbb R)=\bigcup_{w\in{_{\mathbb Q}\overline{\mathcal W}}}\mathbf P_{\beta_0}(\mathbb R)\cdot w\cdot\mathbf F_{w}(\mathbb R).$$ So there exists $\tilde g\in\mathcal C$ such that $$\tilde k^{-1}\in\mathbf P_{\beta_0}(\mathbb R)\cdot\tilde g.$$ By the fact that $\mathbf P_{\beta_0}$ stabilizes $V_{\beta_0}$, we conclude that $$\rho(\tilde g)\cdot\Lambda_g\subset V_{\beta_0}.$$ Note that all the implicit constants above depend only on $\mathbf G$, $\Gamma$ and $\rho$. This completes the proof of the proposition.
\end{proof}

Let $g$ be an element in $S_\rho(\psi)^c\cap R_u(\overline{\mathbf P}_0)(\mathbb R)$. Then from Definition~\ref{def11}, one can deduce that there exists a divergent sequence $\{t_k\}\subset\mathbb R_+$ such that $$\delta(\rho(a_{t_k} g)\cdot\mathbb Z^d)\leq\psi(t_k).$$ Let $\epsilon>0$ be a sufficiently small number. By the definition of $\tau=\tau(\psi)$, we have $$\psi(t)\leq C_\epsilon\cdot e^{-(\tau-\epsilon)t}\quad(\forall t>0)$$ for some $C_\epsilon>0$ depending only on $\epsilon$. Hence $$\delta(\rho(a_{t_k} g)\cdot\mathbb Z^d)\leq C_\epsilon\cdot e^{-(\tau-\epsilon)t_k}.$$ Here by adding a multiplicative constant in the inequality above, we may assume that $t_k\in\mathbb N$. By Proposition~\ref{p61}, for each $k\in\mathbb N$, there exist $\tilde g_k\in\mathcal C$ and a discrete subgroup $\Lambda_k\subset \rho(a_{t_k}g)\cdot\mathbb Z^d$ of rank $\dim V_{\beta_0}$ such that $$\Vol(\Lambda_k)\lesssim (C_\epsilon\cdot e^{-(\tau-\epsilon)t_k})^{\dim V_{\beta_0}}\textup{ and }\rho(\tilde g_k)\cdot\Lambda_k\textup{ is Zariski dense in } V_{\beta_0}.$$ Since ${_{\mathbb Q}\overline{\mathcal W}}$ is finite, by passing to a subsequence, we may write $\tilde g_k=w\cdot f_{k}\in\mathcal C$ where $w\in{_{\mathbb Q}\overline{\mathcal W}}$ is fixed for all $k\in\mathbb N$ and $f_{k}\in\mathbf F_{w}(\mathbb R)$. Then by definition $$\tilde q_k:=\tilde g_k\cdot a_{t_k}\cdot g=w\cdot f_{k}\cdot a_{t_k}\cdot g=w\cdot a_{t_k}\cdot (a_{-t_k}f_{k}a_{t_k})\cdot g$$ $$q_k:=(wa_{-t_k}w^{-1})\cdot\tilde q_k=w\cdot(a_{-t_k}f_{k}a_{t_k})\cdot g$$ are rational elements in $\mathbf G(\mathbb R)$. 

In the $\dim V_{\beta_0}$-exterior product space of $V$, let $e_{V_{\beta_0}}=e_1\wedge e_2\wedge\cdots\wedge e_{\dim V_{\beta_0}}$ be the unit vector which represents the vector space $V_{\beta_0}$ (see \S\ref{intro}), and $\wedge^{\dim V_{\beta_0}}\Lambda_k$ the vector in $\bigwedge^{\dim V_{\beta_0}}V$ which is constructed from a $\mathbb Z$-basis of $\Lambda_k$ and represents the lattice $\Lambda_k$. Then we can compute that $$\height(q_k)=e^{\beta_0(wa_{-t_k}w^{-1})\cdot\dim V_{\beta_0}}\cdot\height(\tilde q_k)$$ $$\height(\tilde q_k)\cdot e_{V_{\beta_0}}=\rho_{\beta_0}(\tilde g_k)\cdot\left(\wedge^{\dim V_{\beta_0}}\Lambda_k\right),\quad\rho_{\beta_0}(\tilde g_k^{-1})\cdot e_{V_{\beta_0}}=\left(\wedge^{\dim V_{\beta_0}}\Lambda_k\right)/\height(\tilde q_k)$$ 
\begin{align*}
\|\rho_{\beta_0}(\tilde g_k^{-1})\cdot e_{V_{\beta_0}}\|=&\Vol(\Lambda_k)/\height(\tilde q_k)\\
\lesssim& (C_\epsilon\cdot e^{-(\tau-\epsilon)t_k})^{\dim V_{\beta_0}}\cdot e^{\beta_0(wa_{-t_k}w^{-1})\cdot\dim V_{\beta_0}}/\height(q_k).
\end{align*}
The implicit constants here depend only on $\mathbf G$, $\Gamma$ and $\rho$. For convenience, we denote by $$\lambda_w(\epsilon):=(\beta_0(wa_{-1}w^{-1})-(\tau-\epsilon))\cdot\dim V_{\beta_0}.$$ 

\begin{definition}\label{def62}
Let $w\in{_{\mathbb Q}\overline{\mathcal W}}$. For any $R>0$, define $$E_{w}(R):=\{f\in\mathbf F_{w}(\mathbb R):\|\rho_{\beta_0}(w fw^{-1})\cdot e_{V_{\beta_0}}\|\leq R\}.$$ We also define the following morphism $$\Psi_{w}:\mathbf F_{w}(\mathbb R)\to\bigwedge^{\dim V_{\beta_0}} V,\quad\Psi_{w}(x)=\rho_{\beta_0}(wxw^{-1})\cdot e_{V_{\beta_0}}.$$ Note that $w\mathbf F_{w}w^{-1}\subset R_u(\overline{\mathbf P}_{\beta_0})$ and the stabilizer of $\mathbb C\cdot e_{V_{\beta_0}}$ is equal to $\mathbf P_{\beta_0}$. Hence $\Psi_{w}$ is an isomorphism onto its image.
\end{definition}

From the discussion above, we know that 
\begin{align*}
\|\rho_{\beta_0}(wf_{k}^{-1}w^{-1})\cdot e_{V_{\beta_0}}\|\lesssim&\|\rho_{\beta_0}(\tilde g_k^{-1})\cdot e_{V_{\beta_0}}\|\lesssim C_\epsilon^{\dim V_{\beta_0}}\cdot e^{\lambda_w(\epsilon)\cdot t_k}/\height(q_k)
\end{align*}
and by definition of $E_{w}(R)$, we have
\begin{align*}
g=&(a_{-t_k}f_{k}^{-1}a_{t_k})\cdot w^{-1}\cdot q_k\\
\in&\left(a_{-t_k}\cdot E_{w}\left(C_1\cdot C_\epsilon^{\dim V_{\beta_0}}\cdot e^{\lambda_w(\epsilon)\cdot t_k}/\height(q_k)\right)\cdot a_{t_k}\right)\cdot w^{-1}\cdot q_k
\end{align*}
for some constant $C_1>0$ depending only on $\mathbf G$, $\Gamma$ and $\rho$. Moreover, by the structure of the representation $\rho_{\beta_0}$, we know that $$\|\rho_{\beta_0}(wf_{k}^{-1}w^{-1})\cdot e_{V_{\beta_0}}\|=\|\Psi_{w}(f_{k}^{-1})\|\gtrsim\|e_{V_{\beta_0}}\|=1$$ so $$\height(q_k)\leq C_2\cdot C_\epsilon^{\dim V_{\beta_0}}\cdot e^{\lambda_w(\epsilon)\cdot t_k}$$ for some constant $C_2>0$ depending only on $\mathbf G$, $\Gamma$ and $\rho$. Therefore, $q_k$ is a rational element in $\mathbf G(\mathbb R)$ with $$\height(q_k)\leq C_2\cdot C_\epsilon^{\dim V_{\beta_0}}\cdot e^{\lambda_w(\epsilon)\cdot t_k}$$ and $w^{-1}\cdot q_k\in R_u(\overline{\mathbf P}_0)(\mathbb R)$. In the following, we take $C_0=\max\{C_1,C_2\}$.

\begin{definition}\label{def63}
For any $w\in{_{\mathbb Q}\overline{\mathcal W}}$ and any $\epsilon>0$, define $\mathcal E_{w,\epsilon}(\psi)$ to be the subset of elements $g\in R_u(\overline{\mathbf P}_0)(\mathbb R)$ for which there exist a divergent sequence $\{t_k\}\subset\mathbb N$ and a sequence of rational elements $\{q_k\}\subset\mathbf G(\mathbb R)$ such that $$g\in\left(a_{-t_k}E_{w}(C_0\cdot C_\epsilon^{\dim V_{\beta_0}}\cdot e^{\lambda_w(\epsilon)\cdot t_k}/\height(q_k))a_{t_k}\right)w^{-1}q_k$$  $$\height(q_k)\leq C_0\cdot C_\epsilon^{\dim V_{\beta_0}}\cdot e^{\lambda_w(\epsilon)\cdot t_k}$$ and $w^{-1}q_k\in R_u(\overline{\mathbf P}_0)(\mathbb R)$.
\end{definition}

\begin{lemma}\label{l64}
Let $w\in{_{\mathbb Q}\overline{\mathcal W}}$. Let $q$ be a rational element in $\mathbf G(\mathbb R)$ with $w^{-1}\cdot q\in R_u(\overline{\mathbf P}_0)(\mathbb R)$. Then there exists a constant $\theta_w>0$ depending only on $w$ such that $d(q)\geq\theta_w$.
\end{lemma}
\begin{proof}
By definition, we know that $R_u(\overline{\mathbf P}_0)=\mathbf H_{w}\cdot\mathbf F_{w}.$ By Corollary~\ref{c23}, let $$w^{-1}\cdot q=u\cdot v$$ for some $u\in\mathbf H_{w}(\mathbb R)$ and $v\in\mathbf F_{w}(\mathbb Q)$. Let $\tilde\Omega_{\mathbf F_{w}}$ be a bounded fundamental domain of $\mathbf F_{w}(\mathbb R)/(\Gamma\cap\mathbf F_{w}(\mathbb R))$ in $\mathbf F_{w}(\mathbb R)$. We may write $$v=\tilde v\cdot\gamma$$ for some $\tilde v\in\tilde\Omega_{\mathbf F_{w}}$ and $\gamma\in\Gamma\cap\mathbf F_{w}(\mathbb R)$. Then $$q\cdot\mathbb Z^d=(w\cdot u\cdot v)\cdot\mathbb Z^d=(wuw^{-1})\cdot(w\tilde v\cdot\mathbb Z^d).$$ By the fact that $w\mathbf H_{w}w^{-1}$ fixes every element in $V_{\beta_0}$ and $\tilde v$ is in the bounded subset $\tilde{\Omega}_{\mathbf F_{w}}$, we can conclude that the co-volume of $(w\tilde v\cdot\mathbb Z^d)\cap V_{\beta_0}$ has a lower bound depending only on $w$ and $\tilde\Omega_{\mathbf F_{w}}$, and so does $q\cdot\mathbb Z^d\cap V_{\beta_0}$. This completes the proof of the lemma.
\end{proof}


\begin{lemma}\label{l65}
Let $w\in{_{\mathbb Q}\overline{\mathcal W}}$ and $\epsilon>0$. If $\mathcal E_{w,\epsilon}(\psi)\neq\emptyset$, then $\tau-\epsilon\leq\beta_0(wa_{-1}w^{-1})$.
\end{lemma}
\begin{proof}
Let $g\in\mathcal E_{w,\epsilon}(\psi)$. By definition, there exist a divergent sequence $\{t_k\}\subset\mathbb N$ and a sequence of rational elements $\{q_k\}\subset\mathbf G(\mathbb R)$ such that $$g\in\left(a_{-t_k}E_{w}(C_0\cdot C_\epsilon^{\dim V_{\beta_0}}\cdot e^{\lambda_w(\epsilon)\cdot t_k}/\height(q_k))a_{t_k}\right)w^{-1}q_k$$  $$\height(q_k)\leq C_0\cdot C_\epsilon^{\dim V_{\beta_0}}\cdot e^{\lambda_w(\epsilon)\cdot t_k}$$ and $w^{-1}q_k\in R_u(\overline{\mathbf P}_0)(\mathbb R)$. By Lemma~\ref{l64}, we have $$\theta_w\leq C_0\cdot C_\epsilon^{\dim V_{\beta_0}}\cdot e^{\lambda_w(\epsilon)\cdot t_k}.$$ Then by taking $t_k\to\infty$, we obtain that $\tau-\epsilon\leq\beta_0(wa_{-1}w^{-1})$. This completes the proof of the lemma.
\end{proof}

\begin{proposition}\label{p66}
For any $\epsilon>0$, we have $$S_\rho(\psi)^c\cap R_u(\overline{\mathbf P}_0)(\mathbb R)\subseteq\bigcup_{w\in{_{\mathbb Q}\overline{\mathcal W}}}\mathcal E_{w,\epsilon}(\psi).$$ Moreover, for any $w\in{_{\mathbb Q}\overline{\mathcal W}}$, $\beta_0(wa_{-1}w^{-1})\leq\beta_0(a_{-1})$.
\end{proposition}
\begin{proof}
The inclusion $$S_\rho(\psi)^c\cap R_u(\overline{\mathbf P}_0)(\mathbb R)\subseteq\bigcup_{w\in{_{\mathbb Q}\overline{\mathcal W}}}\mathcal E_{w,\epsilon}(\psi)$$ follows from the discussion above. Note that for any $w\in{_{\mathbb Q}\overline{\mathcal W}}$, $\beta_{w}(a):=\beta_0(waw^{-1})$ $(\forall a\in\mathbf T)$ is a weight in the representation $\rho$, and $\beta_0$ is the highest weight. It implies that $\beta_{w}(a_{-1})\leq\beta_0(a_{-1})$. This completes the proof of the proposition.
\end{proof}

In the following, we will consider the subsets $S_\rho(\psi)^c\cap\mathcal E_{w,\epsilon}(\psi)$ $(w\in{_{\mathbb Q}\overline{\mathcal W}})$ and compute upper bounds for the Hausdorff dimensions of these subsets. These upper bounds will give an upper bound for $\dim_H(S_\rho(\psi)^c\cap R_u(\overline{\mathbf P}_0)(\mathbb R))$ by Proposition~\ref{p66}. Note that we are only interested in the non-empty subsets $\mathcal E_{w,\epsilon}(\psi)\neq\emptyset$. So here we will only consider the elements $w\in{_{\mathbb Q}\overline{\mathcal W}}$ which satisfy $\tau\leq\beta_0(wa_{-1}w^{-1}).$ For the elements $w\in{_{\mathbb Q}\overline{\mathcal W}}$ with $\tau>\beta_0(wa_{-1}w^{-1})$, we may choose $\epsilon>0$ sufficiently small so that $\tau-\epsilon>\beta_0(wa_{-1}w^{-1})$, and then by Lemma~\ref{l65}, $\mathcal E_{w,\epsilon}(\psi)=\emptyset$. For notational convenience, we may set $C_0\cdot C_\epsilon^{\dim V_{\beta_0}}=1$ as it does not affect the computations in the rest of this section. 

Now let $\epsilon>0$ be a sufficiently small number and fix $w\in{_{\mathbb Q}\overline{\mathcal W}}$ with $\tau\leq\beta_0(wa_{-1}w^{-1})$. To compute the Hausdorff dimension of $S_\rho(\psi)^c\cap\mathcal E_{w,\epsilon}(\psi)$, we first need to estimate the volume of the subset $E_{w}(R)$ in ${\mathbf F}_{w}(\mathbb R)$, which can be written as $$\Vol(E_{w}(R))=\mu_{\mathbf F_{w}}(\Psi_{w}^{-1}(B_R))$$ where $B_R$ is the ball of radius $R$ centered at $0$ in $\bigwedge_{i=1}^{\dim V_{\beta_0}}V$ and $\mu_{\mathbf F_{w}}$ is the Haar measure on ${\mathbf F}_{w}(\mathbb R)$. We need the following result about the asymptotic volume estimates of algebraic varieties.

\begin{theorem}[{\cite[Corollary 16.3]{BO12}}]\label{thm67}
Let $\mathcal O$ be a closed orbit of a group $\mathbf H(\mathbb R)$ of real points of an algebraic group $\mathbf H$ in an $\mathbb R$-vector space $V$, $\mu$ an $\mathbf H(\mathbb R)$-invariant measure on $\mathcal O$ and $\|\cdot\|$ a Euclidean norm on $V$. Let $B_R=\{v\in V: \|v\|\leq R\}$. Then $$\mu(B_R)\sim cR^a(\log R)^b\;(\textup{as }R\to\infty)$$ for some $a\in\mathbb Q_{\geq0}$, $b\in\mathbb Z_{\geq0}$ and $c>0$.
\end{theorem}

We apply Theorem~\ref{thm67} to the morphism $\Psi_{w}$ where the closed orbit $\mathcal O=\Psi_{w}(\mathbf F_{w}(\mathbb R))$ and the measuer $\mu$ is the push-forward $\Psi_{w}^*(\mu_{\mathbf F_{w}})$ of the Haar measure on $\mathbf F_{w}(\mathbb R)$, and obtain the following
\begin{corollary}\label{c68}
There exist constants $a_{w}\in\mathbb Q_{\geq0}$, $b_{w}\in\mathbb Z_{\geq0}$ and $c_{w}>0$ such that $$\Vol(E_{w}(R))=\mu_{\mathbf F_{w}}(\Psi_{w}^{-1}(B_R))\sim c_{w}R^{a_{w}}(\log R)^{b_{w}}\;(\textup{as }R\to\infty).$$ 
\end{corollary}

To compute the Hausdorff dimension of $S_\rho(\psi)^c\cap\mathcal E_{w,\epsilon}(\psi)$, we also need an upper bound for the asymptotic estimate of the number of rational elements in $$(w\mathbf F_{w}w^{-1})(\mathbb Q)\subset R_u(\overline{\mathbf P}_{\beta_0})(\mathbb Q).$$ Consider the morphism $$\Psi_{w}: \mathbf F_{w}(\mathbb R)\to\bigwedge^{\dim V_{\beta_0}} V,\quad\Psi_{w}(x)=\rho_{\beta_0}(wxw^{-1})\cdot e_{V_{\beta_0}}.$$ Let $\Omega_{\beta_0}$ be an open bounded subset in $R_u(\overline{\mathbf P}_{\beta_0})(\mathbb R)$ which contains a fundamental domain of $R_u(\overline{\mathbf P}_{\beta_0})(\mathbb R)/(R_u(\overline{\mathbf P}_{\beta_0})(\mathbb R)\cap\Gamma)$. Let $g\in w\mathbf F_{w}w^{-1}(\mathbb Q)\subset R_u(\overline{\mathbf P}_{\beta_0})(\mathbb Q)$ be a rational element in $\Omega_{\beta_0}$. By definition, $\height(g)$ is the co-volume of $\rho(g)\cdot\mathbb Z^d\cap V_{\beta_0}$ in $V_{\beta_0}(\mathbb R)$. In the $\dim V_{\beta_0}$-exterior product space of $V$, this implies that there exists a constant $C_{\beta_0}>0$ depending only on $\Omega_{\beta_0}$, such that the length of the primitive integral vector in the line spanned by $\rho(g^{-1})\cdot e_{V_{\beta_0}}$ is less than $C_{\beta_0}\cdot\height(g)$. So the number of rational elements $g$ in $\Omega_{\beta_0}\cap(w\mathbf F_{w}w^{-1})(\mathbb Q)$ whose heights are less than $l>0$ is bounded above by the total number of primitive integral points in $\bigwedge_{i=1}^{\dim V_{\beta_0}}V$ whose lengths are less than $C_{\beta_0}\cdot l$. This leads to the results about Manin's conjecture \cite{BM90}, i.e. the asymptotic estimate of the number of rational points in an algebraic variety  (See e.g. \cite{CT02,FMT89,GMO08,GO11,MG14,P95,STT07,ST99,ST16} and the references therein.) For the general situation, we define 
\begin{align*}
A_w=A_w(\Omega_{\beta_0}):=\limsup_{l\to\infty}\frac{\log|\{g\in\Omega_{\beta_0}\cap w\mathbf F_{w}w^{-1}(\mathbb Q): \height(g)\leq l\}|}{\log l}
\end{align*}
and from the discussion above, we know that $A_w$ is a finite number which is bounded above by the growth rate of the number of the primitive integral points in $\bigwedge^{\dim V_{\beta_0}}V$. We also have that for any $\epsilon>0$ $$|\{g\in\Omega_{\beta_0}\cap w\mathbf F_{w}w^{-1}(\mathbb Q): \height(g)\leq l\}|\lesssim_{\epsilon,\Omega_{\beta_0}} l^{A_w+\epsilon}$$ where the implicit constant depends only on $\epsilon,\Omega_{\beta_0},\mathbf G$ and $\rho$. In the following, we will fix the constants $a_{w},b_{w},c_{w}$ in Corollary~\ref{c68} and $A_{w}$ for the bounded open subset $\Omega_{\beta_0}$ in $R_u(\overline{\mathbf P}_{\beta_0})(\mathbb R)$.

Now we compute an upper bound for the Hausdorff dimension of the subset $S_\rho(\psi)^c\cap\mathcal E_{w,\epsilon}(\psi)$ $(\tau\leq\beta_0(wa_{-1}w^{-1}))$. In the computation, we project the subset $S_\rho(\psi)^c\cap\mathcal E_{w,\epsilon}(\psi)$ into the quotient space $R_u(\overline{\mathbf P}_0)(\mathbb R)/\Gamma\cap R_u(\overline{\mathbf P}_0)(\mathbb R)$ by the natural projection map $$\pi_{R_u(\overline{\mathbf P}_0)}:R_u(\overline{\mathbf P}_0)(\mathbb R)\to R_u(\overline{\mathbf P}_0)(\mathbb R)/(\Gamma\cap R_u(\overline{\mathbf P}_0)(\mathbb R))$$ and compute the Hausdorff dimension of $\pi_{R_u(\overline{\mathbf P}_0)}(S_\rho(\psi)^c\cap\mathcal E_{w,\epsilon}(\psi))$. Since $\pi_{R_u(\overline{\mathbf P}_0)}$ is a local isometry, by the countable stability of Hausdorff dimension, we then obtain upper bounds for $\dim_H(\pi_{R_u(\overline{\mathbf P}_0)}^{-1}(\pi_{R_u(\overline{\mathbf P}_0)}(S_\rho(\psi)^c\cap\mathcal E_{w,\epsilon}(\psi))))$ and $\dim_H(S_\rho(\psi)^c\cap\mathcal E_{w,\epsilon}(\psi))$. 

Let $\tilde\Omega$ be a bounded fundamental domain of $R_u(\overline{\mathbf P}_0)(\mathbb R)/(\Gamma\cap R_u(\overline{\mathbf P}_0)(\mathbb R))$ in $R_u(\overline{\mathbf P}_0)(\mathbb R)$. We choose a bounded subset $\Omega_{\mathbf H_{w}}$ in $\mathbf H_{w}(\mathbb R)$ and a bounded subset $\Omega_{\mathbf F_{w}}$ in $\mathbf F_{w}(\mathbb R)$ such that $$\tilde\Omega\subset\Omega_{\mathbf H_{w}}\cdot\Omega_{\mathbf F_{w}}.$$ Without loss of generality, we may assume that $$\tilde\Omega\subset\Omega_{\mathbf H_{w}}\cdot\Omega_{\mathbf F_{w}}\subset\Omega_{\beta_0}.$$ Now suppose that $q\in\mathbf G(\mathbb R)$ is a rational element with $w^{-1}\cdot q\in R_u(\overline{\mathbf P}_0)(\mathbb R)$. Let $\tilde q\in\mathbf G(\mathbb R)$ such that $w^{-1}\tilde q$ is a representative of $w^{-1}q$ in $\tilde\Omega$ with $$w^{-1}\tilde q=w^{-1}q\gamma$$ for some $\gamma\in \Gamma\cap R_u(\overline{\mathbf P}_0)(\mathbb R)$. Then one can check that $\tilde q$ is still a rational element in $\mathbf G(\mathbb R)$ and $\height(\tilde q)=\height(q)$ since $\Gamma$ preserves the lattice $\mathbb Z^d$. So by Corollary~\ref{c23}, we have $$w^{-1}\tilde q\in\mathbf H_{w}(\mathbb R)\cdot\mathbf F_{w}(\mathbb Q).$$ Since $w^{-1}\tilde q$ is an element in the bounded fundamental domain $\tilde\Omega$, it implies that $$w^{-1}\tilde q\in\Omega_{\mathbf H_{w}}\cdot(\mathbf F_{w}(\mathbb Q)\cap\Omega_{\mathbf F_{w}}).$$ We then obtain the following 

\begin{proposition}\label{p610}
 The set $\pi_{R_u(\overline{\mathbf P}_0)}(S_\rho(\psi)^c\cap\mathcal E_{w,\epsilon}(\psi))$ consists of points $g(\Gamma\cap R_u(\overline{\mathbf P}_0)(\mathbb R))$ for which there exist a divergent sequence $\{t_k\}\subset\mathbb N$ and a sequence of rational elements $\{q_k\}\subset\mathbf G(\mathbb R)$ such that $$g(\Gamma\cap R_u(\overline{\mathbf P}_0)(\mathbb R))\in\left(a_{-t_k}E_{w}(e^{\lambda_w(\epsilon)\cdot t_k}/\height(q_k))a_{t_k}\right)\cdot w^{-1}q_k(\Gamma\cap R_u(\overline{\mathbf P}_0)(\mathbb R))$$ $$\height(q_k)\leq e^{\lambda_w(\epsilon)\cdot t_k}$$ and $w^{-1}q_k\in\tilde\Omega$.
\end{proposition}

\begin{lemma}\label{l611}
Let $\delta>0$ and $B_{R_u(\overline{\mathbf P}_0)}(\delta)$ be the small open ball of radius $\delta$ centered at identity in $R_u(\overline{\mathbf P}_0)(\mathbb R)$. Then there exists a constant $C_{\delta}>0$ depending only on $\delta$ and $\rho$ such that for any $R>0$
\begin{align*}
B_{R_u(\overline{\mathbf P}_0)}(\delta)\cdot E_{w}(R)\subset E_{w}(C_\delta\cdot R)\cdot\mathbf H_{w}(\mathbb R).
\end{align*}
\end{lemma}
\begin{proof}
Let $x\in B_{R_u(\overline{\mathbf P}_0)}(\delta)$ and $y\in E_{w}(R)$. Then $x\cdot y\in R_u(\overline{\mathbf P}_0)(\mathbb R)$ and we write $$x\cdot y=u\cdot v$$ where $u\in\mathbf F_{w}(\mathbb R)$ and $v\in\mathbf H_{w}(\mathbb R)$. Then we have 
\begin{align*}
\|\Psi_{w}(u)\|=&\|\rho_{\beta_0}(wuw^{-1})\cdot e_{V_{\beta_0}}\|=\|\rho_{\beta_0}(wuvw^{-1})\cdot e_{V_{\beta_0}}\|\\
=&\|\rho_{\beta_0}(wxyw^{-1})\cdot e_{V_{\beta_0}}\|\leq C_\delta\cdot R
\end{align*}
for some constant $C_\delta>0$ depending only on $\delta>0$ and $\rho$. This completes the proof of the lemma.
\end{proof}

Note that by definition, for any rational element $g\in\mathbf G(\mathbb R)$, every element $x$ in the $\Gamma$-coset $g\Gamma$ of $g$ is also rational, and $\height(x)=\height(g)$. We will use this fact in the following lemma.

\begin{lemma}\label{l612}
Let $w\in{_{\mathbb Q}\overline{\mathcal W}}$ and $L>0$. Then the set of elements $g(\Gamma\cap R_u(\overline{\mathbf P}_0)(\mathbb R))$ in $R_u(\overline{\mathbf P}_0)(\mathbb R)/(\Gamma\cap R_u(\overline{\mathbf P}_0)(\mathbb R))$ with $w\cdot g$ rational and $\height(w\cdot g)=L$ is equal to $$\{(x\cdot y)(\Gamma\cap R_u(\overline{\mathbf P}_0)(\mathbb R)): x\in\Omega_{\mathbf H_{w}}, y\in\Omega_{\mathbf F_{w}}\cap\mathbf F_{w}(\mathbb Q), \height(w\cdot y)= L\}.$$
\end{lemma}
\begin{proof}
Let $g(\Gamma\cap R_u(\overline{\mathbf P}_0)(\mathbb R))\in R_u(\overline{\mathbf P}_0)/(\Gamma\cap R_u(\overline{\mathbf P}_0))$ with $w\cdot g$ rational and $\height(w\cdot g)= L$. We write $$g=(u\cdot v)\cdot\gamma$$ for some $u\in\Omega_{\mathbf H_{w}}$, $v\in\Omega_{\mathbf F_{w}}$ and $\gamma\in\Gamma\cap R_u(\overline{\mathbf P}_0)$. Then we have $$w\cdot g=(w\cdot u\cdot w^{-1})\cdot w\cdot v\cdot\gamma$$ and $w\cdot v$ is rational with $$\height(w\cdot g)=\height(w\cdot v)= L.$$ By Corollary~\ref{c23}, we know that $v\in\mathbf F_{w}(\mathbb Q)$. This completes the proof of the lemma.
\end{proof}
\begin{definition}\label{d613}
For any $L>0$,  we define $$\mathcal S(L):=\{(x\cdot y)(\Gamma\cap R_u(\overline{\mathbf P}_0)(\mathbb R)): x\in\Omega_{\mathbf H_{w}}, y\in\Omega_{\mathbf F_{w}}\cap\mathbf F_{w}(\mathbb Q), \height(w\cdot y)= L\}$$ $$\mathcal F(L):=\{y\in\Omega_{\mathbf F_{w}}\cap\mathbf F_{w}(\mathbb Q): \height(w\cdot y)= L\}.$$
\end{definition}
Now we construct open covers of $\pi_{R_u(\overline{\mathbf P}_0)}(S_\tau^c\cap\mathcal E_{w,\epsilon}(\psi))$ with arbitrarily small diameters. Here we use the right-invariant metric $d_{R_u(\overline{\mathbf P}_0)(\mathbb R)}$ on $R_u(\overline{\mathbf P}_0)(\mathbb R)$ which induces a metric on the quotient space $R_u(\overline{\mathbf P}_0)(\mathbb R)/(\Gamma\cap R_u(\overline{\mathbf P}_0)(\mathbb R))$. Note that by Corollary~\ref{c68}, for any $\epsilon>0$ and $R>0$, we have $$\Vol(E_{w}(R))\lesssim_\epsilon R^{a_{w}+\epsilon}$$ where the implicit constant depends only on $\epsilon$ and $w\in{_{\mathbb Q}\overline{\mathcal W}}$. Let $\nu_0$ be any $\mathbb Q$-root in $R_u(\overline{\mathbf P}_0)$ such that $$\nu_0(a_1)=\max\{\alpha(a_1):\alpha\in\Phi(R_u(\overline{\mathbf P}_{0}))\}.$$ Since the unstable horospherical subgroup of $\{a_t\}_{t\in\mathbb R}$ is contained in $R_u(\overline{\mathbf P}_0)$, we have $\nu_0(a_1)>0$.

Fix a small number $\delta>0$. Let $L>0$ and $t\in\mathbb N$ such that $$L\leq e^{\lambda_w(\epsilon)\cdot t}.$$ Let $q$ be a rational element in $\mathbf G(\mathbb R)$ such that $\height(q)=L$ and $w^{-1}q\in R_u(\overline{\mathbf P}_0)(\mathbb R).$ The following subset $$\left(a_{-t}\cdot E_{w}\left(e^{\lambda_w(\epsilon)\cdot t}/\height(q)\right)\cdot a_t\right)\cdot w^{-1}q\cdot(\Gamma\cap R_u(\overline{\mathbf P}_0)(\mathbb R))$$ can be covered by disjoint boxes of diameter at most $\delta\cdot\exp(\nu_0(a_{-t}))$, and by Lemma~\ref{l611} and Lemma~\ref{l612}, we know that these boxes are contained in
\begin{align*}
&B_{R_u(\overline{\mathbf P}_0)}(\delta\cdot\exp(\nu_0(a_{-t})))\left(a_{-t}\cdot E_{w}\left(e^{\lambda_w(\epsilon)\cdot t}/\height(q)\right)\cdot a_t\right)\cdot w^{-1}q\cdot(\Gamma\cap R_u(\overline{\mathbf P}_0)(\mathbb R))\\
\subset&\left(a_{-t}\cdot B_{R_u(\overline{\mathbf P}_0)}(\delta)\cdot E_{w}\left(e^{\lambda_w(\epsilon)\cdot t}/\height(q)\right)\cdot a_t\right)\cdot w^{-1}q\cdot(\Gamma\cap R_u(\overline{\mathbf P}_0)(\mathbb R))\\\subset&\left(a_{-t}\cdot E_{w}\left(C_\delta\cdot e^{\lambda_w(\epsilon)\cdot t}/L\right)\cdot a_t\right)\cdot\mathbf H_{w}(\mathbb R)\cdot\mathcal S(L).
\end{align*}
Since every point $g(\Gamma\cap R_u(\overline{\mathbf P}_0)(\mathbb R))$ in $\mathbf H_{w}(\mathbb R)\cdot\mathcal S(L)$ satisfies the property that $w\cdot g$ is rational and $\height(w\cdot g)=L$, it follows again from Lemma~\ref{l612} that $$\mathbf H_{w}(\mathbb R)\cdot\mathcal S(L)=\mathcal S(L).$$ We conclude that the subset $$\left(a_{-t}E_{w}(e^{\lambda_w(\epsilon)\cdot t}/\height(q))a_t\right)\cdot w^{-1}q\cdot(\Gamma\cap R_u(\overline{\mathbf P}_0)(\mathbb R))$$ can be covered by disjoint boxes with diameter at most $\delta\cdot\exp(\nu_0(a_{-t}))$ and these boxes are contained in $$a_{-t}E_{w}\left(C_\delta\cdot e^{\lambda_w(\epsilon)\cdot t}/L\right)a_t\cdot\mathcal S(L).$$ We collect all these boxes constructed above in a family which is denoted by $\mathcal P_{t,L}(q)$, and define $$\mathcal P_{t,L}=\bigcup_{\substack{\height(q)=L\\w^{-1}q\in R_u(\overline{\mathbf P}_0)(\mathbb R)}}\mathcal P_{t,L}(q).$$ 

Now since every box in $\mathcal P_{t,L}$ is contained in the same bounded set $$a_{-t}E_{w}(C_\delta\cdot e^{\lambda_w(\epsilon)\cdot t}/L)a_t\cdot\mathcal S(L)$$ we can choose a maximal finite sub-collection $\mathcal Q_{t,L}$ of disjoint boxes in $\mathcal P_{t,L}$. Then by the maximality of $\mathcal Q_{t,L}$ we have $$B_{R_u(\overline{\mathbf P}_0)}(\delta\cdot\exp(\nu_0(a_{-t})))^2\cdot\bigcup_{S\in\mathcal Q_{t,L}}S\supset\bigcup_{S\in\mathcal P_{t,L}}S.$$ Moreover, the number of boxes in $\mathcal Q_{t,L}$ is at most (up to a constant) $$\frac{\mu_{R_u(\overline{\mathbf P}_0)}((a_{-t}E_{w}(C_\delta\cdot e^{\lambda_w(\epsilon)\cdot t}/L)a_t)\cdot\mathcal S(L))}{\delta^{\dim R_u(\overline{\mathbf P}_0)}\cdot\exp(\dim R_u(\overline{\mathbf P}_0)\cdot\nu_0(a_{-t}))}.$$ Now we define $\mathcal G_t$ to be the collection of boxes in $$B_{R_u(\overline{\mathbf P}_0)}(\delta\cdot\exp(\nu_0(a_{-t})))^2\cdot\mathcal Q_{t,L}$$ where $$\theta_w\leq L\leq e^{\lambda_w(\epsilon)\cdot t}.$$ Note that $\diam\mathcal G_t\asymp e^{\nu_{0}(a_{-t})}$ and $\nu_0(a_{-1})\neq 0$. Define $$\mathcal G^k=\bigcup_{t\geq k}\mathcal G_t.$$ Then by definition and Proposition~\ref{p610}, the subset $\mathcal G^k$ is a cover of $\pi_{R_u(\overline{\mathbf P}_0)}(S_\tau^c\cap\mathcal E_{w,\epsilon}(\psi))$ for any $k\in\mathbb N$ and $\diam\mathcal G^k\asymp e^{\nu_{0}(a_{-k})}$. 

Now according to the construction of the open covers $\mathcal G^k$ $(k\in\mathbb N)$ above, we consider the following series with respect to the parameter $s$
\begin{align*}
&\sum_{B\in\mathcal G^k}\diam(B)^s\leq\sum_{t\in\mathbb N}\sum_{B\in\mathcal G_t}\diam(B)^s\\
\lesssim&\sum_{t\in\mathbb N}\frac{\mu_{R_u(\overline{\mathbf P}_0)}(a_{-t}E_{w}(C_\delta\cdot e^{\lambda_w(\epsilon)\cdot t}/L)a_t)\cdot\mathcal S(L))}{\exp(\dim R_u(\overline{\mathbf P}_0)\cdot\nu_0(a_{-t}))}\cdot e^{s\nu_0(a_{-t})}.
\end{align*}
By Corollary~\ref{c68}, for any $\epsilon>0$, we have
\begin{align*}
&\sum_{B\in\mathcal G^k}\diam(B)^s\\
\lesssim&\sum_{t\in\mathbb N}\sum_{\theta_w\leq L\leq e^{\lambda_w(\epsilon)\cdot t}}(e^{\lambda_w(\epsilon)\cdot t}/L)^{a_{w}+\epsilon}\cdot|\mathcal F(L)|\cdot e^{\sum_{\alpha\in\Phi(\mathbf F_{w})}\alpha(a_{-t})}e^{-\dim R_u(\overline{\mathbf P}_0)\cdot \nu_{0}(a_{-t})}\cdot e^{s\nu_{0}(a_{-t})}.
\end{align*} 
Fix $t\in\mathbb N$. Note that $\tau\leq\beta_0(wa_{-1}w^{-1})$. Choose $t_0\in\mathbb Z$ such that $e^{\lambda_w(\epsilon)\cdot t_0}\leq\theta_w.$ Then by the definition of $A_w$, we have
\begin{align*}
&\sum_{\theta_w\leq L\leq e^{\lambda_w(\epsilon)\cdot t}}|\mathcal F(L)|/L^{a_{w}+\epsilon}\\
\lesssim&\sum_{t_0\leq l\leq t-1}\sum_{e^{\lambda_w(\epsilon)\cdot l}\leq L\leq e^{\lambda_w(\epsilon)\cdot(l+1)}}|\mathcal F(L)|/L^{a_{w}+\epsilon}\\
\lesssim&\sum_{t_0\leq l\leq t-1}\frac1{(e^{\lambda_w(\epsilon)\cdot l})^{a_{w}+\epsilon}}\sum_{L\leq e^{\lambda_w(\epsilon)\cdot (l+1)}}|\mathcal F(L)|\\
\lesssim&\sum_{t_0\leq l\leq t-1}\frac1{(e^{\lambda_w(\epsilon)\cdot l})^{a_{w}+\epsilon}}(e^{\lambda_w(\epsilon)\cdot (l+1)})^{A_{w}+\epsilon}\\
\asymp&\sum_{t_0\leq l\leq t-1}(e^{\lambda_w(\epsilon)\cdot l})^{A_{w}-a_{w}}.
\end{align*}
Combining all the inequalities above, we obtain
\begin{align}\label{eqn3}
&\sum_{B\in\mathcal G^k}\diam(B)^s\\
\lesssim&\sum_{t\in\mathbb N}(e^{\lambda_w(\epsilon)\cdot t})^{a_{w}+\epsilon}\cdot e^{\sum_{\alpha\in\Phi(\mathbf F_{w})}\alpha(a_{-t})}e^{-\dim R_u(\overline{\mathbf P}_0)\cdot \nu_{0}(a_{-t})}\cdot e^{s\nu_{0}(a_{-t})}\cdot\sum_{t_0\leq l\leq t-1}(e^{\lambda_w(\epsilon)\cdot l})^{A_{w}-a_{w}}.\nonumber
\end{align}
Computing the series~\eqref{eqn3} in the cases $A_{w}<a_{w}$, $A_{w}=a_{w}$ and $A_{w}>a_{w}$, one can conclude that the series~\eqref{eqn3} converges if $$s>\dim R_u(\overline{\mathbf P}_0)-\sum_{\alpha\in\Phi(\mathbf F_{w})}\frac{\alpha(a_1)}{\nu_{0}(a_1)}+\frac{\lambda_w(\epsilon)}{\nu_{0}(a_1)}\cdot(\max\{A_{w},a_{w}\}+\epsilon).$$ This implies that for sufficiently small $\epsilon>0$ 
\begin{align*}
&\dim_H(S_\rho(\psi)^c\cap\mathcal E_{w,\epsilon}(\psi))\\
\leq&\dim R_u(\overline{\mathbf P}_0)-\sum_{\alpha\in\Phi(\mathbf F_{w})}\frac{\alpha(a_1)}{\nu_{0}(a_1)}+\frac{\lambda_w(\epsilon)}{\nu_{0}(a_1)}\cdot(\max\{A_{w},a_{w}\}+\epsilon).
\end{align*}
By taking $\epsilon\to0$ and Proposition~\ref{p66}, we conclude that
 \begin{align*}
 &\dim_H(S_\rho(\psi)^c\cap R_u(\overline{\mathbf P}_0)(\mathbb R))\\
 &\leq\max_{\substack{w\in{_{\mathbb Q}\overline{\mathcal W}}\\ \beta_0(wa_{-1}w^{-1})\\\geq\tau}}\left\{\dim R_u(\overline{\mathbf P}_0)-\sum_{\alpha\in\Phi(\mathbf F_{w})}\frac{\alpha(a_1)}{\nu_{0}(a_1)}+\frac{(\beta_0(wa_{-1}w^{-1})-\tau)}{\nu_{0}(a_1)}\cdot\max\{A_{w},a_{w}\}\cdot\dim V_{\beta_0}\right\}.
\end{align*}

\begin{proof}[Proof of Theorem~\ref{mthm12}]
By the Bruhat decomposition of $\mathbf G$, we have 
\begin{align*}
\mathbf G(\mathbb R)=\bigcup_{w\in{_{\mathbb Q}\mathcal W}}\mathbf P_{0}(\mathbb R)\cdot w\cdot \mathbf P_{0}(\mathbb R)=\bigcup_{w\in{_{\mathbb Q}\mathcal W}}\mathbf P_{0}(\mathbb R)\cdot w\cdot R_u(\mathbf P_{0})(\mathbb R).
\end{align*}
For any $w\in{_{\mathbb Q}\mathcal W}$, define $$\mathbf U_{w}=(w^{-1}\cdot R_u(\overline{\mathbf P}_0)\cdot w)\cap R_u(\mathbf P_0)\textup{ and }\mathbf Q_{w}=(w^{-1}\cdot\mathbf P_0\cdot w)\cap R_u(\mathbf P_0).$$ Then $R_u(\mathbf P_0)=\mathbf Q_{w}\cdot\mathbf U_{w}$ and we have
\begin{align*}
\mathbf G(\mathbb R)=&\bigcup_{w\in{_{\mathbb Q}\mathcal W}}\mathbf P_{0}(\mathbb R)\cdot w\cdot (\mathbf Q_{w}(\mathbb R)\cdot\mathbf U_{w}(\mathbb R))\\
=&\bigcup_{w\in{_{\mathbb Q}\mathcal W}}\mathbf P_{0}(\mathbb R)\cdot (w\cdot \mathbf U_{w}(\mathbb R)\cdot w^{-1})\cdot w\\
=&\mathbf P_{0}(\mathbb R)\cdot R_u(\overline{\mathbf P}_0)\cdot{_{\mathbb Q}\mathcal W}.
\end{align*}
By definition, if an element $g\in\mathbf G(\mathbb R)$ belongs to $S_\rho(\psi)$, then for any $h\in\mathbf P_0(\mathbb R)$ and any $w\in\mathbf G(\mathbb Q)$, $h\cdot g\cdot w$ also belongs to $S_\rho(\psi)$. Therefore, we have $$S_\rho(\psi)^c=\mathbf P_0(\mathbb R)\cdot(S_\rho(\psi)^c\cap R_u(\overline{\mathbf P}_0)(\mathbb R))\cdot{_{\mathbb Q}\mathcal W}$$ and
\begin{align*}
&\dim_HS_\rho(\psi)^c=\dim\mathbf P_0+\dim_H(S_\rho(\psi)^c\cap R_u(\overline{\mathbf P}_0)(\mathbb R))\\
\leq&\max_{\substack{w\in{_{\mathbb Q}\overline{\mathcal W}}\\ \beta_0(wa_{-1}w^{-1})\geq\tau}}\left\{\dim\mathbf G-\sum_{\alpha\in\Phi(\mathbf F_{w})}\frac{\alpha(a_1)}{\nu_{0}(a_1)}+\frac{(\beta_0(wa_{-1}w^{-1})-\tau)}{\nu_{0}(a_1)}\cdot\max\{A_{w},a_{w}\}\cdot\dim V_{\beta_0}\right\}.
\end{align*}
This completes the proof of Theorem~\ref{mthm12}.
\end{proof}

\section{Proofs of Theorems~\ref{thm10},~\ref{thm11} and \ref{thm12}}\label{cor}
In this section, we derive Theorems~\ref{thm10} from Theorems~\ref{mthm11} and~\ref{mthm12}, and then prove Theorems~\ref{thm11} and \ref{thm12}. Let $\mathbf G$ be a connected $\mathbb Q$-simple group. We assume that the highest weight $\beta_0$ in $\rho:\mathbf G\to\GL(V)$ is a multiple of $\sum_{\alpha\in\Phi(R_u(\overline{\mathbf P}_{\beta_0}))}{\alpha}$ and write $$\beta_0=\kappa\cdot\sum_{\alpha\in\Phi(R_u(\overline{\mathbf P}_{\beta_0}))}{\alpha}$$ for some constant $\kappa<0$,

For any $w\in{_{\mathbb Q}\overline{\mathcal W}}$, define $$\mathbf F^{w}=R_u(\mathbf P_0)\cap w^{-1}R_u(\overline{\mathbf P}_{\beta_0})w,\quad\mathbf H^{w}=R_u(\mathbf P_0)\cap w^{-1}\mathbf P_{\beta_0}w.$$ Since $R_u(\overline{\mathbf P}_0)$ and $R_u(\mathbf P_0)$ contains all the $\mathbb Q$-roots of $\mathbf G$ (with multiplicities), by \cite[Proposition 21.9]{B91}, we have $$\mathbf F^w\cdot \mathbf F_w=w^{-1}R_u(\overline{\mathbf P}_{\beta_0})w,\quad(w\mathbf F^w w^{-1})\cdot(w\mathbf F_ww^{-1})=R_u(\overline{\mathbf P}_{\beta_0}).$$ So by \cite[Proposition 5.26]{K86}, the Haar measure $\mu_{R_u(\overline{\mathbf P}_{\beta_0})}$ on $R_u(\overline{\mathbf P}_{\beta_0})(\mathbb R)$ can be written as a product of the Haar measure $\mu_{w\mathbf F^ww^{-1}}$ on $w\mathbf F^w(\mathbb R) w^{-1}$ and the Haar measure $\mu_{w\mathbf F_ww^{-1}}$ on $w\mathbf F_w(\mathbb R)w^{-1}$. 

Note that in the representation $\rho_{\beta_0}:\mathbf G\to\GL(\bigwedge^{\dim V_{\beta_0}}V)$ (the $\dim V_{\beta_0}$-exterior product of $\rho$), the linear space spanned by $$e_{V_{\beta_0}}:=e_1\wedge e_2\wedge\cdots\wedge e_{\dim V_{\beta_0}}\in\bigwedge^{\dim V_{\beta_0}}V$$ is the unique weight space of highest weight $\dim V_{\beta_0}\cdot\beta_0$ relative to $\mathbf T$. Let $W_{\beta_0}\subset\bigwedge^{\dim V_{\beta_0}}V$ be the irreducible sub-representation of $\mathbf G$ in $\rho_{\beta_0}$ such that $W_{\beta_0}$ contains $\mathbb C\cdot e_{V_{\beta_0}}$. Then $\Psi_w$ and $a_w, b_w, c_w, A_w$ ($w\in{_{\mathbb Q}\overline{\mathcal W}}$), which are defined in \S\ref{ubd}, can be considered as morphisms and quantities defined in the sub-representation $W_{\beta_0}$.

\begin{lemma}\label{l71}
For any $w\in{_{\mathbb Q}\overline{\mathcal W}}$, we have $a_w\leq a_e$ and $A_w\leq A_e$.
\end{lemma}
\begin{proof}
Let $w\in{_{\mathbb Q}\overline{\mathcal W}}$. The inequality $A_w\leq A_e$ follows from the fact that $w\mathbf F_ww^{-1}\subseteq R_u(\overline{\mathbf P}_{\beta_0})=\mathbf F_e$. 

Now let $B_{w\mathbf F^ww^{-1}}(1)$ be the open ball of radius one around the identity in $w\mathbf F^w(\mathbb R)w^{-1}$. For any $R>0$, consider the subset $$B_{w\mathbf F^w w^{-1}}(1)\cdot(w\cdot\Psi_w^{-1}(B_R)\cdot w^{-1})$$ in $R_u(\overline{\mathbf P}_{\beta_0})(\mathbb R)=\mathbf F_e(\mathbb R)$. We know that there exists a constant $C>0$ such that for any $f\in B_{w\mathbf F^w w^{-1}}(1)$ $$\|\rho_{\beta_0}(f)\cdot v\|\leq C\|v\|\quad\left(\forall v\in\bigwedge^{\dim\beta_0}V\right).$$ By the definitions of the morphisms $\Psi_w$ and $\Psi_e$, it implies that for any $R>0$ $$B_{w\mathbf F^w w^{-1}}(1)\cdot(w\cdot\Psi_w^{-1}(B_R)\cdot w^{-1})\subset\Psi_e^{-1}(B_{C\cdot R})$$ and we have
\begin{align*}
&\mu_{w\mathbf F^ww^{-1}}(B_{w\mathbf F^w w^{-1}}(1))\cdot\mu_{\mathbf F_w}(\Psi_w^{-1}(B_R))\\
=&\mu_{w\mathbf F^ww^{-1}}(B_{w\mathbf F^w w^{-1}}(1))\cdot\mu_{w\mathbf F_ww^{-1}}(w\cdot\Psi_w^{-1}(B_R)\cdot w^{-1})\\
\leq&\mu_{\mathbf F_e}(\Psi_e^{-1}(B_{C\cdot R})).
\end{align*}
Note that for any $w\in{_{\mathbb Q}\overline{\mathcal W}}$ $$\mu_{\mathbf F_{w}}(\Psi_{w}^{-1}(B_R))\sim c_{w}R^{a_{w}}(\log R)^{b_{w}}\;(\textup{as }R\to\infty).$$ So $a_w\leq a_e$. This completes the proof of the lemma.
\end{proof}

By the irreducible representation $\rho_{\beta_0}:\mathbf G\to\GL(W_{\beta_0})$, the flag variety $\mathbf G/\mathbf P_{\beta_0}$ can be considered as a subvariety in $\mathbf P(W_{\beta_0})$ via $$g\mathbf P_{\beta_0}\mapsto \rho_{\beta_0}(g)(\mathbb C\cdot e_{V_{\beta_0}}).$$ For any rational point $x\in(\mathbf G/\mathbf P_{\beta_0})(\mathbb Q)\subset\mathbf P(W_{\beta_0})(\mathbb Q)$, define the height of $x$ by $$H_{\beta_0}(x):=\|v\|$$ where $v$ is the unique primitive integral vector (up to sign) in the rational line representing $x\in(\mathbf G/\mathbf P_{\beta_0})(\mathbb Q)$ in $W_{\beta_0}$. Then we may count the number of rational points in $(\mathbf G/\mathbf P_{\beta_0})(\mathbb Q)$ with the height function $\|\cdot\|_{\beta_0}$. By \cite[Theorem 4]{MG14}, there exist $a_0,b_0,c_0>0$ such that $$\left|\{x\in(\mathbf G/\mathbf P_{\beta_0})(\mathbb Q): H_{\beta_0}(x)\leq R\}\right|\sim c_{0}\cdot R^{a_{0}}(\log R)^{b_0}.$$ Moreover, we may compute that $a_0=1/(|\kappa|\cdot\dim V_{\beta_0})$. 

\begin{lemma}\label{l72}
For $e\in{_{\mathbb Q}\overline{\mathcal W}}$, we have $$a_e\leq1/(|\kappa|\cdot\dim V_{\beta_0}),\quad A_e\leq1/(|\kappa|\cdot\dim V_{\beta_0}).$$
\end{lemma}
\begin{proof}
For any $R>0$, let $B_R$ denote the ball of radius $R>0$ around the origin in $W_{\beta_0}$, and $\chi_{R}$ the characteristic function of $B_R$. Let $\Omega_e$ be a bounded fundamental domain of $\mathbf F_e(\mathbb R)/\mathbf F_e(\mathbb Z)$ in $\mathbf F_e(\mathbb R)$. Then $\Omega_e$ is relatively compact. This implies that there exists $C>0$ such that for any $g\in\Omega_e$ $$\|\rho_{\beta_0}(g)\cdot v\|\leq C\|v\|\quad(\forall v\in W_{\beta_0}).$$ So we have 
\begin{align*}
\mu_{\mathbf F_e}(\Psi_e^{-1}(B_R))=&\int_{\mathbf F_e(\mathbb R)}\chi_{R}(u\cdot e_{V_{\beta_0}})d\mu_{\mathbf F_e}(u)\\
=&\sum_{\gamma\in\mathbf F_e(\mathbb Z)}\int_{\Omega_e}\chi_{R}(u\cdot\gamma\cdot e_{V_{\beta_0}})d\mu_{\mathbf F_e}(u)\\
\leq&\sum_{\gamma\in\mathbf F_e(\mathbb Z)}\int_{\Omega_e}\chi_{C\cdot R}(\gamma\cdot e_{V_{\beta_0}})d\mu_{\mathbf F_e}(u)\\
=&\mu_{\mathbf F_e}(\Omega_e)\cdot\sum_{\gamma\in\mathbf F_e(\mathbb Z)}\chi_{C\cdot R}(\gamma\cdot e_{V_{\beta_0}})\\
=&\mu_{\mathbf F_e}(\Omega_e)\cdot|\{\gamma\in\mathbf F_e(\mathbb Z):\|\gamma\cdot e_{V_{\beta_0}}\|\leq C\cdot R\}|\\
\leq&\mu_{\mathbf F_e}(\Omega_e)\cdot\left|\{x\in(\mathbf G/\mathbf P_{\beta_0})(\mathbb Q): H_{\beta_0}(x)\leq C\cdot R\}\right|.
\end{align*}
Here we use the injectivity of the morphism $\Psi_e$. Note that for $e\in{_{\mathbb Q}\overline{\mathcal W}}$ we have $$\mu_{\mathbf F_{e}}(\Psi_{e}^{-1}(B_R))\sim c_{e}R^{a_{e}}(\log R)^{b_{e}}\;(\textup{as }R\to\infty).$$ It implies that $$a_e\leq1/(|\kappa|\cdot\dim V_{\beta_0}).$$ 

Now by the discussion in \S\ref{ubd}, there exists a constant $C_{\beta_0}>0$ depending on $\Omega_{\beta_0}$ such that $$|\{g\in\Omega_{\beta_0}\cap w\mathbf F_{w}w^{-1}(\mathbb Q): \height(g)\leq l\}|\leq\left|\{x\in(\mathbf G/\mathbf P_{\beta_0})(\mathbb Q): H_{\beta_0}(x)\leq C_{\beta_0}\cdot l\}\right|.$$
By taking the limit as $l\to\infty$, we obtain $$A_e\leq1/(|\kappa|\cdot\dim V_{\beta_0}).$$ This completes the proof of the lemma.
\end{proof}

\begin{proof}[Proof of Theorem~\ref{thm10}]
Let $\kappa<0$ such that $$\beta_0=\kappa\cdot\sum_{\alpha\in\Phi(R_u(\overline{\mathbf P}_{\beta_0}))}{\alpha}.$$ By Lemma~\ref{l72}, one can compute that 
\begin{align*}
&\dim\mathbf G-\sum_{\alpha\in\Phi(\mathbf F_{e})}\frac{\alpha(a_1)}{\nu_{0}(a_1)}+\frac{(\beta_0(a_{-1})-\tau(\psi))}{\nu_{0}(a_1)}\cdot\max\{a_{e},A_{e}\}\cdot\dim V_{\beta_0}\\
\leq&\dim\mathbf G-\frac{\tau(\psi)}{\beta_0(a_{-1})\nu_{0}(a_1)}\cdot\sum_{\alpha\in\Phi(R_u(\overline{\mathbf P}_{\beta_0}))}\alpha(a_1).
\end{align*}
Now to derive Theorem~\ref{thm10} from Theorems~\ref{mthm11} and~\ref{mthm12}, it suffices to prove that for any $w\in{_{\mathbb Q}\overline{\mathcal W}}$ and $\tau(\psi)\leq\beta_0(wa_{-1}w^{-1})$
\begin{align*}
&\dim\mathbf G-\sum_{\alpha\in\Phi(\mathbf F_{w})}\frac{\alpha(a_1)}{\nu_{0}(a_1)}+\frac{(\beta_0(wa_{-1}w^{-1})-\tau(\psi))}{\nu_{0}(a_1)}\cdot\max\{a_{w},A_{w}\}\cdot\dim V_{\beta_0}\\
\leq &\dim\mathbf G-\frac{\tau(\psi)}{\beta_0(a_{-1})\nu_{0}(a_1)}\cdot\sum_{\alpha\in\Phi(R_u(\overline{\mathbf P}_{\beta_0}))}\alpha(a_1).
\end{align*}
Note that
\begin{align*}
\beta_0=&\kappa\cdot\sum_{\alpha\in\Phi(R_u(\overline{\mathbf P}_{\beta_0}))}\alpha\\
=&\kappa\cdot\left(\sum_{\alpha\in\Phi(w\mathbf F^ww^{-1})}\alpha+\sum_{\alpha\in\Phi(w\mathbf F_ww^{-1})}\alpha\right)
\end{align*}
and $$\beta_0(wa_{-1}w^{-1})=\kappa\cdot\left(\sum_{\alpha\in\Phi(\mathbf F^w)}\alpha(a_{-1})+\sum_{\alpha\in\Phi(\mathbf F_w)}\alpha(a_{-1})\right).$$
So by Lemma~\ref{l71} and Lemma~\ref{l72}, we have
\begin{align*}
&-\sum_{\alpha\in\Phi(\mathbf F_{w})}\frac{\alpha(a_1)}{\nu_{0}(a_1)}+\frac{(\beta_0(wa_{-1}w^{-1})-\tau(\psi))}{\nu_{0}(a_1)}\cdot\max\{a_{w},A_{w}\}\cdot\dim V_{\beta_0}\\
\leq&-\sum_{\alpha\in\Phi(\mathbf F_{w})}\frac{\alpha(a_1)}{\nu_{0}(a_1)}+\frac{(\beta_0(wa_{-1}w^{-1})-\tau(\psi))}{\nu_{0}(a_1)}\cdot a_0\cdot\dim V_{\beta_0}\\
=&-\sum_{\alpha\in\Phi(\mathbf F_{w})}\frac{\alpha(a_1)}{\nu_{0}(a_1)}+\frac{\beta_0(wa_{-1}w^{-1})}{\nu_{0}(a_1)|\kappa|}-\frac{\tau(\psi)}{\nu_{0}(a_1)|\kappa|}\\
=&\sum_{\alpha\in\Phi(\mathbf F^{w})}\frac{\alpha(a_1)}{\nu_{0}(a_1)}-\frac{\tau(\psi)}{\nu_{0}(a_1)|\kappa|}\\
\leq &-\frac{\tau(\psi)}{\beta_0(a_{-1})\nu_{0}(a_1)}\cdot\sum_{\alpha\in\Phi(R_u(\overline{\mathbf P}_{\beta_0}))}\alpha(a_1).
\end{align*}
Here we use the fact that all $\mathbb Q$-roots in $\Phi(\mathbf F^w)$ are positive and $\alpha(a_1)\leq0$ for any $\alpha\in\Phi(\mathbf F^w)$. This completes the proof of the theorem.
\end{proof}

\begin{proof}[Proofs of Theorems~\ref{thm11} and \ref{thm12}]
Let $\mathbf T$ be the full diagonal group in $\mathbf G=\SL_n$. Without loss of generality, we may write $$\{a_t\}_{t\in\mathbb R}=\left\{\diag(e^{b_1 t}, e^{b_2 t},\cdots, e^{b_n t}):t\in\mathbb R\right\}$$ where $b_1\geq b_2\geq\cdots\geq b_n$ and $b_1+b_2+\cdots+b_n=0$. We write $$\mathbf T(\mathbb R)^0=\{\diag(e^{t_1},e^{t_2},\cdots, e^{t_n}):t_1+t_2+\cdots+t_n=0\}$$ where $\mathbf T(\mathbb R)^0$ is the connected component of identity in $\mathbf T(\mathbb R)$. Let $\mathbf P_0$ be the lower triangular subgroup in $\mathbf G$ and $R_u(\overline{\mathbf P}_0)$ is the upper triangular unipotent subgroup in $\mathbf G$.  Then all the $\mathbb Q$-roots in $\mathbf G$ with respect to $\mathbf T$ are $$\alpha_{i,j}(a)=t_i-t_j\quad(1\leq i\neq j\leq n)$$ where $a=\diag(e^{t_1},e^{t_2},\cdots, e^{t_n})$.

Let $\rho$ be the standard representation of $\mathbf G=\SL_n$. Let $\{e_1,e_2,...,e_n\}$ be the standard basis in $V$ where $$e_1=\left(\begin{array}{c} 1 \\ 0\\0\\ \vdots\\0\end{array}\right)\quad e_2=\left(\begin{array}{c} 0 \\ 1\\0\\ \vdots\\0\end{array}\right)\quad\cdots\cdots\quad e_n=\left(\begin{array}{c} 0 \\ 0\\ \vdots\\0\\1\end{array}\right).$$ We choose $V_{\lambda_0}=V_{\beta_0}=\mathbb C\cdot e_n$ to be the highest weight space and $$V=\mathbb C\cdot e_1\oplus\mathbb C\cdot e_2\oplus\cdots\oplus\mathbb C\cdot e_n$$ is the weight space decomposition of $V$. The stabilizer $\mathbf P_{\beta_0}$ of $V_{\beta_0}$ and $R_u(\overline{\mathbf P}_{\beta_0})$ are $$\mathbf P_{\beta_0}=\left(\begin{array}{cccc} * &*&\cdots& 0 \\ * & *&\cdots& 0\\ \vdots & \vdots &\cdots &\vdots\\ * & * & \cdots & * \end{array}\right)\textup{ and } R_u(\overline{\mathbf P}_{\beta_0})=\left(\begin{array}{cc} I_{n-1} &*\\ 0 & 1 \end{array}\right).$$ One may verify that $\beta_0=(-\frac1n)\cdot\sum_{\alpha\in\Phi(R_u(\overline{\mathbf P}_{\beta_0}))}\alpha$. Theorem~\ref{thm11} follows from Theorem~\ref{thm10}.

Now consider the adjoint representation of $\mathbf G=\SL_n$. Let $V_{\beta_0}$ be the highest weight space in $\mathfrak{sl}_n$ $$V_{\beta_0}=\left(\begin{array}{cccc} 0 & 0 & \cdots & 0\\ \vdots & \vdots & \vdots & \vdots\\ 0 & 0 & \cdots & 0\\ * & 0 & \cdots & 0\end{array}\right).$$ Note that here $\nu_0=-\beta_0$ is the highest root in $\mathfrak{sl}_n$ and $\beta_0(a_t)=(b_n-b_1)t<0$ $(t>0)$. The stabilizer $\mathbf P_{\beta_0}$ of $V_{\beta_0}$ and $R_u(\overline{\mathbf P}_{\beta_0})$ are $$\mathbf P_{\beta_0}=\left(\begin{array}{ccccc} * & 0 & \cdots & 0 & 0 \\  * & * & \cdots & * & 0 \\ \vdots & \vdots & \cdots & \vdots & \vdots \\ * & * & \cdots & * & 0 \\ * & * & \cdots & * & *  \end{array}\right)\textup{ and }R_u(\overline{\mathbf P}_{\beta_0})=\left(\begin{array}{ccccc} 1 & * & \cdots & * & * \\  0 & 1 & \cdots & 0 & * \\ \vdots & \vdots & \cdots & \vdots & \vdots \\ 0 & 0 & \cdots & 1 & * \\ 0 & 0 & \cdots & 0 & 1  \end{array}\right).$$ One may verify that $\beta_0=(-\frac1{n-1})\cdot\sum_{\alpha\in\Phi(R_u(\overline{\mathbf P}_{\beta_0}))}\alpha$. Theorem~\ref{thm12} then follows from Theorem~\ref{thm10}.
\end{proof}

Let us explain how to deduce \cite[Corollary 1.3]{FZ22} (or equivalently \cite[Theorem 1.2]{FZ22}) from Theorem~\ref{thm12}. In \cite{FZ22}, we consider a regular one-parameter diagonal subgroup $\{a_t\}_{t\in\mathbb R}$ acting on the homogeneous space $X_3=\SL_3(\mathbb R)/\SL_3(\mathbb Z)$. According to \cite[Definition 1.1]{FZ22}, a point $p=g\cdot\SL_3(\mathbb Z)\in X_3$ is Diophantine of type $\tau$ if and only if there exists a constant $C>0$ such that for any $t>0$ $$\eta(a_t\cdot p)\geq Ce^{-\tau t}$$ where $\eta$ is the injectivity radius function on $\SL_3(\mathbb R)/\SL_3(\mathbb Z)$. Now let $p=g\SL_3(\mathbb Z)\in X_3$ which is not Diophantine of type $\tau$. Then by definition, for any $\epsilon>0$, there exists $t_\epsilon>0$ such that $$\eta(a_{t_\epsilon}\cdot p)<\epsilon e^{-\tau t_\epsilon}.$$ By \cite[Corollary 11.18]{R72}, for sufficiently small $\epsilon>0$, there exists a unipotent element $u_\epsilon\in\SL_3(\mathbb Z)\setminus\{e\}$ such that $g\cdot u_\epsilon\cdot g^{-1}\in\operatorname{Stab}(p)=g\SL_3(\mathbb Z)g^{-1}$ (the stabilizer of $p$) and $$d_{\SL_3}(a_{t_\epsilon}g\cdot u_\epsilon\cdot g^{-1}a_{-t_\epsilon},e)<\epsilon e^{-\tau t_\epsilon}$$ where $d_{\SL_3}$ is the metric on $\SL_3(\mathbb R)$ induced by a norm $\|\cdot\|_{\mathfrak{sl}_3}$ on the Lie algebra $\mathfrak{sl}_3(\mathbb R)$ of $\SL_3(\mathbb R)$. Note that $u_\epsilon$ is unipotent $$\log u_\epsilon=(u_\epsilon-I_3)-\frac{(u_\epsilon-I_3)^2}{2}$$ and $2\log u_\epsilon\in\mathfrak{sl}_3(\mathbb Z)\setminus\{0\}$. Then one can deduce that there exists a constant $\tilde C>0$ depending only on $\SL_3$ such that $$\|\Ad_{\SL_3}(a_{t_\epsilon}g)(2\log u_\epsilon)\|_{\mathfrak{sl}_3}<\tilde C\epsilon e^{-\tau t_\epsilon}$$ where $\rho=\Ad_{\SL_3}$ is the adjoint representation of $\SL_3$. This implies that $g\in S_{\rho}(\psi)^c$ for $\psi(t)=e^{-\tau t}$ by Definition~\ref{def11}.

Conversely, let $g\in\SL_3(\mathbb R)$ such that $g\in S_{\rho}(\psi)^c$ where $\rho=\Ad_{\SL_3}$ and $\psi(t)=e^{-\tau t}$. Then by Definition~\ref{def11}, for any $\epsilon>0$, there exists $t_\epsilon>0$ such that $$\delta(\Ad_{\SL_3}(a_{t_\epsilon}\cdot g)\mathfrak{sl}_3(\mathbb Z))<\epsilon\cdot e^{-\tau t_\epsilon}.$$ By \cite[Proposition 3.3]{TW03}, for sufficiently small $\epsilon>0$, there exists a nilpotent element $n_\epsilon\in\mathfrak{sl}_3(\mathbb Z)\setminus\{0\}$ such that $$\|\Ad_{\SL_3}(a_{t_\epsilon}\cdot g) n_\epsilon\|_{\mathfrak{sl}_3}<\epsilon\cdot e^{-\tau t_\epsilon}$$ where $\|\cdot\|_{\mathfrak{sl}_3}$ is a norm on the Lie algebra $\mathfrak{sl}_3(\mathbb R)$ of $\SL_3(\mathbb R)$. Note that $$\exp(2n_\epsilon)=I_3+(2n_\epsilon)+(2n_\epsilon)^2/2\in\SL_3(\mathbb Z)\setminus\{e\}.$$ Then one can deduce that there exists a constant $\tilde C>0$ depending only on $\SL_3$ such that $$d_{\SL_3}(a_{t_\epsilon}g\cdot\exp(2n_\epsilon)g^{-1}a_{-t_\epsilon},e)<\tilde C\cdot\epsilon\cdot e^{-\tau t_\epsilon}.$$ By \cite[Definition 1.1]{FZ22}, this implies that $g\SL_3(\mathbb Z)$ is not Diophantine of type $\tau$. Consequently, Theorem~\ref{thm12} and Remark~\ref{r19} imply \cite[Corollary 1.3]{FZ22} when $\mathbf G=\SL_3$.

\begin{remark}\label{r73}
The arithmetic case in \cite[Theorem 1.1]{Z19} concerns diagonalizable flows on $\mathbf G(\mathbb R)/\mathbf G(\mathbb Z)$ where $\mathbf G$ is a simple $\mathbb Q$-group with $\mathbb Q$-rank and $\mathbb R$-rank both equal to one. Using the same argument as above, we can also derive this arithmetic case from Theorem~\ref{thm10} with the adjoint representation of $\mathbf G$.
\end{remark}

\section{Connections to Diophantine approximation}\label{Diophantine}
In this section, we prove Theorems~\ref{thm13} and~\ref{thm15}. 

\subsection{Connection to Diophantine approximation on flag varieties}
We first prove Theorem~\ref{thm13}. To study the Diophantine approximation on $\mathbf X(\mathbb R)$, or specifically, to study the Diophantine subset $E_{\mathbf X}(\psi)$, we need to study the $\{a_t\}_{t\in\mathbb R}$-action on the space $\mathcal X$ of lattices in $V(\mathbb R)$ (see \cite[Chapter 1, Section 2]{d21}). For any $x\in\mathbf X(\mathbb R)$, write $x=\mathbf P(\mathbb R)\cdot s_x$ for some $s_x\in\mathbf G(\mathbb R)$ and define $\Lambda_x=s_x\cdot\mathbb Z^d\in\mathcal X$. Note that for any $g\in\mathbf P(\mathbb R)$, the set $\{a_t\cdot g\cdot a_{-t}\}_{t\geq0}$ is bounded in $\mathbf G(\mathbb R)$. So for any $s_1$ and $s_2$ with $\mathbf P(\mathbb R)\cdot s_1=\mathbf P(\mathbb R)\cdot s_2$, the forward orbits $\{a_t\cdot s_1\mathbb Z^d\}_{t\geq0}$ and $\{a_t\cdot s_2\mathbb Z^d\}_{t\geq0}$ in $\mathcal X$ stay within bounded distance from each other for any $t>0$, and the behaviors of $\{a_t\cdot s_1\mathbb Z^d\}_{t\geq0}$ and $\{a_t\cdot s_2\mathbb Z^d\}_{t\geq0}$ are almost identical.

Let $\mathcal C_\chi:=\mathbf G(\mathbb R)\cdot V_\chi\setminus\{0\}$. Let $\pi_{\chi}:V\to V_\chi$ be the projection from $V$ to $V_\chi$ along all the other weight spaces in $V$. For any lattice $\Lambda\in\mathcal X$, define $$\delta_\chi(\Lambda)=\inf\left\{\|v\|:v\in\Lambda\cap\mathcal C_\chi,\|\pi_\chi(v)\|\geq\frac12\|v\|\right\}.$$ We will need the following important correspondence theorem.

\begin{proposition}[{\cite[Proposition 3.2.4]{d21}}]\label{p82}
Let $x\in\mathbf X(\mathbb R)$ and $s_x\in\mathbf G(\mathbb R)$ such that $x=\mathbf P(\mathbb R)\cdot s_x$. Let $\beta_\chi>0$ be the Diophantine exponent defined in Theorem~\ref{thm110}. Then there exists a constant $C>0$ such that the following holds. Let $\psi:\mathbb R_+\to\mathbb R_+$ be a decreasing function, and $\Psi$ the function defined by $$\Psi(u)=Cu^{-\beta_\chi}\psi(u).$$ If the inequality $d_{\textup{CC}}(x,v)\leq H_\chi(v)^{-\beta_\chi}\psi(H_\chi(v))$ admits infinitely many solutions $v\in\mathbf X(\mathbb Q)$, then there exists arbitrarily large $t>0$ such that $$\delta_\chi(a_ts_x\mathbb Z^d)\leq 2e^{-t/\beta_\chi}\Psi^{-1}(e^{-t}).$$ If we have that there exists $t>0$ arbitrarily large such that $$\delta_\chi(a_ts_x\mathbb Z^d)\leq e^{-t/\beta_\chi}\Psi^{-1}(e^{-t})$$ then the inequality $$d_{\textup{CC}}(x,v)\leq C^2 H_\chi(v)^{-\beta_\chi}\psi(H_\chi(v))$$ admits infinitely many solutions $v\in\mathbf X(\mathbb Q)$.
\end{proposition}

\begin{proof}[Proof of Theorem~\ref{thm13}]
Since $\psi$ is decreasing, we have $\psi(t)\leq\psi(1)$ $(t\geq1)$ which implies that $\gamma(\psi)\geq0$.

We first compute an upper bound for the Hausdorff dimension of $E_{\mathbf X}(\psi)^c$. Let $x\in E_{\mathbf X}(\psi)^c$. From Definition~\ref{def111}, one may deduce that there exist infinitely many $v\in\mathbf X(\mathbb Q)$ such that $$d_{\textup{CC}}(x,v)\leq (H_\chi(v))^{-\beta_\chi}\psi(H_\chi(v)).$$ By Proposition~\ref{p82}, this implies that there exists $t>0$ arbitrarily large such that $$\delta_\chi(a_ts_x\mathbb Z^d)\leq 2e^{-t/\beta_\chi}\Psi^{-1}(e^{-t}).$$ Note that $\delta\leq\delta_\chi$. So by Definition~\ref{def11}, for any $\epsilon>0$, $s_x\in S_\rho(2e^{-t(1-\epsilon)/\beta_\chi}\Psi^{-1}(e^{-t}))^c$. This implies that $$\dim\mathbf P+\dim_HE_{\mathbf X}(\psi)^c\leq\dim_HS_\rho(2e^{-t(1-\epsilon)/\beta_\chi}\Psi^{-1}(e^{-t}))^c.$$ By Theorem~\ref{thm10}, we have $$\dim_HE_{\mathbf X}(\psi)^c\leq\dim\mathbf X-\frac{\tau}{\beta_0(a_{-1})\nu_0(a_1)}\cdot\sum_{\alpha\in\Phi(R_u(\mathbf P))}\alpha(a_{-1})$$ where $\tau=\tau(2e^{-t(1-\epsilon)/\beta_\chi}\Psi^{-1}(e^{-t}))$. One may compute that $$\tau=\frac{1-\epsilon}{\beta_\chi}-\frac1{\gamma(\psi)+\beta_\chi}.$$ Let $\epsilon\to0$ and we obtain $$\dim_HE_{\mathbf X}(\psi)^c\leq\dim\mathbf X-\left(\frac1{\beta_\chi}-\frac1{\gamma(\psi)+\beta_\chi}\right)\frac{1}{\beta_0(a_{-1})\nu_0(a_1)}\cdot\sum_{\alpha\in\Phi(R_u(\mathbf P))}\alpha(a_{-1}).$$

Now we compute a lower bound for the Hausdorff dimension of $E_{\mathbf X}(\psi)^c$. Let $\overline{\mathbf P}_0$ be the opposite minimal parabolic $\mathbb Q$-subgroup of $\mathbf P_0$ with the same Levi factor $Z(\mathbf T)$, and denote by $R_u(\overline{\mathbf P}_0)$ the unipotent radical of $\overline{\mathbf P}_0$. In section~\ref{lbd1}, for any function $\phi:\mathbb R_+\to\mathbb R_+$ with $\tau(\phi)\in[0,\chi(a_{-1}))$ and any sufficiently small $\epsilon>0$, we construct a Cantor type subset $\mathbf A_\infty$ in $R_u(\overline{\mathbf P}_0)(\mathbb R)$ such that for any $p\in\mathbf A_\infty$ we have
\begin{align*}
\delta(\rho(a_{t_k}\cdot p)\mathbb Z^d)&\lesssim\phi(t_k)\cdot e^{-\frac\epsilon2 t_k}
\end{align*}
for some divergent sequence $\{t_k\}_{k\in\mathbb N}$. Note that in the construction $$a_{t_k}p\in B_{\mathbf F_e}(\tilde{\delta}_0)\cdot a_{t_k}q_k$$ and there exist $v_k\in\rho(a_{t_k}q_k)\mathbb Z^d\cap V_\chi$ $(V_\chi=V_{\beta_0})$ and $x_k\in B_{\mathbf F_e}(\tilde\delta_0)$ such that $$\rho(x_k)v_k\in\rho(a_{t_k}p)\mathbb Z^d$$ and $$\delta(\rho(a_{t_k}\cdot p)\mathbb Z^d)\leq\|\rho(x_k)v_k\|\asymp e^{\chi(a_{t_k})}\cdot l_k\leq\phi(t_k)\cdot e^{-\frac\epsilon2 t_k}.$$ One may verify that for small $\tilde\delta_0>0$ we have $$\|\pi_\chi(\rho(x_k)v_k)\|
\geq\frac12\|\rho(x_k)v_k\|.$$ This implies that for any $p\in\mathbf A_\infty$ $$\delta_\chi(\rho(a_{t_k}\cdot p)\mathbb Z^d)\lesssim\phi(t_k)\cdot e^{-\frac\epsilon2 t_k}.$$ 

Now we may choose $$\phi(t)=2e^{-t/\beta_\chi}\Psi_{\epsilon}^{-1}(e^{-t})$$ where $\Psi_{\epsilon}(u)=Cu^{-\beta_\chi}\psi(u)u^{-\epsilon}$. By Proposition~\ref{p82}, for $x=\mathbf P(\mathbb R)g$ with $g\in\mathbf A_\infty$ there exist infinitely many $v\in\mathbf X(\mathbb Q)$ such that $$d_{\textup{CC}}(x,v)\leq C^2 H_\chi(v)^{-\beta_\chi}\psi(H_\chi(v))(H_\chi(v))^{-\epsilon}$$ which implies that $x\in E_{\mathbf X}(\psi)^c$. So we have $$\dim_H E_{\mathbf X}(\psi)^c+\dim\mathbf P\geq\dim_H\mathbf A_\infty+\dim\mathbf P_0$$ and $$\dim_HE_{\mathbf X}(\psi)^c\geq\dim\mathbf X-\sum_{\alpha\in\Phi(\mathbf F_e)}\frac{\alpha(a_1)(\tau+\epsilon)}{\beta_0(a_{-1})\nu_0(a_1)}$$ where $\tau=\tau(\phi)=\tau(2e^{-t/\beta_\chi}\Psi_{\epsilon}^{-1}(e^{-t}))$. One may compute that $$\tau=\tau(\phi)=\frac{1}{\beta_\chi}-\frac1{\gamma(\psi)+\epsilon+\beta_\chi}$$ and when $\gamma(\psi)<\infty$, $\tau=\tau(\psi)<\chi(a_{-1})$. Let $\epsilon\to0$ and we obtain that $$\dim_HE_{\mathbf X}(\psi)^c\geq\dim\mathbf X-\left(\frac1{\beta_\chi}-\frac1{\gamma(\psi)+\beta_\chi}\right)\frac{1}{\beta_0(a_{-1})\nu_0(a_1)}\cdot\sum_{\alpha\in\Phi(R_u(\mathbf P))}\alpha(a_{-1}).$$ 

Note that if $\gamma(\psi)=\infty$, then $\tau(\phi)=\chi(a_{-1})$. In this case, $\phi(t)e^{\chi(a_{-1})t}=2\Psi_\epsilon^{-1}(e^{-t})$ is an unbounded function as $\psi$ is decreasing. So we can choose the Cantor-type subset $\mathbf A_\infty$ constructed in Proposition~\ref{p56} and apply the argument above again to conclude that $$\dim_HE_{\mathbf X}(\psi)^c\geq\dim\mathbf X-\sum_{\alpha\in\Phi(\mathbf F_e)}\frac{\alpha(a_1)\tau(\phi)}{\beta_0(a_{-1})\nu_0(a_1)}.$$ This completes the proof of Theorem~\ref{thm13}.
\end{proof}

\subsection{Connection to rational approximation to linear subspaces}
Now we prove Theorem~\ref{thm15}. Let $1\leq k\leq l<n$. We consider the representation $\rho_k:\SL_n\to\GL(V)$ $(V=\bigwedge^k\mathbb C^n)$ defined by the $k$-th exterior product of the standard representation $\rho: \SL_n\to\GL(\mathbb C^n)$ of $\SL_n$. Note that the action of $\SL_n(\mathbb R)$ on the Grassmann variety of $k$-dimensional subspaces in $\mathbb R^n$ is transitive. So the representation $\rho_k$ is irreducible. Let $\{e_1,e_2,\cdots,e_n\}$ be the standard basis in $\mathbb C^n$, where the $i$-th coordinate of $e_i$ equals one and the other coordinates of $e_i$ equals zero $(1\leq i\leq n)$. Then $$\{e_{i_1}\wedge e_{i_2}\wedge\cdots\wedge e_{i_k}:1\leq i_1<i_2<\cdots<i_k\leq n\}$$ forms an integral basis of $\bigwedge^k\mathbb C^n$. We choose $$V_\chi=\mathbb C\cdot e_\chi,\quad e_\chi=e_1\wedge e_2\wedge\cdots\wedge e_k$$ to be the weight space of highest weight $\chi$ in $\bigwedge^k\mathbb C^n$. Denote by $\mathbf P_\chi$ the stablizer of the weight space $\mathbb C\cdot e_\chi$ and by $\pi_\chi:\bigwedge^k\mathbb C^n\to V_\chi$ the projection onto $V_\chi$ along the direct sum of all the other weight spaces. We also choose $\mathbf P_0$ to be the upper triangular subgroup, $\mathbf T$ the full diagonal subgroup and $\overline{\mathbf P}_0$ the lower triangular subgroup in $\SL_n$. Note that $\mathbf P_0$ is a minimal parabolic $\mathbb Q$-subgroup contained in $\mathbf P_\chi$. We denote by $R_u(\mathbf P_\chi)$ (resp. $R_u(\overline{\mathbf P}_0)$) the unipotent radical of $\mathbf P_\chi$ (resp. $\overline{\mathbf P}_0$), and write $\Phi(R_u(\mathbf P_\chi))$ (resp. $\Phi(R_u(\overline{\mathbf P}_0))$) for the set of $\mathbb Q$-roots in $R_u(\mathbf P_\chi)$ (resp. $R_u(\overline{\mathbf P}_0)$) relative to $\mathbf T$. The symbol $\sum_{\alpha\in\Phi(R_u(\mathbf P_\chi))}$ stands for the sum over all $\mathbb Q$-roots $\alpha\in\Phi(R_u(\mathbf P_\chi))$. One may verify that for the irreducible representation $\rho_k$, we have $\chi=n\cdot\sum_{\alpha\in\Phi(R_u(\mathbf P_\chi))}\alpha$. So $\rho_k$ satisfies the assumption of Theorem~\ref{thm10}.

Let $x_0$ be the $l$-dimensional subspace $\mathbb R$-$\textup{span}\{e_1,e_2,\dots,e_l\}$ in $X_l(\mathbb R)$, and $P$ the stabilizer of $x_0$ in $\SL_n(\mathbb R)$. Then the Grassmann variety $X_l(\mathbb R)$ can be identified with $P\backslash\SL_n(\mathbb R)$ by $$P\backslash\SL_n(\mathbb R)\mapsto X_l(\mathbb R),\quad P\cdot g\mapsto g^{-1}\cdot x_0.$$ Let $$a_t=\diag(e^{-t/l},\cdots,e^{-t/l},e^{t/(n-l)},\cdots,e^{t/(n-l)})\subset\SL_n(\mathbb R).$$ The matrix $a_t$ has an eigenvalue $e^{-\frac kl t}$ in $\bigwedge^k\mathbb R^n$. Denote by $\pi^+:\bigwedge^k\mathbb R^n\to\bigwedge^k\mathbb R^n$ the projection to the eigenspace of $a_t$ associated to the eigenvalue $e^{-\frac kl t}$. We write $\nu_{0}$ for any $\mathbb Q$-root relative to $\mathbf T$ in $\Phi(R_u(\overline{\mathbf P}_0))$ satisfying $$\nu_{0}(a_{1})=\max\{\alpha(a_{1}):\alpha\in\Phi(R_u(\overline{\mathbf P}_0))\}.$$

To study the Diophantine subset $E_{l,k}(\psi)$ in $X_l(\mathbb R)$, we need to study the $\{a_t\}_{t\in\mathbb R}$-action on the space of unimodular lattices in $\bigwedge^k\mathbb R^n$. The following is a correspondence theorem which is crucial in the analysis of $E_{l,k}(\psi)$.

\begin{proposition}[{\cite[Proposition 1]{d25}}]\label{p83}
Let $x\in X_l(\mathbb R)$ and $s_x\in\SL_n(\mathbb R)$ such that $x=P\cdot s_x$.
\begin{enumerate}
\item Let $v\in X_k(\mathbb Q)$ be close to $x$, and $t>0$ such that $e^{-t(\frac1l+\frac1{n-l})}=d(v,x)$. Then the pure tensor $\mathbf v\in\bigwedge^k\mathbb Z^n$ associated to $v$ satisfies $$\|\pi^+(a_ts_x\mathbf v)\|\gtrsim\|a_ts_x\mathbf v\|\textup{ and }\|a_ts_x\mathbf v\|\lesssim e^{-t\frac kl}H(v)$$ where the implicit constants depend on the choice of $s_x$.
\item Let $t>0$ and $\mathbf v\in\bigwedge^k\mathbb Z^n$ a pure tensor such that $$\|\pi^+(a_ts_x\mathbf v)\|\geq c\|a_ts_x\mathbf v\|$$ for some fixed constant $c>0$. Then the rational subspace $v\in X_k(\mathbb Q)$ assocaited to $\mathbf v$ satisfies $$H(v)\lesssim e^{\frac kl t}\|a_ts_x\mathbf v\|\textup{ and }d(v,x)\lesssim e^{-t(\frac1l+\frac1{n-l})}$$ where the implicit constants depend on the choice of $s_x$ and $c$.
\end{enumerate}
\end{proposition}

\begin{lemma}\label{l81}
Let $g\in R_u(\overline{\mathbf P}_0)(\mathbb R)\subset\SL_n(\mathbb R)$ be a rational element in the sense of Definition~\ref{def21}. If $\mathbf v\in\bigwedge^k\mathbb Z^n\setminus\{0\}$ is a primitive integral vector in $\bigwedge^k\mathbb R^n$ such that $\rho_k(g) \mathbf v\in V_\chi$, then $\mathbf v$ is a pure tensor in $\bigwedge^k\mathbb Z^n$.
\end{lemma}
\begin{proof}
Let $g$ be a rational element in $R_u(\overline{\mathbf P}_0)(\mathbb R)$. Then by Corollary~\ref{c23}, we may write $g=h\cdot f$ for some $h\in\mathbf H_e(\mathbb R)$ and $f\in\mathbf F_e(\mathbb Q)$. Note that $h$ fixes every element in $V_\chi$. So we have $\rho_k(f)\mathbf v\in V_\chi(=\mathbb C\cdot e_1\wedge\cdots\wedge e_k)$. We know that $\{\rho(f^{-1})e_1,\rho(f^{-1})e_2,\cdots,\rho(f^{-1})e_k\}$ spans a linear $\mathbb Q$-subspace $W_f$ in $\mathbb R^d$, and we may take an integral basis $\{f_1,f_2,\cdots,f_k\}\subset\mathbb Z^n$ for $W_f$. Then $f_1\wedge\cdots\wedge f_k\in\bigwedge^k\mathbb Z^n$ is a primitive integral vector in $\bigwedge^k\mathbb R^n$, and $f_1\wedge\cdots\wedge f_k\in\rho_k(f^{-1}) V_\chi$. It implies that $\mathbf v=f_1\wedge\cdots\wedge f_k$. This completes the proof of the lemma.
\end{proof}

\begin{proof}[Proof of Theorem~\ref{thm15}]
Since $\psi$ is decreasing, we have $\psi(t)\leq\psi(1)$ $(t\geq1)$ which implies that $\gamma(\psi)\geq0$.

We first compute an upper bound for the Hausdorff dimension of $E_{l,k}(\psi)^c$. Let $x\in E_{l,k}(\psi)^c$ and write $x=P\cdot s_x$ $(s_x\in\SL_n(\mathbb R))$. From Definition~\ref{def114}, one may deduce that there exist infinitely many $v\in X_k(\mathbb Q)$ such that $$d(x,v)\leq H(v)^{-\frac n{k(n-l)}}\cdot\psi(H(v)).$$ By Proposition~\ref{p83}(1), for any such $v\in X_k(\mathbb Q)$, there exists $t>0$ such that $$e^{-t(\frac1l+\frac1{n-l})}=d(v,x)$$ and the pure tensor $\mathbf v\in\bigwedge^k\mathbb Z^n$ associated to $v$ satisfies $$\|\pi^+(a_ts_x\mathbf v)\|\gtrsim\|a_ts_x\mathbf v\|\textup{ and }\|a_ts_x\mathbf v\|\lesssim e^{-t\frac kl}H(v)$$ where the implicit constants depend on the choice of $s_x$. Let $\Psi(u)=u^{-\frac n{k(n-l)}}\cdot\psi(u)$. Then we have $$e^{-t(\frac1l+\frac1{n-l})}=d(v,x)\leq\Psi(H(v))\implies H(v)\leq\Psi^{-1}(e^{-t(\frac1l+\frac1{n-l})})$$ and $$\|a_ts_x\mathbf v\|\lesssim e^{-t\frac kl}H(v)\leq e^{-t\frac kl}\cdot\Psi^{-1}(e^{-t(\frac1l+\frac1{n-l})})$$ which implies that $$\delta\left(a_ts_x\bigwedge^k\mathbb Z^n\right)\lesssim e^{-t\frac kl}\cdot\Psi^{-1}(e^{-t(\frac1l+\frac1{n-l})}).$$ So by Definition~\ref{def11}, for any $\epsilon>0$, $s_x\in S_{\rho_k}(e^{-t(1-\epsilon)\frac kl}\cdot\Psi^{-1}(e^{-t(\frac1l+\frac1{n-l})}))^c$. This implies that $$\dim P+\dim_HE_{l,k}(\psi)^c\leq\dim_HS_{\rho_k}(e^{-t(1-\epsilon)\frac kl}\cdot\Psi^{-1}(e^{-t(\frac1l+\frac1{n-l})}))^c.$$ By Theorem~\ref{thm10}, we have $$\dim_HE_{l,k}(\psi)^c\leq\dim X_l-\frac{\tau}{\chi(a_{-1})\nu_0(a_1)}\cdot\sum_{\alpha\in\Phi(R_u(\mathbf P_\chi))}\alpha(a_{-1})$$ where $\tau=\tau(e^{-t(1-\epsilon)\frac kl}\cdot\Psi^{-1}(e^{-t(\frac1l+\frac1{n-l})}))$. One may compute that $$\tau=(1-\epsilon)\frac kl-\frac{1/l+1/(n-l)}{\gamma(\psi)+n/(k(n-l))}$$ and $\nu_0(a_1)=1/l+1/(n-l)$. Let $\epsilon\to0$ and we obtain $$\dim_HE_{l,k}(\psi)^c\leq(l-k)(n-l)+\frac n{n/(k(n-l))+\gamma(\psi)}.$$

Now we compute a lower bound for the Hausdorff dimension of $E_{l,k}(\psi)^c$. Let $\Lambda_0=\bigwedge^k\mathbb Z^n$. In section~\ref{lbd1}, for any function $\phi:\mathbb R\to\mathbb R$ with $\tau(\phi)\in[0,\chi(a_{-1}))$ and any sufficiently small $\epsilon>0$, we construct a Cantor type subset $\mathbf A_\infty$ in $R_u(\overline{\mathbf P}_0)(\mathbb R)$ such that for any $p\in\mathbf A_\infty$ we have
\begin{align*}
\delta(\rho_k(a_{t_i}\cdot p)\Lambda_0)&\lesssim\phi(t_i)\cdot e^{-\frac\epsilon2 t_i}
\end{align*}
for some divergent sequence $\{t_i\}_{i\in\mathbb N}$. Note that in the construction $$a_{t_i}p\in B_{\mathbf F_e}(\tilde{\delta}_0)\cdot a_{t_i}q_i$$ and there exist a primitive integral vector $v_i\in\Lambda_0\setminus\{0\}$ and $x_i\in B_{\mathbf F_e}(\tilde\delta_0)$ such that  $$\rho_k(a_{t_i}q_i)v_i\in\rho(a_{t_i}q_i)\Lambda_0\cap V_\chi\;(V_\chi=V_{\beta_0}),\quad\rho_k(x_i)\rho_k(a_{t_i}q_i)v_i\in\rho_k(a_{t_i}p)\Lambda_0$$ and $$\delta(\rho_k(a_{t_i}\cdot p)\Lambda_0)\leq\|\rho_k(x_i)\rho_k(a_{t_i}q_i)v_i\|\asymp e^{\chi(a_{t_i})}\cdot l_i\leq\phi(t_i)\cdot e^{-\frac\epsilon2 t_i}.$$ One may verify that for small $\tilde\delta_0>0$ $$\|\pi_\chi(\rho_k(x_i)\rho_k(a_{t_i}q_i)v_i)\|
\geq\frac12\|\rho_k(x_i)\rho_k(a_{t_i}q_i)v_i\|$$ and hence $$\|\pi^+(\rho_k(a_{t_i}p)v_i)\|
\geq\frac12\|\rho_k(a_{t_i}p)v_i\|.$$ Moreover, by Lemma~\ref{l81}, $v_i$ is a pure tensor in $\bigwedge^k\mathbb Z^n$. Denote by $y_i\in X_k(\mathbb Q)$ the $k$-dimensional $\mathbb Q$-subspaces associated to $v_i$.

Now we may choose $$\phi(t)=e^{-t\cdot \frac kl}\cdot\Psi_{\epsilon}^{-1}(e^{-t(\frac1l+\frac1{n-l})})$$ where $\Psi_{\epsilon}(u)=u^{-n/(k(n-l))}\psi(u)u^{-\epsilon}$. Then by Proposition~\ref{p83}(2), for $x=P\cdot g\in X_l(\mathbb R)$ with $g\in\mathbf A_\infty$, there exist infinitely many $y_i\in X_k(\mathbb Q)$ and sufficiently large $t_i\in\mathbb R_+$ such that $$H(y_i)\lesssim e^{\frac kl t_i}\|\rho_k(a_{t_i}g)v_i\|\lesssim e^{\frac kl t_i}\cdot\phi(t_i)e^{-\frac\epsilon2 t_i}\implies H(y_i)\leq\Psi_{\epsilon}^{-1}(e^{-t_i(\frac1l+\frac1{n-l})})$$ and $$d(y_i,x)\lesssim e^{-t_i(\frac1l+\frac1{n-l})}\leq\Psi_\epsilon(H(y_i))=H(y_i)^{-n/(k(n-l))}\psi(H(y_i))H(y_i)^{-\epsilon}$$ which implies that $x\in E_{l,k}(\psi)^c$. So we have $$\dim_H E_{l,k}(\psi)^c+\dim P\geq\dim_H\mathbf A_\infty+\dim\mathbf P_0$$ and $$\dim_HE_{l,k}(\psi)^c\geq\dim X_l-\sum_{\alpha\in\Phi(\mathbf F_e)}\frac{\alpha(a_1)(\tau+\epsilon)}{\chi(a_{-1})\nu_0(a_1)}$$ where $\tau=\tau(\phi)=\tau(e^{-t\cdot \frac kl}\cdot\Psi_{\epsilon}^{-1}(e^{-t(\frac1l+\frac1{n-l})}))$. One may compute that $$\tau=\frac kl-\frac{1/l+1/(n-l)}{\gamma(\psi)+n/(k(n-l))+\epsilon}$$ and $\nu_0(a_1)=1/l+1/(n-l)$ and when $\gamma(\psi)<\infty$, $\tau=\tau(\phi)<\chi(a_{-1})=k/l$. Let $\epsilon\to0$ and we obtain that $$\dim_HE_{l,k}(\psi)^c\geq(l-k)(n-l)+\frac n{n/(k(n-l))+\gamma(\psi)}.$$ 

Note that if $\gamma(\psi)=\infty$, then $\tau(\phi)=\chi(a_{-1})$. In this case, the following function $$\phi(t)e^{\chi(a_{-1})t}=\Psi_\epsilon^{-1}(e^{-t(\frac1l+\frac1{n-l})})$$ is unbounded as $\psi$ is decreasing. So we can choose the Cantor-type subset $\mathbf A_\infty$ constructed in Proposition~\ref{p56} and apply the argument above again to conclude that $$\dim_HE_{l,k}(\psi)^c\geq(l-k)(n-l)+\frac n{n/(k(n-l))+\gamma(\psi)}=(l-k)(n-l).$$ This completes the proof of Theorem~\ref{thm15}.
\end{proof}

\vspace{0.2in}
\noindent\textbf{Acknowledgements.} The author thanks Nicolas de Saxce, Reynold Fregoli and Barak Weiss for valuable discussions.


\end{document}